\documentclass[12pt]{amsart} 
               
\usepackage{amsmath,amssymb}                                          
\usepackage{graphicx}
\usepackage{amsfonts,amscd,latexsym,bbm,epsfig,epic,eepic,oldgerm,psfrag} 

\usepackage{a4wide} 

\numberwithin{figure}{section}
\numberwithin{equation}{section}

\newcommand{\proofend}{\hspace*{\fill} $\Box$\\}

\def\Cr{\operatorname{Cr}}
\def\PD{\operatorname{PD}}

\def\cc{{\mathcal C}}

\def\ce{{\mathcal E}}

\def\CC{\mathbbm{C}}

\def\NN{\mathbbm{N}}

\def\RR{\mathbbm{R}}

\def\ZZ{\mathbbm{Z}}

\def\aa{\boldsymbol{a}}
\def\mm{\boldsymbol{m}}

\def\uu{\boldsymbol{u}}
\def\vv{\boldsymbol{v}}
\def\ww{\boldsymbol{w}}
\def\xx{\boldsymbol{x}}
\def\s{\smallskip}
\def\ni{\noindent}
\def\b{\bigskip}
\def\m{\medskip}
\def\1{\:\!}

\def\eps{\epsilon}
\def\gve{\varepsilon}
\def\gf{\varphi}

\def\gl{\lambda}
\def\la{\lambda}
\def\go{\omega}

\newcommand{\e}{{\rm e}}

\renewcommand{\eps}{{\varepsilon}}

\newcommand{\pt}{{\rm pt}}
\def\CP{\operatorname{\mathbbm{C}P}}

\newcommand{\se} {\;{\stackrel{s}\hookrightarrow}\;}
\newcommand{\bn}{{\mathbf{b}}}

\newtheorem{theorem}{Theorem}[section]
\newtheorem{thm}[theorem]{Theorem}
\newtheorem{corollary}[theorem]{Corollary}

\newtheorem{lemma}[theorem]{Lemma}

\newtheorem{proposition}[theorem]{Proposition}

\newtheorem*{claim*}{Claim}
\newtheorem{conjecture}[theorem]{Conjecture}

\newtheorem{remark}[theorem]{Remark}
\newtheorem{remarks}[theorem]{Remarks}

\newtheorem{notation}[theorem]{Notation}
\begin{document}

\title[Symplectic embeddings $E(1,a) \to P(\la, \la b)$]{Symplectic embeddings of four-dimensional ellipsoids into integral polydiscs}

\author{Daniel Cristofaro-Gardiner} \thanks{DCG partially supported by NSF grant DMS-1402200.}
\address{(D.~Cristofaro-Gardiner)
Mathematics Department, Harvard University, Cambridge MA, USA} 
\email{gardiner@math.harvard.edu}
\author{David Frenkel} 
\address{(D.~Frenkel)
Institut de Math\'ematiques,
Universit\'e de Neuch\^atel, 
Rue \'Emile Argand~11, 
CP~158,
2000 Neuch\^atel,
Switzerland} 
\email{david.frenkel@unine.ch}
\author{Felix Schlenk} \thanks{FS partially supported by SNF grant 200021-163419.}
\address{(F.~Schlenk) 
Institut de Math\'ematiques,
Universit\'e de Neuch\^atel, 
Rue \'Emile Argand~11, 
CP~158, 
2000 Neuch\^atel,
Switzerland} 
\email{schlenk@unine.ch}
\date{\today}

\begin{abstract} 
In previous work, the second author and M\"uller determined the function~$c(a)$ 
giving the smallest dilate of the polydisc $P(1,1)$ into which the ellipsoid $E(1,a)$ symplectically embeds.  
We determine the function of two variables~$c_b(a)$
giving the smallest dilate of the polydisc~$P(1,b)$ into which the ellipsoid~$E(1,a)$ 
symplectically embeds for all integers $b \geqslant 2$.  

It is known that for fixed $b$, if $a$ is sufficiently large then all obstructions to 
the embedding problem vanish except for the volume obstruction.  
We find that there is another kind of change of structure that appears as one instead increases~$b$:  
the number-theoretic ``infinite Pell stairs" from the $b=1$~case almost completely disappears 
(only two steps remain), 
but in an appropriately rescaled limit, the function $c_b(a)$ converges as $b$ tends to infinity 
to a completely regular infinite staircase with steps all of the same height and width.
\end{abstract}

\maketitle

\tableofcontents

\section{Introduction and result}

\subsection{Introduction}
Since Gromov's classic paper ~\cite{Gr85}, it has been known that 
symplectic embedding problems are intimately related to many phenomena in symplectic geometry, Hamiltonian dynamics, and other fields.  The smallest interesting dimension is four, and all our results are in this 
dimension. So consider the standard four-dimensional symplectic vector space~$(\RR^4,\omega)$,
where $\omega = dx_1 \wedge dy_1 + dx_2 \wedge dy_2$.
Open subsets in~$\RR^4$ are endowed with the same symplectic form.
Given two such sets~$U$ and~$V$, a symplectic embedding of~$U$ into~$V$
is a smooth embedding $\varphi \colon U \to V$ that preserves the symplectic form:
$\varphi^* \omega = \omega$. 
We write $U \se V$ if there exists a symplectic embedding $U \to V$.
Deciding whether $U \se V$ is very hard in general. 
One thus looks at simple sets, 
such as the open ball $B^4(a)$ of radius $\sqrt a$, 
or polydiscs $P(a,b) = B^2(a) \times B^2(b) \subset \RR^2(x_1,y_1) \times \RR^2(x_2,y_2)$, 
or ellipsoids
$$
E(a,b) \,:=\, \left\{\frac{x_1^2+y_1^2}{a}+\frac{x_2^2+y_2^2}{b} < 1 \right\} .
$$ 
In four dimensions, Gromov's Nonsqueezing Theorem states that  
$B^4(a) \se B^2(b) \times \RR^2(x_2,y_2)$ only if $a \leqslant b$.
In other words, one cannot do better than the identity mapping.
After this rough rigidity result, the ``fine structure of symplectic rigidity''
was investigated by looking at other embedding problems.
The first important results were on the ``packing problem'', 
where $U$ is a disjoint union of balls, see~\cite{Gr85,McPo94,Bi97,Bi99}.
Further understanding on the fine structure came with the study of embeddings of 
ellipsoids~\cite{Sch03a,Sch:book,McD09,McSch12,FrMu12,Hu11,McD11,CCFHR}.
Note that $E(a,b) \se V$ if and only if $E(1,\frac ba) \se \frac{1}{\sqrt{a}} V$. 
We can thus take $E(1,a)$ with $a \geqslant 1$ as~$U$.
Encode the embedding problems $E(1,a) \se B^4(b)$ and $E(1,a) \se P(b,b) =: C^4(b)$ in the functions
\begin{eqnarray*}
c_B(a) &:=& \inf \bigl\{ \la >0 \mid E(1,a) \se B^4(\la) \bigr\}, \\
c_C(a) &:=& \inf \bigl\{ \la >0 \mid E(1,a) \se C^4(\la) \bigr\}.
\end{eqnarray*}
Since symplectic embeddings are volume preserving, 
$c_B(a) \geqslant \sqrt{a}$ and $c_C(a) \geqslant \sqrt{\frac{a}{2}}$.
The functions $c_B(a)$ and~$c_C(a)$ were computed in~\cite{McSch12} and \cite{FrMu12}:

The function $c_B(a)$ has three parts: 
On $[1, \tau^4]$, with $\tau = \frac{1+\sqrt 5}{2}$ the golden ratio, 
$c_B$ is given by the ``Fibonacci stairs'', 
namely an infinite stairs each of whose steps 
is made of a segment on a line going through the origin and a horizontal segment, 
with foot-points on the volume constraint $\sqrt a$, and both the foot-points and the edge
determined by Fibonacci numbers.
Then there is one step over $[\tau^4, 7 \frac 19]$, whose left part over $[\tau^4,7]$ is affine but non-linear: $c_B(a) = \frac{a+1}{3}$.
Finally, for $a \geqslant 7\frac 19$ the graph of $c_B(a)$ is given by eight
strictly disjoint steps made of two affine segments, 
and $c_B(a) = \sqrt{a}$ for $a \geqslant 8 \frac{1}{36}$.

The function $c_C(a)$ has a similar structure:
On $[1, \sigma^2]$, with $\sigma = 1+\sqrt 2$ the silver ratio, 
$c_C$ is given by the ``Pell stairs'', 
namely an infinite stairs each of whose steps 
is made of a segment on a line going through the origin and a horizontal segment, 
with foot-points on the volume constraint $\sqrt{\frac a2}$, and both the foot-points and the edge
determined by Pell numbers.
Then there is one step over $[\sigma^2, 6 \frac 18]$, 
whose left part over $[\sigma^2,6]$ is affine but non-linear: $c_C(a) = \frac{a+1}{4}$.
Finally, for $a \geqslant  6 \frac 18$ the graph of $c_C(a)$ is given by six
strictly disjoint steps made of two affine segments, 
and $c_C(a) = \sqrt{\frac a2}$ for $a \geqslant  7 \frac{1}{32}$.

\subsection{Result}
We are interested in understanding what happens with the rich structure of the functions $c_B$ and~$c_C$
if we take as targets ``longer'' sets. 
To this end, we look at the embedding problems $E(1,a) \se P(b,c)$ for $c = kb$ with $k \geqslant 2$ an integer, 
that we encode in the functions
\begin{eqnarray} \label{e:cb}
c_b(a) &:=& \inf \bigl\{ \la >0 \mid E(1,a) \se P(\la,\la b) \bigr\} , \quad
b \in \NN_{\geqslant 2} .
\end{eqnarray}
%
Note that $c_1 = c_C$.
The volume constraint is now $c_b(a) \geqslant \sqrt{\frac{a}{2b}}$.
To formulate our result, we define for $b \in \NN_{\geqslant 2}$ and for 
$k \in \left\{ 0,1,2, \dots, \lfloor \sqrt{2b} \rfloor \right\}$ the numbers
$$
u_b(k) \,:=\, \frac{(2b+k)^2}{2b} \,=\, 2b+2k+\frac{k^2}{2b}, 
\qquad
v_b(k) \,:=\, 2b \left( \frac{2b+2k+1}{2b+k} \right)^2
$$
and
$$
\alpha_b \,:=\, \frac{1}{b} \left( b^{2}+2b+\sqrt{\left( b^{2}+2b\right)^{2}-1} \right),
\qquad
\beta_{b} \,:=\, 2b+4 + \frac{1}{2b(b+1)^2} .
$$
Note that $u_b(k) \leqslant 2b+2k+1 \leqslant v_b(k)$ with strict inequalities for $k^2 < 2b$ and equalities
for $k^2 = 2b$,
and that 
$$
2b+2k \,<\, u_b(k) \,\leqslant\, v_b(k) \,<\, 2b+2k+2 \quad \mbox{ for }\, k \geqslant 1.
$$
Further, $v_b(1) < \alpha_b < 2b+4 < \beta_b < u_b(2)$.
The intervals $I_b(k) := [u_b(k), v_b(k)]$ thus have positive length except for $k^2=2b$,
and the intervals
$$
I_b(0), \; I_b(1), \; [\alpha_b,\beta_b], \; I_b(2), \; \dots, \; 
I_b({\lfloor \sqrt{2b} \rfloor})
$$
are in the right order and are disjoint except that $I_b(0)$ touches $I_b(1)$.

\begin{theorem} \label{t:main}
For every integer $b \geqslant 2$ the function $c_b (a)$ describing the symplectic embedding problem
$E(1,a) \se P(\la,\la b)$ is given by the volume constraint $c_b(a) = \sqrt{\frac{a}{2b}}$ 
except for the following $\big\lceil \sqrt{2b} \, \big\rceil +2$ intervals: 

\m
\begin{itemize}
\item[(i)]
$c_b(a) = 1$ if $a \in [1,2b]$.
 
\s
\item[(ii)]
For $k \in \bigl\{ 0,1,2, \dots,  \lfloor \sqrt{2b} \rfloor  \bigr\}$ and on the interval $I_b(k)$,
$$
c_b(a) \,=\,
\left\{\begin{array} {cl}        
\frac{a}{2b+k}    &           \mbox{if }\;  a \in [u_b(k),2b+2k+1],  \\ [0.2em]
\frac{2b+2k+1}{2b+k} &        \mbox{if }\;  a \in [2b+2k+1, v_b(k)].
\end{array}\right.
$$

\s
\item[(iii)]
On the interval $[\alpha_b, \beta_b]$,
$$
c_b(a) \,=\, 
\left\{
\begin{array} {cl}
\frac{ba+1}{2b(b+1)}   & \mbox{if }\;  a \in [\alpha_b,2b+4], \\ [0.2em]
1+\frac{2b+1}{2b(b+1)} & \mbox{if }\;  a \in [2b+4, \beta_b].
\end{array}\right.
$$
\end{itemize}
\end{theorem}

\begin{remarks} \label{rem:EP}
{\rm 
{\bf 1.}
Theorem~\ref{t:main} also solves the problem $E(1,a) \se E(\la, \la 2b)$ for integers~$b \geqslant 2$,
since for every integer~$b$,
\begin{equation} \label{e:EP}
E(1,a) \se P (\la, \la b) \quad \Longleftrightarrow \quad
E(1,a) \se E (\la, \la 2b) .
\end{equation}
This has been shown in~\cite[Cor.~1.6]{FrMu12} for~$b=1$
by using that ECH-capacities provide a complete set of invariants for the embedding problem
$E(1,a) \se P(b,c)$, and this proof generalizes to all~$b \in \NN$.
In \S~\ref{ss:proof.eq} we shall prove~\eqref{e:EP}  by using the ``reduction method'' 
(Method~2 of \S~\ref{ss:trans}).

\s
{\bf 2.}
One can replace the infimum in definition~\eqref{e:cb} by the minimum. 
This follows from the previous remark and from the fact that
$E(1,a) \se E (\la, \la 2b)$ also for $\la = c_b(a)$,
see~\cite[Cor.~1.6]{McD09} and also~\cite[Cor.~1.6]{CG14} for a generalization.
Altogether, we see that
\begin{equation*} 
E(1,a) \se P (\la, \la b) \quad \Longleftrightarrow \quad 
E(1,a) \se E (\la, \la 2b) \quad \Longleftrightarrow \quad
\la \geqslant c_b(a) .
\end{equation*}
}
\end{remarks}

\ni
{\bf Geometric description of the result.}
We proceed with describing the functions~$c_b(a)$ given in Theorem~\ref{t:main}
more geometrically.
The left part of the steps described in part~(ii) of the theorem
lie on a line passing through the origin, 
while the left part of the step described in part~(iii) lies on a line crossing
the $y$-axis at $\frac{1}{2b(b+1)}$.
We call the steps in~(ii) the ``linear steps'', and the step in~(iii) the ``affine step''.

\begin{figure}[ht]
 \begin{center}
 \leavevmode\epsfbox{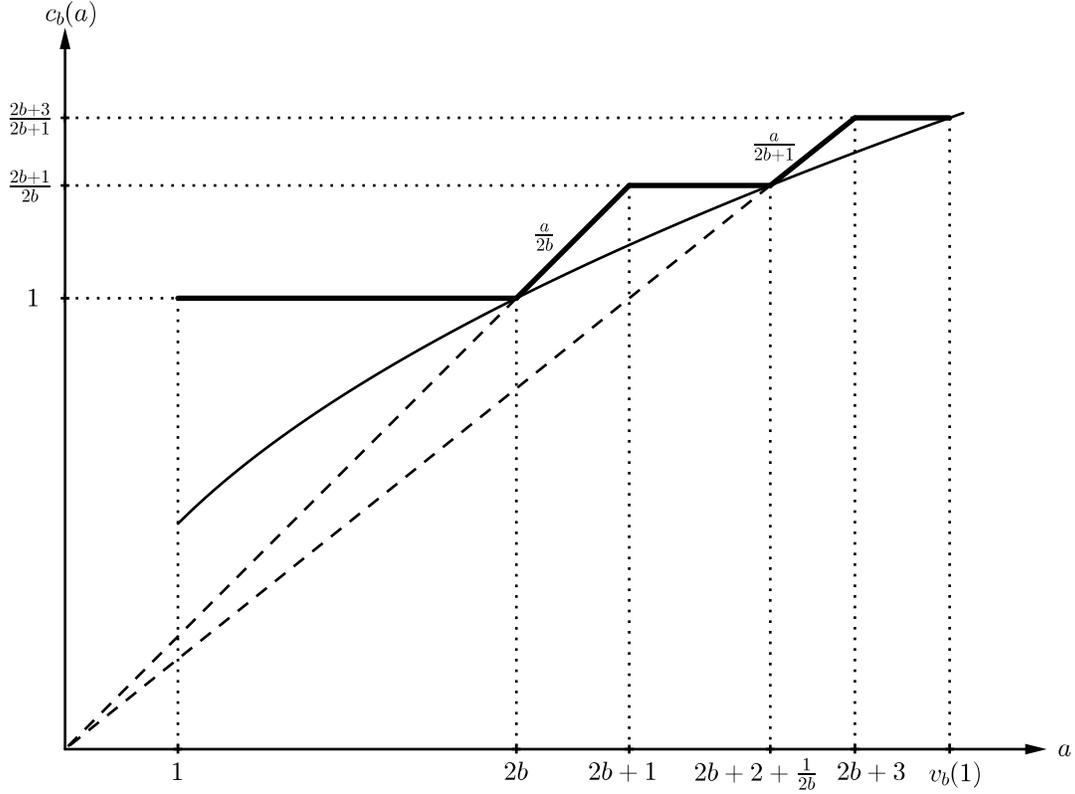}
 \end{center}
 \caption{The graph of $c_b(a)$ on $[1,v_b(1)]$}
 \label{fig:left}
\end{figure}
%
%

The graph of~$c_b(a)$ on $[1,v_b(1)]$ is given by
$$
c_b(a) \,=\,
\left\{\begin{array} {cl}  
1    &           \mbox{if }\;  a \in [1,2b],\\      
\frac{a}{2b}    &      \mbox{if }\;  a \in [2b,2b+1],\\
\frac{2b+1}{2b} &      \mbox{if }\;  a \in \bigl[2b+1,2b+2+\frac{1}{2b}\bigr],\\
\frac{a}{2b+1}  &      \mbox{if }\;  a \in \bigl[2b+2+\frac{1}{2b},2b+3\bigr],\\
\frac{2b+3}{2b+1} &    \mbox{if }\;  a \in \bigl[2b+3,2b+4-\frac{4}{(2b+1)^2}\bigr],
\end{array}\right.
$$
see Figure~\ref{fig:left}.
\begin{figure}[ht]
 \begin{center}
 \leavevmode\epsfbox{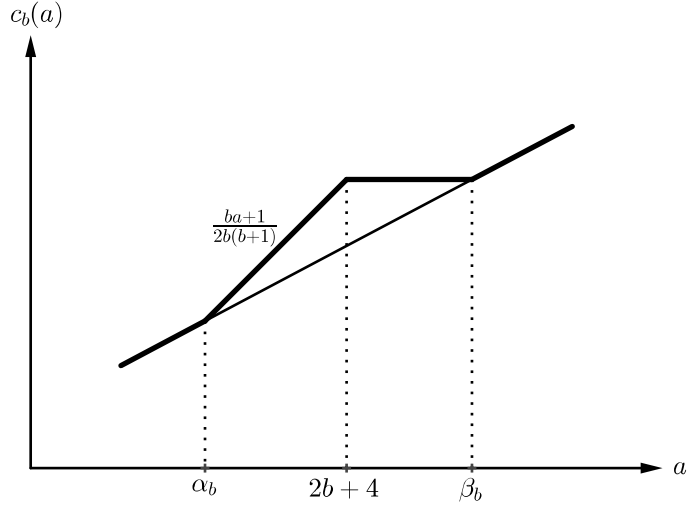}
 \end{center}
 \caption{The affine step}
 \label{fig:affine}
\end{figure}
%
%

This part of the graph touches the volume constraint only in three points.
Then follows a ``volume interval'', and then the affine step described in 
part~(iii) and Figure~\ref{fig:affine}.
For $b = 2$ there are no further obstructions (Figure~\ref{fig:b=2}), 
but for $b \geqslant 3$ there are $\big\lceil \sqrt{2b} \, \big\rceil -2$ more linear steps, that are 
strictly disjoint and made of a linear and a horizontal segment 
(Figures~\ref{fig:linear} and~\ref{fig:c9}).

\begin{figure}[ht]
 \begin{center}
 \leavevmode\epsfbox{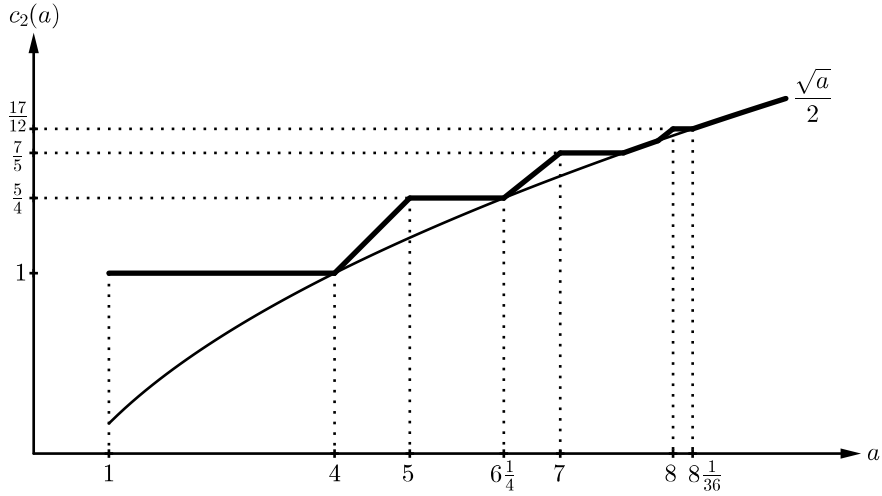}
 \end{center}
 \caption{The graph of $c_2(a)$}
 \label{fig:b=2}
\end{figure}
%
%

%
%
\begin{figure}[ht]
 \begin{center}
 \leavevmode\epsfbox{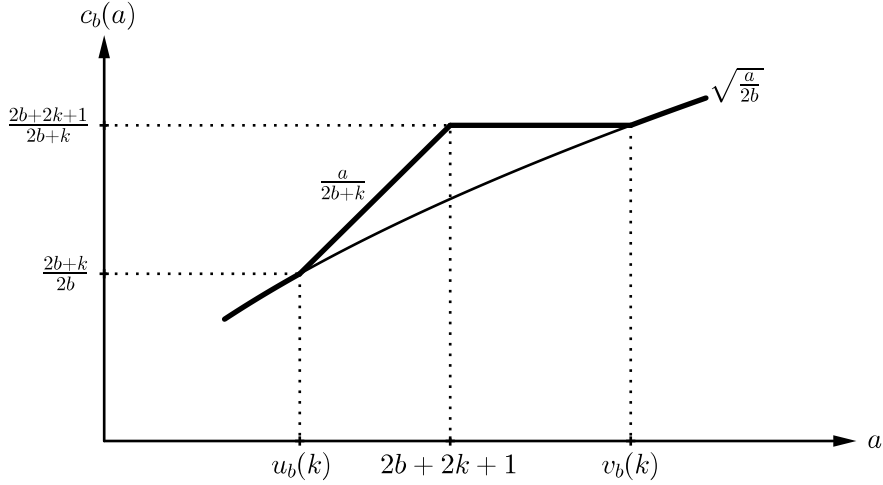}
 \end{center}
 \caption{One of the $\big\lceil \sqrt{2b} \big\rceil$ linear steps}
 \label{fig:linear}
\end{figure}
%
%

%
%
\begin{figure}[ht]
 \begin{center}
 \includegraphics[width=10cm]{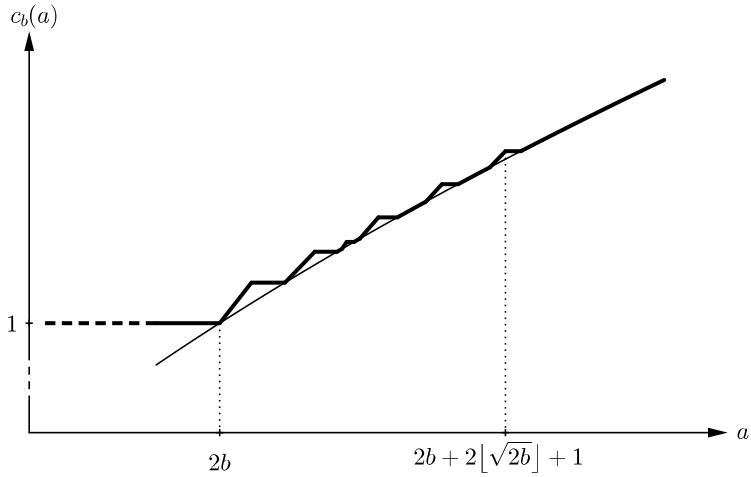}
 \end{center}
 \caption{The graph of $c_9(a)$}
 \label{fig:c9}
\end{figure}
%
%

The length of the affine step is 
$\beta_b-\alpha_b < \beta_b-v_b(1) = \frac{1}{2b(b+1)^2} + \frac{4}{(2b+1)^2}$, and hence this step 
becomes very small for $b$ large.
The length of the $k$'th linear step is 
$$
\ell_b(k) \,:=\, v_b(k)-u_b(k) \,=\,
(2b-k^2) \, \frac{8b^2+k^2+(2+8k)b}{2b(2b+k)^2}.
$$
For fixed $b$, the function~$\ell_b(k)$ is strictly decreasing, with $\ell_b (\sqrt{2b}) =0$.
For fixed~$k$, however, $\lim_{b \to \infty} \ell_b(k) = 2$.
More precisely, $\ell_b(0)$ is strictly decreasing to~$2$,
and $\ell_b(k)$ is strictly increasing to~$2$ for every $k \geqslant 1$.
Since the edge of the $k$'th step is at $2b+2k+1$, we see that for $b \to \infty$,
an arbitrarily large (but fixed) part of the graph of~$c_b(a)$ consists of linear
steps of length almost~$2$, that almost form a connected staircase 
(Figure~\ref{fig:c85}).

\begin{figure}[ht]
 \begin{center}
 \includegraphics[width=13cm]{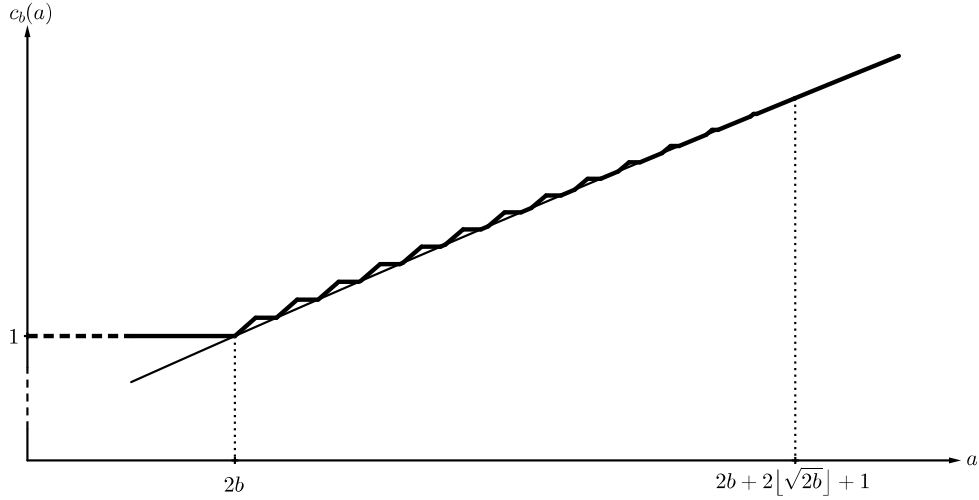}
 \end{center}
 \caption{The graph of $c_{85}(a)$}
 \label{fig:c85}
\end{figure}
%
%

We reformulate this behaviour of $c_b(a)$ for large~$b$ in terms of a rescaled limit function:
Consider the rescaled functions
$$
\hat c_b(a) \,=\, 2b\, c_b(a+2b)-2b, \qquad a \geqslant 0,
$$
that are obtained from $c_b(a)$ by first forgetting about the horizontal line 
$c_b(a)=1$ over~$[1,2b]$ that comes from the Nonsqueezing Theorem,
then vertically rescaling by~$2b$, and finally translating the graph by the vector~$(-2b,-2b)$.
Further, consider the function $c_\infty \colon [0,\infty) \to \RR$ drawn in 
Figure~\ref{fig:cinfty}; 
its graph consists of infinitely many steps of width~$2$ and slope~$1$ 
that are based at the line~$\frac a2$.
Then 
\begin{equation}  \label{e:climit}
\lim_{b \to \infty} \hat c_b(a) \,=\, c_{\infty}(a) , \qquad   a \in [0,\infty) ,
\end{equation}
uniformly on bounded sets. 
Indeed, applying the same rescaling to $\sqrt{\frac{a}{2b}}$ yields $2b \sqrt{\frac{a+2b}{2b}} -2b$, 
which is $\frac a2 +O(\frac{a^2}{2b})$ for $b \geqslant a$.
One can also check that $\hat c_b(a)$ is increasing to $c_\infty (a)$ for all~$a$. 

\begin{figure}[ht]
 \begin{center}
 \includegraphics[width=10cm]{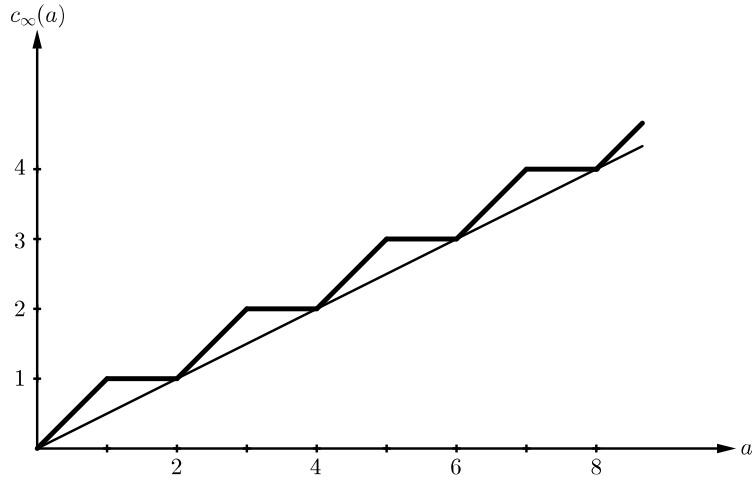}
 \end{center}
 \caption{The graph of the rescaled limit function $c_{\infty}(a)$}
 \label{fig:cinfty}
\end{figure}
%
%

\subsection{Interpretation}
Recall from the introduction that the graph of~$c_C(a)$ has three parts:
Fist the infinite Pell stairs, then one affine step, and then six more steps.

If we take $b=1$ in the above description of $c_b(a)$ on $[1,v_b(1)]$, we exactly obtain
$c_C(a)$ on $[1,v_1(1)] = [1,\frac{50}{9}]$. 
Further, if we take $b=1$ in the description~(iii) of the affine step of~$c_b(a)$, 
we exactly obtain the affine step of $c_C(a)$ over $[\sigma^2, 6 \frac 18]$.
Hence $c_b(a)$ generalizes $c_C(a)$ on the first two steps and on the affine step.
This is not a coincidence.
Indeed, the two exceptional classes giving rise to the first two steps of the Pell stairs
are the first two in the sequence~\eqref{e:EF0} of exceptional classes~$E_n$ 
giving rise to all the linear steps
of~$c_b(a)$,
and the exceptional class giving rise to the affine step of~$c_C(a)$ is the first in 
a sequence of exceptional classes~$F_b$ giving rise to the affine step in~$c_b(a)$;
see~\S~\ref{s:method1}.

On the other hand, the remaining infinitely 
many steps of the Pell stairs have no counterpart for $b \geqslant 2$.
Similarly, the linear steps described in (ii) of Theorem~\ref{t:main} 
are more regular than the affine steps on the right part of~$c_C(a)$, 
none of which consists of a linear and a horizontal segment.
We thus see that the first two steps and the affine step of~$c_C(a) = c_1(a)$ are stable under
the deformations of~$b$ we consider, while the other steps are not.

\s
By Theorem~\ref{t:main}, $c_b(a)$ equals the volume constraint $\sqrt{\frac{a}{2b}}$
for $a \geqslant v_b (\lfloor \sqrt{2b} \rfloor) = 2b + O(\sqrt b)$,
that is, there are no packing obstructions for the embedding problem
$E(1,a) \se P(\la,\la b)$ for $a$ sufficiently large.
This is not a surprise. Indeed, 
this phenomenon was already observed for the embedding problems $E(1,a) \se B^4(b)$ and $E(1,a) \se C^4(b)$,
and it fits well with previous results:
It is known for many closed connected symplectic manifolds~$(M,\omega)$ that there is
a number $N(M,\omega)$ such that $(M,\omega)$ admits a full symplectic packing by
$k$ equal balls for every $k \geqslant N(M,\omega)$ 
(``packing stability'', see~\cite{Bi97, Bi99, BuHi11, BuHi13, BuHiOp, BuPi13}).
Similarly, an explicit construction implies that 
for any connected symplectic manifold $(M,\go)$ of finite volume, 
the proportion of the volume that can be filled by a dilate of the
ellipsoid $E(1, \dots, 1,a)$ tends to~$1$ as $a \to \infty$, see~\cite[\S~6]{Sch:book}: 
The packing obstruction tends to zero as the {\it domain}\/ is more and more elongated.

Theorem~1 exhibits a different phenomenon:
If in the problem $E(1,a) \se P(\la, \la b)$ the {\it target}\/ is elongated ($b \to \infty$), 
then the regular Pell stairs in the graph of $c_1(a)$ first almost disappears
(only two linear steps and the affine step remain), but then for large~$b$ the graph
of~$c_b(a)$ reorganizes to a staircase that asymptotically is infinite and completely regular.

\subsection{Stabilization and connection with symplectic folding}

Let $a,b \geqslant 1$ be real numbers.
Following~\cite{CrHi15} we consider for each $N \geqslant 3$ the stabilized problem
\begin{equation*} 
c_b^N(a) \,:=\,
\inf \bigl\{ \la >0 \mid E(1,a) \times \CC^{N-2} \,\se\, P(\la,\la b) \times \CC^{N-2} \bigr\} .
\end{equation*}
Then $c_b^N(a) \leqslant c_b(a)$.

\begin{lemma} \label{le:Nfolding}
For every $N \geqslant 3$ and all real numbers $a,b \geqslant 1$,
$$
c_b^N(a) \,\leqslant\, f_b(a) \,:=\, \frac{2a}{a+2b-1} .
$$
\end{lemma}

\proof
Set $\mu = \frac{a(2b-1)}{a+2b-1}$ and $\la = 2 (1-\frac{\mu}{a})$.
Then $\mu+\frac{\la}{2} = b \la$.
Since $b \geqslant 1$ we have $\mu \geqslant \frac{\la}{2}$.
Note that $\frac \la 2 = 1-\frac{\mu}{a}$ is the area of a $z_2$-disc in $E(a,1)$ over a point $z_1$
on the boundary of the disc $D(\mu)$ of area~$\mu$.
Applying Hind's folding construction in~\cite[\S~2]{Hi15} with $\mu$ (instead of~$\frac{S}{S+1}$)
we obtain for every $\gve >0$ a symplectic embedding
$$
E(1,a) \times \CC \,\se\, P(\mu + \tfrac{\gl}{2}+\gve, 2 \,\tfrac{\gl}{2}+\gve) \times \CC .
$$
Now recall that $\mu+\frac{\la}2 = b\la$ and note that $\la = f_b(a)$.
\proofend

In view of the above proof, we call the graph of~$f_b(a)$ the folding curve.
Now note that
$$
f_b(2b+2k+1) \,=\, \frac{2b+2k+1}{2b+k} , \qquad k \geqslant 0.
$$
For $b \in \NN$
this is also the value of $c_b$ at the edge points of the $k$th linear step.
In other words, the linear steps oscillate between the volume constraint~$\sqrt{\frac{a}{2b}}$
and the folding curve,
see Figures~\ref{fig:linear} and~\ref{fig:stab}.
\begin{conjecture}
\label{conj:stable}
The edge points of the linear steps are stable,
in the sense that at these points we have $c_b^N = c_b$ for all $N \geqslant 3$.
\end{conjecture}

This conjecture is based on the main result of~\cite{CrHi15}, where it is shown
that the edge points of the Fibonacci stairs for the problem $E(a,1) \se B^4(\la)$ are stable.  
It is likely that one can prove it by a similar method as in~\cite{CrHi15}, 
see also the discussion at the end of the next section.  A proof of Conjecture~\ref{conj:stable} is not the concern of the present work, but a positive answer would imply that the folding construction in the proof of Lemma~\ref{le:Nfolding}
is sharp at the edge points of the linear steps.

\begin{figure}[ht]
 \begin{center}
 \includegraphics[width=12cm]{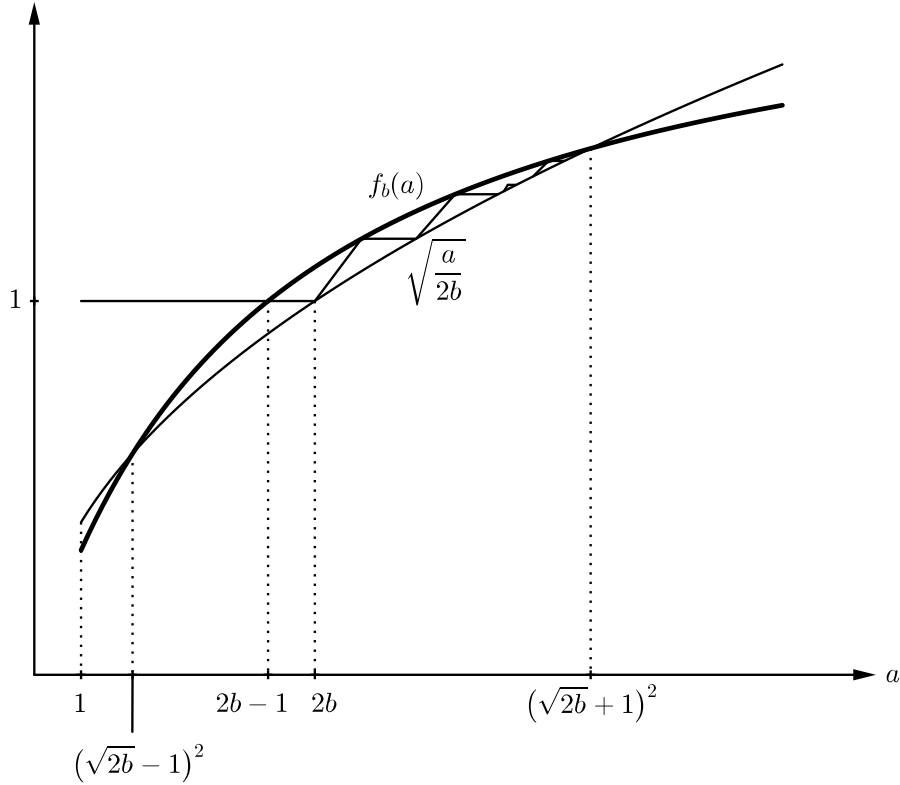}
 \end{center}
 \caption{The volume constraint, $c_b(a)$, and the folding curve, for $b=5$}
 \label{fig:stab}
\end{figure}
%
%

Recall that $c_b(a) = 1$ for $a \in [1,2b]$. 
As we shall see in Proposition~\ref{p:alarge}~(ii),
$c_b(a) = \sqrt{\frac{a}{2b}}$ for all $a \geqslant  (\sqrt{2b}+1)^2$
and all real $b \geqslant 2$.
Now notice that $f_b(a) \geqslant \sqrt{\frac{a}{2b}}$ if and only if 
$a \in \bigl[ (\sqrt{2b}-1)^2, (\sqrt{2b}+1)^2 \bigr]$.
It follows that
$$
c_b^N(a) \,<\, c_b(a) \quad \mbox{ if }\, 
a \notin \bigl[ 2b-1, (\sqrt{2b}+1)^2 \bigr]
$$
for all $b \geqslant 2$ and $N \geqslant 3$.

We finally notice that under the rescaling yielding the limit function $c_\infty(a)$,
we have $\hat f_b(a) = 2b\, f_b(a+2b)-2b = \frac{2b(a+1)}{a+4b-1}$,
and so
$$
f_\infty (a) \,:=\, \lim_{b \to \infty} \hat f_b(a) \,=\, \frac{a+1}{2} .
$$
This means that also the limit function $c_\infty$ oscillates, 
between the limit function $\frac{a}{2}$ of the volume constraint $\sqrt{\frac{a}{2b}}$
and the limit function $\frac{a+1}{2}$ of the folding curve.

\subsection{Method}
In principle, there are two 
methods to prove Theorem~\ref{t:main}:
The first method (Method~1 in~\S~\ref{ss:trans}, that was used in~\cite{McSch12,FrMu12}) 
is to find the strongest obstruction for the embedding problem $E(1,a) \se P(\la, \la b)$ 
coming from exceptional classes 
(i.e., homology classes in a certain multiple blow-up of~$\CP^2$
represented by embedded $J$-holomorphic $-1$ spheres). 
The second method (Method~2 in \S~\ref{ss:trans}, that was first used in~\cite{BuPi13}) 
is a cohomological version of the first method: One associates to a hypothetical embedding 
$E(1,a) \se P(\la, \la b)$ a cohomology class, and checks whether this class transforms to a 
``reduced vector'' under Cremona transforms. 
While the first method is sufficient for solving the problems $E(1,a) \se B^4(\la)$ and 
$E(1,a) \se C^4(\la)$, see~\cite{McSch12,FrMu12}, it does not lead to a proof of 
the entire Theorem~\ref{t:main}, because the known upper bound for 
the number of obstructive exceptional classes 
tends to infinity with~$b$. 
On the other hand, Method~2 
does yield a proof of Theorem~\ref{t:main}, as will become clear from our proof. 
We shall not follow such a puristic approach, however, but an opportunistic one, 
that uses both methods:
Given $b$, we first write down a finite set of exceptional classes that yield embedding obstructions, 
namely $E_0 = (1,0;1)$ and
\begin{eqnarray} \label{e:EF0}
E_n &:=& \left( n,1; 1^{\times (2n+1)}\right) , \quad n=b, \dots, b+\lfloor \sqrt{2b} \rfloor, \\
F_b &:=& \left( b(b+1), b+1; b+1, b^{\times (2b+3)} \right) , \notag
\end{eqnarray}
(see \S~\ref{ss:trans} for the notation), 
and then use Method~2 to show that the obstruction $f_b(a)$ given by these classes is complete. 
In other words, we use Method~1 to show $c_b(a) \geqslant f_b(a)$ and Method~2 to show $c_b(a) \leqslant f_b(a)$
(with the exception that for $a$ large and for $b=2$ and $a \in [8 \frac{1}{36},9]$ 
we use Method~1 to show that $c_b(a)$ equals the volume constraint $\sqrt{\frac{a}{2b}}$).

This hybrid approach yields the shortest proof of Theorem~\ref{t:main} we know. 
Further, knowing a set of exceptional classes that provide all embedding obstructions
is interesting for at least two reasons:
First, the holomorphic spheres underlying these classes provide a geometric explanation
of the graphs of the functions~$c_b(a)$.
Second, one should be able to use these holomorphic spheres to prove Conjecture~\ref{conj:stable}; it is probably the case that one can find the needed obstructions by stretching these spheres and then ``stabilizing" as in \cite{CrHi15,HiKe14}.

\subsection{Outlook}
Our ultimate goal is to see the {\it continuous}\/ film of graphs~$c_b(a)$ for $b \geqslant 1$ real.
It would be particularly interesting to understand this film for $b \in [1,2]$,
or just for $b \in [1,1+\eps]$ for some $\eps >0$, namely to understand how the
Pell stairs disappear.
In~\cite{BPT15}, ECH-capacities are used to compute~$c_b(a)$ for $b=\frac{13}{2}$ and to 
get an idea of this film.
In accordance with Theorem~\ref{t:main}, Conjecture~6.3 in~\cite{BPT15} and further investigations 
we make the
\begin{conjecture} \label{con:b}
For any real $b \geqslant 2$ the function $c_b(a)$ is given by the maximum of the volume
constraint $\sqrt{\frac{a}{2b}}$ and the obstructions coming from the exceptional classes~$E_n$ and~$F_n$ in~\eqref{e:EF0}. 
\end{conjecture}
The obstructions given by the exceptional classes~$E_n$ and~$F_n$ are readily computed, see~\S~\ref{ss:breal}:
While the classes $E_n$ again give rise to a finite staircase with linear steps, 
the classes~$F_n$ give an obstruction only for $b \in (n-\frac{n}{(n+1)^2}, n+\frac{1}{n+2})$.
While our proof of Theorem~\ref{t:main} should extend to a proof of Conjecture~\ref{con:b}, 
the analysis is more involved, since fractional parts arise, that are harder
to estimate.

Our only definite result for $b$ real is that for every real $b \geqslant 2$
we have $c_b(a) = \sqrt{\frac{a}{2b}}$ for all $a \geqslant  (\sqrt{2b}+1)^2$, 
see Proposition~\ref{p:alarge}~(ii).

\m
\ni
{\bf Acknowledgment.}
We cordially thank Dusa McDuff, who already in~2010 suggested to us to use the
reduction method for analyzing the embedding problem $E(1,a) \se C^4(\la)$.


\section{Methods of proof}

In this section we describe the methods we will use in the proof
of Theorem~\ref{t:main}.
For more details we refer to the surveys~\cite{CG16,Hu11b,Sch15} and the given references.

\subsection{Translation to a ball packing problem}
Fix $b \geqslant 1$.
Since the function $c_b(a)$ is continuous in~$a$, it suffices to compute $c_b(a)$
for $a \geqslant 1$ rational.
The {\bf weight expansion} $\ww (a)$ of such an~$a$
is the finite decreasing sequence
\begin{eqnarray} \label{eq:xs}
\ww(a) &:=& \bigl(\underbrace{1,\dots,1}_{\ell_0}, \,
\underbrace{w_1,\dots,w_1}_{\ell_1}, \,
\dots, \, \underbrace{w_N,\dots,w_N}_{\ell_N} \bigr) \\ \notag
&\equiv& \bigl( 1^{\times \ell_0},\, w_1^{\times \ell_1}, \, \dots, \, w_N^{\times \ell_N} \bigr) 
\end{eqnarray}
such that $w_1 = a-\ell_0 < 1$, $w_2 = 1-\ell_1 w_1 < w_1$, and so on.
For example, $a=25/9$ has weight expansion 
$\ww(a) = (1,1,\frac 79,\frac 29,\frac 29,\frac 29,\frac 19,\frac19) \equiv
(1^{\times2}, \frac 79, \frac 29\,\!^{\times3},\frac 19\,\!^{\times2})$. 

Write $B(\ww (a))$ for the disjoint union of balls $B(1) \coprod \dots \coprod B(w_N)$ 
whose weights are those appearing in~$\ww(a)$, with multiplicities.
Based on~\cite{McD09}
it was shown in~\cite[Prop.~1.4]{FrMu12} that $E(1,a) \se P (\la, \la b)$
if and only if
\begin{equation} \label{e:p1}
B(\ww (a)) \coprod B(\lambda) \coprod B(\lambda b) 
\,\se\, B (\lambda (b+1)) ,
\end{equation}
cf.\ the moment map picture on the left of Figure~\ref{fig:moment}.

\subsection{Three translations to a combinatorial problem} \label{ss:trans}

In order to reformulate problem~\eqref{e:p1}, we look at the
general ball packing problem
\begin{equation} \label{e:p2}
\coprod_{i=1}^n B(a_i) \se B(\mu) .
\end{equation}
We shall describe three combinatorial solutions of~\eqref{e:p2}.

Denote by $X_n$ the $n$-fold complex blow-up of $\CP^2$,
endowed by the orientation induced by the complex structure.
Its homology group $H_2(X_n;\ZZ)$ has the canonical basis
$\{ L, E_1, \dots, E_n \}$, where $L=[\CP^1]$ and the $E_i$ are the classes of 
the exceptional divisors.
The Poincar\'e duals of these classes are denoted $\ell, e_1, \dots, e_n$.
Let $K := - 3L + \sum_{i=1}^n E_i$ be the Poincar\'e dual of $-c_1(X_n)$,
and consider the $K$-symplectic cone $\cc_K(X_n) \subset H^2(X_n;\RR)$,
namely the set of cohomology classes that can be represented by
symplectic forms~$\go$ on~$X_n$ that are compatible with the orientation of~$X_n$
and have first Chern class $c_1(\go) = c_1(X_n) = \PD(-K)$.
Denote by $\overline{\cc_K}(X_n)$ its closure in $H^2(X_n;\RR)$.

McDuff--Polterovich~\cite{McPo94} proved that an embedding~\eqref{e:p2} exists 
if and only if
$$
\mu \ell - \sum_{i=1}^n a_i e_i \,\in\, \overline{\cc_K} (X_n) .
$$

We thus need to describe $\overline{\cc_K} (X_n)$. For this consider the set 
$\ce_K(X_n) \subset H_2(X_n;\ZZ)$ of classes~$E$ with $-K \cdot E = c_1 (E) = 1$, $E \cdot E = -1$
that can be represented by smoothly embedded spheres. 
Li--Liu~\cite{LiLiu} characterized $\overline{\cc_K} (X_n)$ as
\begin{equation} \label{e:char}
\overline{\cc_K} (X_n) \,=\,
\left\{ \alpha \in H^2 (X_n;\RR) \mid \alpha^2 \geqslant 0 \mbox{ and } 
\alpha (E) \geqslant 0 \mbox{ for all } E \in \ce_K(X_n) \right\} .
\end{equation}

We thus need to describe~$\ce_K(X_n)$. For this
define for $n \geqslant 3$ the {\it Cremona transform}\/ 
$\Cr \colon \RR^{1+n} \to \RR^{1+n}$ as the linear map taking $(x_0;\, x_1, \dots, x_n)$ to
\begin{equation} \label{e:Crem}
\left(
2x_0 - x_1-x_2-x_3;\, x_0 -x_2-x_3,\, x_0 -x_1-x_3,\,  x_0 -x_1-x_2,\, x_4,\, \dots,\, x_n    
\right) .
\end{equation}

A vector $(x_0;\, x_1, \dots, x_n)$ is {\it ordered}\/ if $x_1 \geqslant \dots \geqslant x_n$.
The {\it standard Cremona move}\/ takes an ordered vector~$(x_0; \xx)$
to the vector obtained by ordering $\Cr (x_0; \xx)$.
More generally, a {\it Cremona move}\/ is a Cremona transform followed
by any permutation of the components of~$\xx$.
%

For later use we recall the geometric origin of $\Cr$ and of Cremona moves.
For any non-zero vector $u$ in an inner-product space, 
the map $r_u(x) = x-2 \,\frac{\langle u,x \rangle}{\langle u,u \rangle} \,u$
is the reflection about~$u$, and hence an involution.
Similarly, for a class $A \in H_2(X_n;\RR)$ with $A \cdot A \neq 0$ the map
$r_A (B) = B - 2 \,\frac{A \cdot B}{A \cdot A} \,A$
is an involution of $H_2(X_n;\RR)$.
For $|A \cdot A| \in \{1,2\}$, this map is also an automorphism of $H_2(X_n;\ZZ)$.
Now take the classes $A_0 = L-E_1-E_2-E_3$ and $A_{ij} = E_i-E_j$ for $1 \leqslant i < j \leqslant n$.
Their self-intersection number is $-2$, and so for these classes, 
\begin{equation} \label{e:Pic}
r_A(B) \,=\, B + (A \cdot B) \,A .
\end{equation}
With respect to the basis $\{ L, E_1, \dots, E_n \}$ we have that $r_{A_0}$ 
is given by~\eqref{e:Crem}, that is, $r_{A_0} = \Cr \colon \ZZ^{1+n} \to \ZZ^{1+n}$ takes
the integral vector $(d;\mm) = (d;m_1, \dots, m_n)$ to
\begin{equation} \label{e:Cremdm}
\left(
2d - m_1-m_2-m_3;\, d -m_2-m_3,\, d -m_1-m_3,\,  d -m_1-m_2,\, m_4,\, \dots,\, m_n    
\right) ,
\end{equation}
and $r_{A_{ij}}$ is the transposition~$\tau_{ij}$ interchanging the $i$th and $j$th coordinate.
These involutions of $H_2(X_n;\ZZ)$ are induced by orientation preserving diffeomorphisms 
of~$X_n$.
This is clear for $\tau_{ij}$ (lift to $X_n$ an isotopy of $\CP^2$ interchanging holomorphically
small discs around the $i$th and $j$th blow-up points), 
and it holds for all classes $A_0, A_{ij}$ because each of them can be represented by a smoothly 
embedded sphere~$S$, and the smooth version of the Dehn--Seidel twist along~$S$, \cite{Se99},
is a diffeomorphism inducing~\eqref{e:Pic}, 
in view of the Picard--Lefschetz formula~\cite[p.\ 26]{AGV}.
Since the maps $\Cr$ and $\tau_{ij}$ preserve both the intersection product on~$H_2 (X_n;\ZZ)$ 
and the class~$K$, they preserve the set~$\ce_K(X_n)$.

Based on~\cite{LiLi,LiLiu} it was shown in \cite[Prop.\ 1.2.12]{McSch12} that
a homology class $E = d L - \sum_{i=1}^n m_i E_i$ belongs to~$\ce_K(X_n)$ if and only if the vector
$(d;\mm) = (d; m_1, \dots, m_n)$ is equal to $(0;\,-1,0,\dots, 0)$ up to a permutation of the~$m_i$, 
or if $(d;\mm) \in \NN \cup (\NN \cup \{0\})^n$ satisfies the Diophantine system
\begin{equation} \label{e:dio}
\sum_{i=1}^n m_i = 3d-1, \qquad \sum_{i=1}^n m_i^2 = d^2+1
\end{equation}
and reduces to $(0;\,-1, 0, \dots, 0)$ under repeated standard Cremona moves.
Summarizing, we find

\b \ni
{\bf Method 1 (Obstructive classes)}
{\it 
An embedding~\eqref{e:p2} exists if and only if
$\sum_{i=1}^n a_i^2 \leqslant\mu^2$ and $\sum_{i=1}^n a_i\, m_i \leqslant\mu d$ for all vectors $(d; \mm)$ of non-negative integers
satisfying~\eqref{e:dio} and reducing to $(0;-1,\, 0,\, \dots, \, 0)$ under repeated standard Cremona moves.
}

\begin{remark}
{\rm 
It is shown in~\cite{McD09} (see also~\cite{Hu11b}) that~\eqref{e:p2} is also equivalent to
{\it $\sum_{i=1}^n a_i\, m_i \leqslant\mu d$ for all vectors $(d; \mm)$ of non-negative integers
satisfying the Diophantine system~\eqref{e:dio}.}
It follows that if we use exceptional classes only to give {\it lower}\/ bounds
for~$c_b(a)$ (as we do in this paper),
then we do not need to show that these classes reduce to $(0;\,-1, 0, \dots, 0)$ under repeated 
standard Cremona moves.
We shall nevertheless perform these reductions, since they are readily done (see \S~\ref{ss:classes})
and since we wish to know explicit exceptional classes 
responsible for the embedding obstructions beyond the volume constraint.
} 
\end{remark}

In view of~\eqref{e:p1} we find that $E(1,a) \se P(\la, \la b)$ if and only if 
$\la \geqslant \sqrt{\frac{a}{2b}}$ and 
\begin{equation} \label{e:nonhandy}
\la (b+1) \,\geqslant\, \la \left( b m_1 + m_2 \right) + m_3 \, w_1 + \dots + m_{k+2} \, w_k 
\end{equation}
for all vectors $(d; \mm)$ of non-negative integers
satisfying~\eqref{e:dio} with $n=k+2$ 
and reducing to $(0;-1,\, 0,\, \dots, \, 0)$ under repeated standard Cremona moves.

Condition~\eqref{e:nonhandy} is not handy, since $\lambda$ appears on both sides. 
We thus better work directly in $P(\la, \la b)$ or in its compactification $S^2 \times S^2$ endowed with
the product symplectic form of the same volume.
Let $Y_{k+1}$ be the complex blow-up of~$S^2 \times S^2$ in $k+1$ points.
Then the classes $S_1 = [S^2 \times \pt]$, $S_2 = [\pt \times S^2]$ and the classes $F_1, \dots , F_{k+1}$ 
of the exceptional divisors form a basis of~$H_2 (Y_{k+1})$.
As one can guess from the picture on the right of Figure~\ref{fig:moment},
there exists a diffeomorphism 
$\psi \colon Y_{k+1} \to X_{k+2}$
such that the induced map $\psi_* \colon H_2(Y_{k+1}) \to H_2(X_{k+2})$ is given by
\begin{eqnarray*}
S_1 &\mapsto& L-E_1, \\
S_2 &\mapsto& L \phantom{-E_1} -E_2, \\
F_1 &\mapsto& L -E_1 -E_2, \\ 
F_i &\mapsto& \phantom{L -E_1 -E_2} -E_{i+1}, \quad i \geqslant 2 .
\end{eqnarray*}
If we write $(d,e;m_1, \dots, m_{k+1})$ for $d S_1 + e S_2 - m_1F_1 - \dots - m_{k+1} F_{k+1}$, we thus have
\begin{equation} \label{e:psi}
\psi_*(d,e;\mm) \,=\, \left( d+e-m_1; d-m_1, e-m_1, m_2, \dots, m_{k+1} \right) .
\end{equation}
Given $\uu \in \RR^{n_1}, \vv \in \RR^{n_2}$ we write $\langle \uu, \vv \rangle = \sum_{i=1}^{\max (n_1,n_2)} u_i \, v_i$.
In the basis $S_1, S_2, F_1, \dots, F_{k+1}$, we can reformulate Method~1 as

\b \ni
{\bf Method 1' (Obstructive classes)}
{\it 
An embedding~$E(1,a) \se P(\la, \la b)$ exists if and only if $\la \geqslant \sqrt{\frac{a}{2b}}$ and 
\begin{equation} \label{e:obs}
\la \,\geqslant\, \frac{\langle \mm, \ww (a) \rangle}{d+be} \,=:\, \mu_b(d,e;\mm)(a)
\end{equation}
for all vectors $(d,e; \mm)$ of non-negative integers that satisfy the Diophantine system
\begin{equation} \label{e:dio'}
\sum m_i = 2(d+e)-1, \qquad \sum m_i^2 = 2de+1
\end{equation}
and for which $\psi_*(d,e;\mm)$ reduces to $(0;\,-1, 0, \dots, 0)$ under repeated standard Cremona moves.
}

For the detailed translation of Method~1 to Method~1' we refer 
to the proof of Proposition~3.9 in~\cite{FrMu12}.
As we shall see in Section~\ref{s:method1}, 
the obstructions to embeddings $E(1,a) \se P(\la, \la b)$ beyond the volume 
(that is, the steps in the graphs~$c_b(a)$)
are all given by the following two series of exceptional classes~$(d,e,\mm)$:
\begin{eqnarray} \label{e:EF}
E_n &:=& \left( n,1; 1^{\times (2n+1)}\right) , \\
F_n &:=& \left( n(n+1), n+1; n+1, n^{\times (2n+3)} \right) . \notag
\end{eqnarray}

\s
In Method 1, the Cremona moves acted on integral homology classes~$(d;\mm)$.  
The second method applies Cremona moves to real cohomology classes~$\alpha$, and verifies 
by a finite algorithm whether $\alpha \in \overline{\cc_K}(X_n)$.

For convenience, we write $(\mu; a_1, \dots, a_n)$ instead of $\mu \ell - \sum_{i=1}^n a_i e_i$.
Recall that the Cremona transform~$\Cr$ on~$H_2(X_n;\ZZ)$ is induced by an orientation
preserving diffeomorphism~$\gf$ of~$X_n$.
Since $\Cr = \gf_*$ is an involution,
the map $\gf^*$ induced on cohomology~$H^2(X_n;\RR)$
is also given by formula~\eqref{e:Crem}, with respect to the Poincar\'e dual
basis $\{ \ell, e_1, \dots, e_n \}$, that is, $\gf^* = \Cr \colon \RR^{1+n} \to \RR^{1+n}$ 
takes the vector $(\mu; a_1, \dots, a_n)$ to
\begin{equation} \label{e:Cremmua}
\left(
2\mu - a_1-a_2-a_3;\, \mu -a_2-a_3,\, \mu -a_1-a_3,\,  \mu -a_1-a_2,\, a_4,\, \dots,\, a_n    
\right) .
\end{equation}

Call an ordered vector $(\mu; a_1, \dots, a_n)$ {\it reduced}\/ if $\mu \geqslant a_1+a_2+a_3$.
Using the characterisation~\eqref{e:char} and building on~\cite{LiLi,LiLiu}, 
Buse--Pinsonnault~\cite[\S 2.3]{BuPi13} and Karshon--Kessler~\cite[\S 6.3]{KaKe14} 
designed the following algorithm to decide whether an embedding~\eqref{e:p2} exists.

\m \ni
{\bf Method 2 (Reduction at a point)}
{\it Let $\alpha = (\mu; a_1, \dots, a_n)$ be an ordered vector with $\mu \geqslant 0$ and $\alpha^2 \geqslant 0$.
The sequence obtained from applying to~$\alpha$ standard Cremona moves
contains a reduced vector.
Let $(\hat \mu;\, \hat a_1, \dots, \hat a_n)$ be the first reduced vector in this sequence.
Then $\alpha \in \overline{\cc_K}(X_n)$ if and only if $\hat a_1, \dots, \hat a_n \geqslant 0$.
}

\m
We shall only need the if-part of this equivalence. 
In fact, we shall use a version thereof that will permit us to avoid finding the reordering 
after each Cremona transform: 

\begin{proposition} \label{p:crit}
Let $\alpha = (\mu; a_1, \dots, a_n)$ be a vector with $\mu \geqslant 0$ and $\alpha^2 \geqslant 0$,
and assume that there is a sequence $\alpha = \alpha_0, \alpha_1, \dots, \alpha_m$ of vectors
such that $\alpha_{j+1}$ is obtained from $\alpha_j$ by a Cremona move.
If $\alpha_m = (\hat \mu;\, \hat a_1, \dots, \hat a_n)$ is reduced and 
$\hat a_1, \dots, \hat a_n \geqslant 0$,
then $\alpha \in \overline{\cc_K}(X_n)$.
\end{proposition}

\proof
According to Proposition~4.9~(3) in~\cite{LiLiu},
a reduced vector with non-negative coefficients belongs to $\overline{\cc_K}(X_n)$.
Hence $\alpha_m \in \overline{\cc_K}(X_n)$.
By assumption, $\alpha_m = (\pi \circ \Cr) (\alpha_{m-1})$, where $\pi$ is a coordinate-permutation of~$\RR^n$.
Write $\pi$ as a product $\tau_s \circ \dots \circ \tau_1$ of transpositions.
Since $\Cr$ and $\tau_i$ are involutions, 
$$
\alpha_{m-1} \,=\, \left( \Cr \circ \tau_1 \circ \dots \circ \tau_s \right) (\alpha_m) .
$$
Recall that $\Cr$ and $\tau_i$ preserve the set~$\ce_K(X_n)$. In view of~\eqref{e:char}, 
these maps also preserve~$\overline{\cc_K}(X_n)$.
Thus $\alpha_{m-1} \in \overline{\cc_K}(X_n)$. 
Iterating this argument yields $\alpha = \alpha_0 \in \overline{\cc_K}(X_n)$.  
\proofend

It turns out that for transforming a (reducible) vector to a reduced vector
by Cremona moves, it is best to reorder every vector in the process.
In our reduction schemes in Sections~\ref{s:leftpart}--\ref{s:betagamma}
we will usually do this, 
but not always, to avoid distinguishing even more cases. 
The point of Proposition~\ref{p:crit} is that even when we do restore
the order of a vector, we do not need to prove this,
except for the head of the last vector:
All we need to make sure is that we eventually arrive at a vector
$(\hat \mu; \hat a_1, \hat a_2, \hat a_3, \hat a_4, \dots )$ that is reduced and has $\hat a_j \geqslant 0$ for all~$j$,
i.e., 
is such that 
$$
\min \{ \hat a_1, \hat a_2, \hat a_3 \} \geqslant \max \{ \hat a_4, \dots, \hat a_n \}, \quad
\hat \mu \geqslant \hat a_1 + \hat a_2 + \hat a_3, \quad
\hat a_j \geqslant 0 \mbox{ for all~$j$.}
$$
On the other hand, we will always immediately check in each step that the new 
coefficients are non-negative, 
since otherwise we may easily forget checking a coefficient at the end.

\m
Recall that an embedding $E(1,a) \se P(\la, \la b)$ exists if and only if 
an embedding~\eqref{e:p1} exists. 
Together with Proposition~\ref{p:crit} we find the following recipe.

\begin{proposition} \label{p:Kochrezept} 
An embedding $E(1,a) \se P(\la, \la b)$ exists if there exists a finite sequence of Cremona moves 
that transforms the vector~\eqref{class:eq} to an ordered vector with non-negative entries
and defect $\delta \geqslant 0$.
\end{proposition}

In our applications of this proposition we will have $\la \in (1,2)$.
The first Cremona transform thus maps 
\begin{equation*} 
\bigl( (b+1) \la;\, b \la,\, \la,\, 1^{\times \lfloor a \rfloor},\, 
                               w_1^{\times \ell_1},\, \dots \bigr) 
\end{equation*}
with $\delta = -1$ to the vector
\begin{equation*} 
\bigl( (b+1)\la -1;\, b \la -1,\, \la - 1,\, 0,\, 1^{\times (\lfloor a \rfloor -1)},\, 
                              w_1^{\times \ell_1}, \dots \bigr) 
\end{equation*}
which reorders to
\begin{equation*} 
\bigl( (b+1)\la -1;\, b \la -1,\, 1^{\times (\lfloor a \rfloor-1)} \parallel \la-1,\,
                              w_1^{\times \ell_1}, \dots \bigr) .
\end{equation*}
The action of this Cremona move on the balls 
$$
B(\ww (a)) \coprod B(\lambda) \coprod B(b \la) \,\se\, B ((b+1)\la)
$$
with $B(\ww (a)) \se P(\la, b \la)$
is illustrated in Figure~\ref{fig:moment}.

\begin{figure}[ht]
 \begin{center}
 \includegraphics[width=14cm]{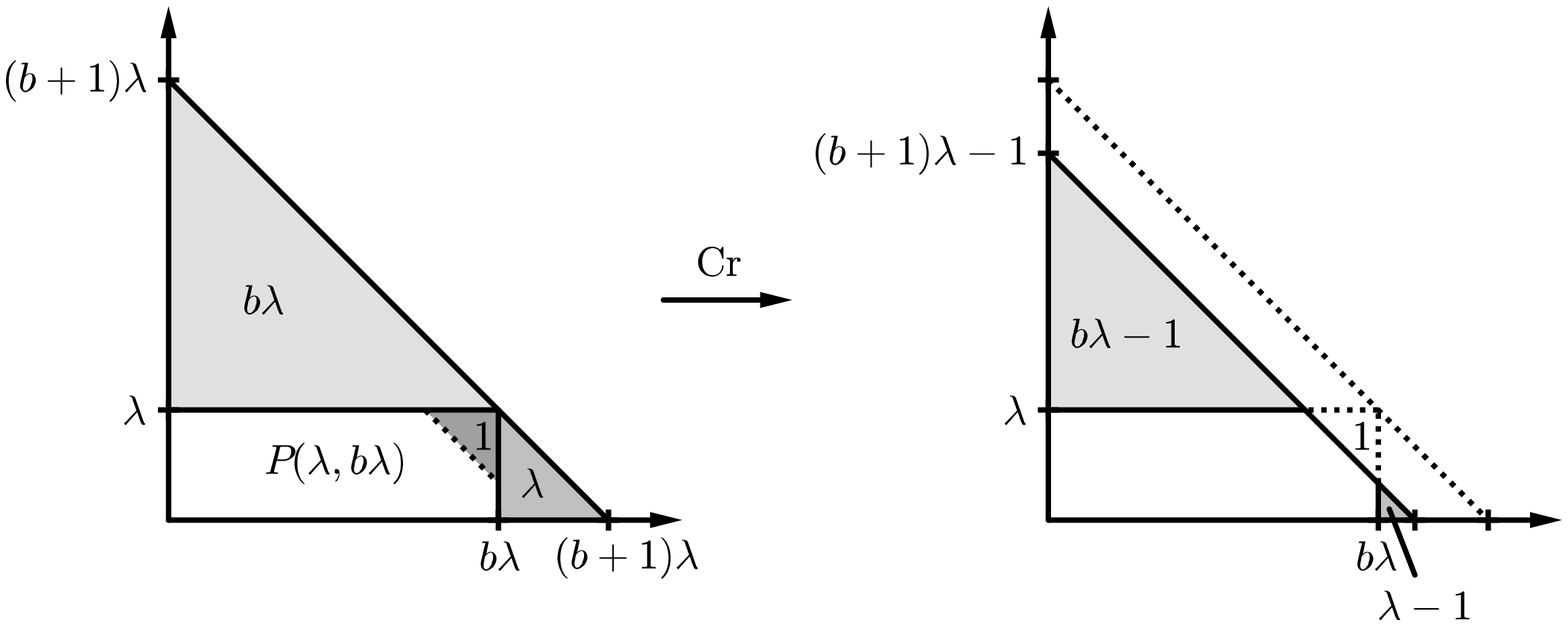}
 \end{center}
 \label{fig:moment}
\end{figure}
%
%

\begin{notation} \label{notation}
{\rm
Above, the symbol $\parallel$ indicates that the terms before~$\parallel$ 
are ordered, while the terms after~$\parallel$ are possibly not ordered, and that all 
terms before~$\parallel$ are not less than the terms after~$\parallel$.
}
\end{notation}


\b \ni
{\bf Method 3 (ECH capacities)}
In \cite{Hu11}, Hutchings used his embedded contact homology to associate with every bounded
starlike domain $U \subset \RR^4$ a sequence of symplectic capacities
$c_1 (U) \leqslant c_2 (U) \leqslant \dots$.
For an ellipsoid $E(a,b)$, this sequence is given by arranging the numbers of the form $ma+nb$
with $m, n \geqslant 0$ in nondecreasing order, with multiplicities.
For instance, 
$$
\bigl( c_k (E(1,1)) \bigr) = \bigl( 1, 1, 2, 2, 2, 3, 3, 3, 3, 4, \dots \bigr) .
$$
McDuff showed in~\cite{McD11} that ECH-capacities provide a complete set of invariants for the 
embedding problem $E(a,b) \se E(c,d)$:
$$
E(a,b) \se E(c,d) \quad \Longleftrightarrow\; c_k (E(a,b)) \leqslant c_k (E(c,d)) \;\mbox{ for all } k \geqslant 1.
$$
Since the embedding problems 
$E(1,a) \se E (\la, \la 2b)$ and $E(1,a) \se P (\la, \la b)$
are equivalent, 
it follows that
\begin{equation} \label{e:ECH}
c_b(a) \,=\, \sup_{k \geqslant 1} \left\{ \frac{c_k(E(1,a))}{c_k(E(1,2b))} \right\} .
\end{equation}
It is not clear, though, how to derive from this description of~$c_b(a)$ the graphs given 
in Theorem~\ref{t:main}.

\m
We say that an exceptional class $E = (d,e;\mm) \in \ce_K(X_n)$ is {\it $b$-obstructive}\/ 
if there is some $a \geqslant 1$ such that the obstruction function~\eqref{e:obs} is larger than 
the volume constraint,
$$
\mu_b(d,e;\mm) (a) \,>\, \sqrt{\frac{a}{2b}} .
$$
According to Method~1, it suffices to find all $b$-obstructive classes: 
The graph of~$c_b(a)$ is given as the supremum of the constraints of the $b$-obstructive classes
and of the volume constraint.
Since exceptional classes are represented by holomorphic spheres, 
this method gives insight into the nature of the obstruction to a full embedding.
It is also useful for guessing the graph of~$c_b(a)$, by first guessing a relevant set 
of $b$-obstructive classes (see Section~\ref{s:method1}).
On the other hand, it is sometimes hard to find all $b$-obstructive classes for a point~$a$. 
Method~2 is very efficient at a given point~$a$, at least if one has an idea what~$c_b(a)$
should be. 
However, the reduction scheme often depends rather subtly on the point~$a$, 
see Sections~\ref{s:leftpart}--\ref{s:betagamma}.
The reduction method is thus quite ``local in~$a$''.
While it is usually impossible to compute~$c_b(a)$ by Method~3 (see however~\cite{BPT15, GaKl13}),
this method is very useful for guessing the graph of~$c_b(a)$, since using~\eqref{e:ECH} 
and a computer one gets good lower bounds for~$c_b(a)$.

Accordingly, we have found Theorem~\ref{t:main} as follows.
We first found the exceptional classes $E_n$, $F_n$ in~\eqref{e:EF},
then used ECH-capacities to convince ourselves that there are no further
constraints besides the volume, 
and then proved this by the reduction method.
This seems to be a convenient procedure for solving 
symplectic embedding problems for which ECH-capacities
are known to form a complete set of invariants, 
such as those studied in~\cite{CG14}.

\section{Applications of Method 1} \label{s:method1}

Fix a real number $b \geqslant 1$.
As in~\eqref{e:obs} we associate with every solution $(d,e;\mm)$ of the Diophantine system~\eqref{e:dio'}
the obstruction function
\begin{equation} \label{def:ob}
\mu_b (d,e;\mm)(a) \,=\, \frac{\langle \mm, \ww(a) \rangle}{d+be}   
\end{equation}
where as before $\ww(a)$ is the weight expansion of~$a \geqslant 1$.
Further, define the error vector $\eps := \eps(a)$ by
\begin{equation*} 
\mm \,=\, \frac{d+be}{\sqrt{2ba}} \,\ww(a) + \eps.
\end{equation*}
(Here, we add zeros to $\mm$ or $\ww (a)$ if they do not have the same length.)

\subsection{Recollections} \label{ss:rec}

The following proposition generalizes Lemma~4.8 in~\cite{FrMu12}.

\begin{proposition}  \label{prop:obs} 
Fix a real number $b \geqslant 1$. 
Given a non-negative solution $(d,e;\mm)$ of~\eqref{e:dio'} and $a \geqslant 1$, we have

\begin{itemize}
\item[(i)]
$\mu_b (d,e;\mm)(a) \,\le\, \frac{\sqrt{2de+1}\sqrt a}{d+be}$;

\s
\item[(ii)]
$\mu_b (d,e;\mm)(a) > \sqrt{\frac{a}{2b}}  \;\; \Longleftrightarrow\;\;  \langle \eps,\ww(a) \rangle > 0$;

\s
\item[(iii)] 
If $\mu_b (d,e;\mm)(a) > \sqrt{\frac{a}{2b}}$, then $d = be+h$ with $|h| < \sqrt{2b}$, and
$\langle \eps,\eps \rangle = \sum \eps_i^2 < 1- \frac{h^2}{2b}$.
\end{itemize}
\end{proposition}

\proof
By the Cauchy--Schwarz inequality and since $\sum w_i^2 =a$,
$$
(d+be) \, \mu_b (d,e;\mm)(a) \,=\, \langle \mm, \ww(a) \rangle  
\,\le\, \| \mm\| \|\ww(a)\| \,=\, \sqrt{2de+1}\sqrt a ,
$$
proving~(i).
Assertion~(ii) is immediate.
To prove~(iii), we compute
\begin{eqnarray*}
2(be+h)e+1 \,=\, 
2de+1 \,=\,
\langle \mm ,\mm \rangle &=& 
\bigg\langle \frac{2be+h}{\sqrt{2ba}} \,\ww(a)+\eps, 
\frac{2be+h}{\sqrt{2ba}} \,\ww(a)+\eps \bigg\rangle \\
&=&
\frac{(2be+h)^2}{2ba} \,a + 2 \,\frac{2be+h}{\sqrt{2ba}} \langle \ww (a), \eps \rangle + \langle \eps,\eps \rangle .
\end{eqnarray*}
The first of the three summands is $2be^2 +2eh+ \frac{h^2}{2b}$, and so 
$$
1 \,=\, 
\frac{h^2}{2b} + 2 \frac{2be+h}{\sqrt{2ba}} \langle \ww (a), \eps \rangle + \langle \eps,\eps \rangle .
$$
Hence, if $\mu_b (d,e;\mm) (a) > \sqrt{\frac{a}{2b}}$, then, by~(ii),
$\langle \ww (a), \eps \rangle >0$, whence $0 \leqslant\langle \eps, \eps \rangle < 1- \frac{h^2}{2b}$.
This also shows that $|h| < \sqrt{2b}$.
\proofend

%


\subsection{Two sequences of exceptional classes, and their constraints} \label{ss:classes}

In our analysis of the functions $c_b(a)$, two sequences of exceptional homology classes
will play a role.
For each $n \in \NN$ we define the classes
\begin{eqnarray*} 
E_n &:=& \left( n,1; 1^{\times (2n+1)}\right) , \\
F_n &:=& \left( n(n+1), n+1; n+1, n^{\times (2n+3)} \right) .
\end{eqnarray*}
Notice that $E_n$ is a perfect class at $a=2n+1$,
in the sense that $\mm$ is a multiple of~$\ww (a)$.
Similarly, $F_n$ is nearly perfect at~$a=2n+4$.
While the constraints of the classes~$E_b, E_{b+1}, \dots, E_{b+\lfloor \sqrt{2b} \rfloor}$ will give the $\lceil \sqrt{2b} \rceil$ 
linear steps in the graph of~$c_b(a)$ centred at $2b+2k+1$, 
the constraint of~$F_b$ will give the affine step of~$c_b(a)$ centred at~$2b+4$.

\begin{lemma} \label{l:classes}
The classes $E_n$ and $F_n$ satisfy the Diophantine system~\eqref{e:dio'},
and their image under $\psi_*$ reduces to $(0;\,-1, 0, \dots, 0)$ under repeated standard Cremona moves.
\end{lemma}

\proof
One readily checks that the classes $E_n$ and $F_n$ satisfy the 
Diophantine system~\eqref{e:dio'}.  

For the sequel it is useful to rewrite the Cremona transform~$\Cr$ as follows:
Define the {\it defect}\/ of a vector $(d;\mm) = (d;m_1, \dots, m_k)$ by $\delta := d-m_1-m_2-m_3$.
Then~\eqref{e:Cremdm} can be written as
$$
\Cr (d;\mm) \,=\, 
\left(
d + \delta;\, m_1+\delta,\, m_2+\delta,\,  m_3+\delta,\, m_4,\, \dots,\, m_k    
\right) .
$$

The isomorphism $\psi_*$ from \eqref{e:psi} maps
$E_n = (n,1;1^{\times (2n+1)})$ to the class $(n;n-1, 1^{\times 2n})$,
which under one standard Cremona move is mapped to $(n-1;n-2, 1^{\times 2(n-1)})$, 
and hence under $n$~such moves to~$(0;-1)$.
Next, $\psi_*$ maps~$F_1$ to the class $(2;1^{\times 5})$, which reduces to~$(0;-1)$ under 
two standard Cremona moves,   
Further, for $n \geqslant 2$, 
$$
\psi_*(F_n) \,=\, \left( n^2+n; n^2-1, n^{\times (2n+3)} \right) .
$$
Under $n$ standard Cremona moves with $\delta = -n+1$ this vector reduces to
$$
\left( 2n; n^{\times 3}, n-1, 1^{\times 2n} \right) .
$$
Applying one more standard Cremona move with $\delta = -n$ yields the vector  
$(n;n-1, 1^{\times 2n})$, which reduces in $n$ steps to $(0;-1)$, as we have seen above.
%
\proofend

We next compute the constraints given by the classes~$E_n$ and~$F_n$.
In view of definition~\eqref{def:ob} and the definition of these classes, 
$$
\mu_b (E_{b+k})(a) \,=\, \frac{ \langle 1^{\times (2b+2k+1)}, \ww (a) \rangle}{2b+k} 
\quad \mbox{ and } \quad
\mu_b (F_b)(a) \,=\, \frac{ \langle \bigl( b+1, b^{\times (2b+3)} \bigr), \ww (a) \rangle}{2b(b+1)} .
$$
From this we readily find

\begin{lemma} \label{l:edges}
Fix an integer $b \geqslant 2$.

\begin{itemize}
\item[(i)]
For $k \in \bigl\{ 0,1,2, \dots,  \lfloor \sqrt{2b} \rfloor  \bigr\}$,
$$
\mu_b(E_{b+k})(a) \,=\,
\left\{\begin{array} {cl}        
\frac{a}{2b+k}    &           \mbox{if }\;  a \in [2b+2k,2b+2k+1],\\ [0.2em]
\frac{2b+2k+1}{2b+k} &           \mbox{if }\;  a \geqslant 2b+2k+1 .
\end{array}\right.
$$

\s
\item[(ii)]
$$
\mu_b(F_{b})(a) \,=\, 
\left\{
\begin{array} {cl}
\frac{ba+1}{2b(b+1)}   & \mbox{if }\;  a \in [2b+3,2b+4],\\ [0.2em]
1+\frac{2b+1}{2b(b+1)} & \mbox{if }\;  a \geqslant 2b+4 .
\end{array}\right.
$$
\end{itemize}
\end{lemma}

We in particular see that the class~$E_{b+k}$ gives rise to the linear step over~$I_b(k)$
and $F_b$ gives rise to the affine step over $[\alpha_b,\beta_b]$.


\subsection{The constraints of $E_n,F_n$ for real $b \geqslant 2$}
\label{ss:breal}

In this paragraph we compute the obstructions to the problem $E(1,a)  \to P(\la, \la b)$ 
given by the exceptional classes $E_n$ and~$F_n$
for all real $b \geqslant 2$.
This is not used in the proof of Theorem~\ref{t:main}, but supports Conjecture~\ref{con:b}.

Let $b \geqslant 2$ be a real number.
Recall that for $a \geqslant 1$ every exceptional class $E = (d,e;\mm)$ yields the constraint
$$
\mu_b(E) (a) \,=\, \frac{\langle \mm,\ww(a)\rangle}{d+be}.
$$
For $E_0 = (1,0;1)$ we have 
\begin{equation} \label{e:E0}
\mu_b(E_0)(a) \,=\, 1 ,
\end{equation}
and for $E_n = (n,1;1^{\times (2n+1)})$ with $n \geqslant 1$ we have
$$
\mu_b(E_n)(a) \,=\,
\left\{\begin{array} {cl}        
\frac{a}{n+b}    &           \mbox{if }\;  a \in [2n,2n+1],\\ [0.2em]
\frac{2n+1}{n+b} &           \mbox{if }\;  a \geqslant 2n+1 .
\end{array}\right.
$$
The class $E_n$ is $b$-obstructive on $[2n,\infty)$ only if $\frac{2n+1}{n+b} > \sqrt{\frac{2n+1}{2b}}$,
and in view of~\eqref{e:E0} we can also assume that $\frac{2n+1}{n+b} >1$, or, $n > b-1$.
The relevant values of~$n$ are thus 
$$
n \,\in\, \bigl\{ \lfloor b \rfloor, \dots, \lfloor b+\sqrt{2b} \rfloor \bigr\} 
$$
where $\lfloor b \rfloor$ is the largest integer not greater than~$b$.
The constraint~$1$ of~$E_0$ meets the first linear step, given by $E_{\lfloor b \rfloor}$, at $a = b+\lfloor b \rfloor$,
and is thus strictly above $\sqrt{\frac{a}{2b}}$ if $b \notin \NN$.
For $n \geqslant \lfloor b \rfloor$ the step of~$E_n$ meets the step of~$E_{n+1}$ at $a = \frac{(2n+1)(n+b+1)}{n+b}$,
which is $\geqslant \sqrt{\frac{a}{2b}}$ if and only if $b-n \geqslant (b-n)^2$.
The step of $E_{\lfloor b \rfloor}$ thus meets the one of $E_{\lfloor b \rfloor+1}$ above the volume constraint, 
with equality if and only if $b \in \NN$,
and all other linear steps are strictly disjoint.

Next, let $\bn$ be the ``integer closest to~$b$'', 
namely $b = \bn + \varepsilon$ with $\varepsilon \in (-\frac 12, \frac 12]$. 
Then
$$
\mu_b(F_\bn)(a) \,=\, 
\left\{
\begin{array} {cl}
\frac{\bn a+1}{(\bn+b)(\bn+1)}      & \mbox{if }\;  a \in [2\bn+3,2\bn+4],\\ [0.3em]
\frac{2\bn^2+4\bn+1}{(\bn+b)(\bn+1)} & \mbox{if }\;  a \geqslant 2\bn+4 .
\end{array}\right.
$$
But notice that this constraint is stronger than $\sqrt{\frac{a}{2b}}$ only if
$$
\mu_b(F_\bn)(2 \bn+4) \,=\, \frac{2\bn^2+4\bn+1}{(2\bn+\varepsilon)(\bn+1)} \,>\, \sqrt{\frac{\bn+2}{\bn+\varepsilon}} 
$$
or, equivalently, $\varepsilon \in \bigl(- \frac{\bn}{(\bn+1)^2}, \frac{1}{\bn+2} \bigr)$.
One readily checks that the affine step defined by $\mu_b(F_\bn)$ is strictly disjoint from the two neighbouring
linear steps given by $E_{\bn+1}$ and~$E_{\bn+2}$.

For $a \geqslant 1$ and $b \geqslant 2$ let $d_b(a)$ be the maximum of the volume constraint $\sqrt{\frac{a}{2b}}$
and the obstructions $\mu_b(E_n)(a)$ and $\mu_b(F_{\bn})$ discussed above. 
Then $d_b(a) \geqslant c_b(a)$ of course,
and Conjecture~\ref{con:b} claims that $d_b(a) = c_b(a)$ for all real $b \geqslant 2$.

\subsection{The value of $c_b$ at $2b+2+\frac{1}{2b}$}
Set $a_b := 2b+2+\frac{1}{2b}$.
We will show in~\S~\ref{ss:2b2} by the reduction method that $c_b(a_b) = \frac{2b+1}{2b}$.
(Notice that this value equals the volume constraint~$\sqrt{\frac{a_b}{2b}}$.)
Here we show this by using positivity of intersection with the class
$$
G_b \,:=\, \left( b(2b+1),2b+1; (2b)^{\times (2b+2)}, 1^{\times (2b+1)}\right) , \quad b \in \NN. 
$$
The $\mm$ of~$G_b$ is obtained from $2b \, \ww(a_b)$ by adding one~$1$,
whence $G_b$ is nearly perfect at~$a_b$.
One readily checks that $G_b$ satisfies the Diophantine system~\eqref{e:dio'} and that its image under~$\psi_*$ reduces to $(0;\,-1, 0, \dots, 0)$ under repeated standard Cremona moves. Hence $G_b$ is an exceptional class.
Its obstruction at~$a_b$ is
$$
\mu_b (G_b)(2b+2+\tfrac{1}{2b}) \,=\, \frac{2b(2b+2)+1}{2b(2b+1)} 
\,=\, \frac{2b+1}{2b}.
$$
Write $G_b = \bigl( b(2b+1), 2b+1; \mm_b,1 \bigr)$ with 
$\mm_b := \bigl( (2b)^{\times (2b+2)}, 1^{\times 2b} \bigr) = 2 b \,\ww (a_b)$.
Recall that exceptional classes are represented by embedded $J$-holomorphic spheres, whence by positivity of intersection $E \cdot E' \geqslant 0$ for any two different exceptional classes $E \neq E'$.
Applying this to~$G_b$ and any different exceptional class~$(d,e;\mm)$, 
we obtain
$$
(be+d)(2b+1) \,=\, b(2b+1)e + (2b+1)d \,\geqslant\, 
\langle \mm, (\mm_b,1) \rangle \,\geqslant\, 
\langle \mm, \mm_b \rangle \,=\,
2b \langle \mm, \ww (a_b) \rangle .
$$
Hence 
$$
\mu_b (d,e;\mm)(a_b) \,=\, 
\frac{\langle \mm,\ww(a_b) \rangle}{be+d} 
\,\le\, \frac{2b+1}{2b} , 
$$
as we wished to show.
\proofend

\begin{remarks}
{\rm 
(i)
The classes $E_1$, $E_2$ also give rise to the first two steps of $c_C(a) = c_1(a)$, 
and the class~$F_1$ gives rise to the affine step of~$c_C(a)$, 
see~\cite{FrMu12}.
This is the ``holomorphic reason'' why the first two steps of the Pell stairs 
and the affine step of~$c_C(a)$ survive to all functions $c_b(a)$, $b \geqslant 2$.
On the other hand, none of the classes $E_n$ with $n \geqslant 3$ and $F_n$ with $n \geqslant 2$
is obstructive for the problem $E(1,a) \se C^4(\la)$,
and none of the classes giving rise to the other steps of the Pell stairs, nor any of the classes
giving rise to the six exceptional steps of~$c_C(a)$ gives an obstruction for the problems 
$E(1,a) \se P (\la,\la b)$ with $b \geqslant 2$.

Similarly, $G_1$ is the first of the sequence of exceptional classes $E(\alpha_n)$ in~\cite{FrMu12}
that imply via positivity of intersection that at the foot points of the Pell stairs there is no  embedding obstruction beyond the volume constraint.

\m
(ii)
We do not know all all $b$-obstructive classes.
However, using positivity of intersection and the analogues of 
Lemmata~\ref{le:atmost1} and~\ref{le:I} 
we checked that $\mu_b (E)(2b+2k+1) < \frac{2b+2k+1}{2b+k}$ for any exceptional class $E \neq E_{b+k}$,
and that $\mu_b (E)(2b+4) \leqslant \sqrt{\frac{2b+4}{2b}}$ for any exceptional class $E \neq F_b$,
that is, $F_b$ is the only $b$-obstructive class at~$2b+4$.
For~$F_2$ this is carried out in Lemma~\ref{le:onlyF2}. 
}
\end{remarks}

\subsection{$c_b(a)$ for $a$ large}

For $b \in \NN_{\geqslant 2}$ we abbreviate
$$
v_b^+ \,:=\, v_b (  \lfloor \sqrt{2b} \rfloor) \,=\, 2b \left( \frac{2b+2\lfloor \sqrt{2b} \rfloor +1}{2b+\lfloor \sqrt{2b} \rfloor} \right)^2 .
$$
Assertion~(ii) of the following proposition improves Theorem~1.1 of~\cite{BPT15}.

\begin{proposition} \label{p:alarge}
(i)
For every $b \in  \NN_{\geqslant 2}$,
$$
c_b(a) \,=\, 
\left\{\begin{array} {cl}
\frac{2b+2 \lfloor \sqrt{2b} \rfloor+1}{2b+\lfloor \sqrt{2b} \rfloor}  & 
            \mbox{if }\;  a \in \bigl[ 2b+2 \lfloor \sqrt{2b} \rfloor+1, v_b^+ \bigr], \\ [0.2em]
\sqrt{\frac{a}{2b}}    &           \mbox{if }\;  a \geqslant v_b^+ .
\end{array}
\right.
$$

(ii)
For every real $b \geqslant 2$ we have
$c_b(a) = \sqrt{\frac{a}{2b}}$ for all $a \geqslant (\sqrt{2b}+1)^2$. 
\end{proposition}

Notice that the length of the interval $\bigl[ 2b+2 \lfloor \sqrt{2b} \rfloor+1, v_b^+ \bigr]$ in~(i) is
$$
\frac{ \bigl(2b+2 \lfloor \sqrt{2b} \rfloor + 1 \bigr) \bigl( 2b- \lfloor \sqrt{2b} \rfloor^2 \bigl)}{(2b+\lfloor \sqrt{2b} \rfloor)^2}
$$ 
and hence positive if and only if $ \lfloor \sqrt{2b} \rfloor < \sqrt{2b}$, i.e., $2b$ is not a perfect square.

\proof
Assume that $(d,e;\mm)$ is a non-negative solution of~\eqref{e:dio'}.
If $e=0$, then $(d,e;\mm) = (1,0;1)$, and so $\mu_b(d,e;\mm) (a) =1$ is smaller than the values of $c_b(a)$
claimed in~(i) and~(ii).
We can thus assume that $e \geqslant 1$.

Suppose that  $\mu_b (d,e;\mm)(a) > \sqrt{\frac{a}{2b}}$ for some $a \geqslant 1$.
Then, by Proposition~\ref{prop:obs}~(iii), $d < be + \sqrt{2b}$.
We estimate
\begin{equation} \label{e:fbed}
\mu_b (d,e;\mm)(a) \,=\, \frac{\langle \mm, \ww (a) \rangle}{be+d} \,\le\, 
                                                \frac{\sum m_i}{be+d} 
\,=\, \frac{2(d+e)-1}{be+d} \,=:\, f_{b,e}(d) .
\end{equation}
The function $d \mapsto f_{b,e}(d)$ is increasing.
We can thus further estimate
\begin{equation} \label{e:Lbe}
\mu_b (d,e;\mm)(a) \,\leqslant\, f_{b,e}(be+ \sqrt{2b}) \,=\, 
\frac{2 \bigl( be+ \sqrt{2b} +e \bigr)-1}{2be+\sqrt{2b}} \,=:\, L(b,e) .
\end{equation}

\ni
{\bf Claim 1.} $\frac{\partial}{\partial e} L(b,e) \leqslant 0$.

\proof
We compute
$$
\frac{\partial}{\partial e} L(b,e) \,=\,
\frac{2(b+1)\bigl( 2be + \sqrt{2b} \bigr)  -2b \bigl( 2(be+ \sqrt{2b}+e)-1\bigr)}{(2be+ \sqrt{2b})^2} ,
$$
which is $\leqslant 0$ if and only if the nominator is $\leqslant0$.
Expanding the nominator, we see that this holds if and only if
$b +\sqrt{2b} \leqslant b \sqrt{2b}$, which holds true because~$b \geqslant 2$.
\proofend

\ni
{\it Proof of (ii):}
Assume that $(d,e;\mm)$ is an exceptional class with $e \geqslant 1$ and $\mu_b (d,e;\mm)(a) > \sqrt{\frac{a}{2b}}$ for some $a \geqslant (\sqrt{2b}+1)^2$.
By~\eqref{e:Lbe} and Claim~1, 
$$
\mu_b (d,e;\mm)(a) \,\leqslant\,
L(b,e) \,\leqslant\ L(b,1) 
\,=\, \frac{\sqrt{2b}+1}{\sqrt{2b}} 
\,\leqslant\, \sqrt{\frac{a}{2b}},
$$
a contradiction.

\m
\ni
{\it Proof of (i):}
Assume from now on that $b \in \NN_{\geqslant 2}$.
If $e=1$, then~\eqref{e:dio'} becomes
$$
\sum m_i \,=\, \sum m_i^2 \,=\, 2d+1
$$
and so $(d,e;\mm)$ is the exceptional class $E_d = (d,1;1^{\times (2d+1)})$.
Recall that on $[2d,2d+2]$ the obstruction function 
$\mu_b (E_d)(a) = \frac{\langle \ww (a), 1^{\times (2d+1)} \rangle}{b+d}$
gives a linear step with edge at~$2d+1$.
If $\lfloor \sqrt{2b} \rfloor < \sqrt{2b}$, then the largest~$k$ for which $E_{b+k}$ 
yields a constraint strictly stronger than the volume
is $k = \lfloor \sqrt{2b} \rfloor$, because 
$\frac{2b+2k+1}{2b+k} > \sqrt{\frac{2b+2k+1}{2b}}$ if and only if $2b >k^2$.

We are left with showing that for $e \geqslant 2$ 
we have $\mu_b(d,e;\mm)(a) \leqslant \sqrt{\frac{a}{2b}}$ for all solutions $(d,e;\mm)$ of~\eqref{e:dio'}
and all $a \geqslant v_b^+$.
Assume first that $e \geqslant 3$.
Then \eqref{e:Lbe} and Claim~1 yield
$$
\mu_b(d,e;\mm) (a) \,\leqslant\, L(b,e) \,\leqslant\, L(b,3) .
$$

\ni
{\bf Claim 2.} 
$L(b,3) \leqslant \sqrt{\frac{a}{2b}}$ for all $b \in \NN_{\geqslant 2}$ and $a \geqslant v_b^+$.

\proof
It suffices to prove the claim for $a= v_b^+$.
We have 
$$
L(b,3)-1 \,=\, \frac{\sqrt{2b}+5}{6b+\sqrt{2b}} \quad \mbox{ and } \quad
\sqrt{\frac{v_b^+}{2b}} -1 \,=\, \frac{\lfloor \sqrt{2b} \rfloor +1}{2b+ \lfloor \sqrt{2b} \rfloor} .
$$
For $b \in \{2,3,4\}$ 
the inequality $\frac{\lfloor \sqrt{2b} \rfloor +1}{2b+ \lfloor \sqrt{2b} \rfloor} \geqslant 
\frac{\sqrt{2b}+5}{6b+\sqrt{2b}}$ is readily verified.
For $b \geqslant 5$ we use that $x \mapsto \frac{x+1}{2b+x}$ is increasing, and estimate
$$
\sqrt{\frac{v_b^+}{2b}} - L(b,3) \,\geqslant\, 
\frac{(\sqrt{2b}-1) +1}{2b+ (\sqrt{2b}-1)} - \frac{\sqrt{2b}+5}{6b+\sqrt{2b}}.
$$
The right hand side multiplied with the product of the denominators equals
$f(b) := 4 b \sqrt{2b} -10 \1 b - 4 \sqrt{2b} +5$.
Since $b \, f'(b) = 6b \sqrt{2b} - 2 \sqrt{2b} -10\1 b \geqslant 0$ for $b \geqslant 2$
and since $f(5) >0$, the claim follows.
\proofend

Assume now that $e=2$.
We first treat the case $b \geqslant 5$.
In view of~\eqref{e:Lbe} it suffices to show that $L(b,2) \leqslant \sqrt{\frac{v_b^+}{2b}}$, or
\begin{equation*} 
\frac{\sqrt{2b}+3}{4b+\sqrt{2b}} \,\leqslant\, 
\frac{\lfloor \sqrt{2b} \rfloor+1}{2b+\lfloor \sqrt{2b} \rfloor} .
\end{equation*}
This inequality is readily verified for $b=5$.
For $b \geqslant 6$ the stronger inequality
$$
\frac{\sqrt{2b}+3}{4b+\sqrt{2b}} \,\leqslant\, \frac{(\sqrt{2b} -1)+1}{2b+(\sqrt{2b}-1)}  
$$
holds true.
Indeed, this inequality is equivalent to $g(b) := 2 b \sqrt{2b} -6b-2\sqrt{2b}+3 \geqslant 0$,
which holds true since $b \,g'(b) = 3b \sqrt{2b} - \sqrt{2b} - 6 b \geqslant 0$ for $b \geqslant 6$
and $g(6) \geqslant 0$.

Assume now that $b \in \{2,3,4\}$.
Then $\sqrt{\frac{v_b^+}{2b}} = 1 + \frac{3}{2b+2}$.
Using~\eqref{e:fbed} this time with $d \leqslant \lfloor be+ \sqrt{2b} \rfloor$ we find
$$
\mu_b (d,2;\mm)(a) \,\leqslant\, f_{b,2}(\lfloor 2b+\sqrt{2b} \rfloor) \,=\, 
\frac{2\lfloor 2b+\sqrt{2b} \rfloor +3}{2b+ \lfloor 2b+\sqrt{2b} \rfloor} .
$$ 
For $b \in \{2,3,4\}$ the right hand side is $\leqslant 1 + \frac{3}{2b+2}$.
Proposition~\ref{p:alarge} is proven.
\proofend

\subsection{The interval $[8 \frac{1}{36},9]$ for $b=2$}

\begin{proposition} \label{prop:c2}
$c_2(a) = \frac{\sqrt a}{2}$ for $a \in [8 \frac{1}{36},9]$.
\end{proposition}

\proof
The arguments in this section are close to those 
in~\cite[\S~5.3]{McSch12} and~\cite[\S~7.3]{FrMu12}.
In fact, the last step of~$c_B(a)$ and of~$c_2(a)$ both end at~$8 \frac{1}{36}$
and are given by the class~$F_2$.
There are some differences, however,  
and so we give a complete exposition for the convenience of the reader.

\s
Fix a rational number $a = \frac pq \in (8,9)$, with $\frac pq$ in reduced form, 
with weight expansion 
\begin{eqnarray} \label{eq:xs2}
\bigl( 1^{\times \ell_0},\, w_1^{\times \ell_1}, \, \dots, \, w_N^{\times \ell_N} \bigr) .
\end{eqnarray}
Then $w_N = \frac 1q$ and $\sum_{j=0}^N \ell_j w_j = a+1-\frac 1q$ by Lemma~1.2.6 of~\cite{McSch12}.
Set $M := \ell (a) := \sum_{j=0}^N \ell_j$ and $L = \sum_{j=1}^N \ell_j = \ell (a)-8$.
Then $q \geqslant L$ by Sublemma~5.1.1 of~\cite{McSch12}.

For $b=2$ the error vector~$\varepsilon$ of an exceptional class~$(d,e;\mm)$ at~$a$ is 
\begin{equation} \label{e:me}
\mm \,=\, \frac{d+2e}{2\sqrt{a}} \, \ww (a) + \varepsilon.
\end{equation}
Define the partial error sums
$$
\sigma \,:=\, \sum_{i=\ell_0+1}^M \varepsilon_i^{2}
\quad \mbox{and} \quad
\sigma' \,:=\, \sum_{i=\ell_0+1}^{M-\ell_N} \varepsilon_i^2 \,\leqslant\, \sigma .
$$
Recall from Proposition~\ref{prop:obs}~(iii) that for an obstructive class $(d,e;\mm)$
we have $d=2e+h$ with $h \in \{-1,0,1\}$,
and $\sigma<1$ if $h=0$ and $\sigma < \frac{3}{4}$ if $|h| =1$.
For the function 
$$
y(a) \,:=\, a-3\sqrt{a}+1
$$ 
we have $y(\tfrac pq) > \frac 1q$ for all $\tfrac pq \in (8,9)$.
Write $\ell(\mm)$ for the number of positive entries in~$\mm$.

\begin{lemma} \label{le:bornd}
Let $(d,e;\mm)$ be an exceptional class such that there exists 
$a = \frac{p}{q} \in (8,9)$ with $\ell(a) = \ell(\mm)$ and
$\mu_2(d,e;\mm)(a) > \tfrac{\sqrt{a}}{2}$.
Set $v_M := \frac{d+2e}{2q\sqrt{a}}$. 
Then
\begin{itemize}
\item[(i)] 
$\left|\sum\varepsilon_{i}\right|\leqslant\sqrt{\sigma L}$.

\s
\item[(ii)] 
If\emph{ }$v_{M}<1$, then $\left|\sum\varepsilon_{i}\right|\leqslant\sqrt{\sigma'L}$.

\s
\item[(iii)] 
If $v_M \leqslant\frac{1}{2}$, then $v_{M}>\frac{1}{3}$ and $\sigma'\leqslant\frac{1}{2}$.
If $v_M \leqslant\frac{2}{3}$, then $\sigma' \leqslant \frac 79$.

\s
\item[(iv)] 
With $\delta := y(a)-\frac{1}{q}$ we have
$$
4e+h \,\leqslant\,
\tfrac{2\sqrt{a}}{\delta} \left( \sqrt{\sigma q}- (1-\tfrac h2) \right) 
\,\leqslant\,
\tfrac{2 \sqrt{a}}{\delta} \left( \tfrac{\sigma}{\delta v_M} - (1-\tfrac h2) \right).
$$
If $v_M<1$, then $\sigma$ can be replaced by $\sigma'$.
\end{itemize}
\end{lemma}

\proof
The proofs of (i), (ii) and (iii) are as for Lemma~5.1.2 in~\cite{McSch12}. 
To prove~(iv) we compute
\begin{eqnarray*}
-\sum_{i=1}^M \varepsilon_i \,=\, \tfrac{d+2e}{2 \sqrt a} \sum_{j=0}^N \ell_j w_j - \sum_{i=1}^M m_i
&=& \tfrac{d+2e}{2 \sqrt a} \bigl(a+1-\tfrac 1q\bigr) - (2d+2e-1) \\
&=& \tfrac{4e+h}{2 \sqrt a} \bigl(a+1-\tfrac 1q\bigr) - (6e+2h-1) \\
&=& \tfrac{4e+h}{2 \sqrt a} \bigl( y(a)-\tfrac 1q \bigr) + \bigl( 1 -\tfrac h2 \bigr),
\end{eqnarray*}
where we have used~\eqref{e:me} and~\eqref{e:dio'}.
Then, using $q \geqslant L$ and (i), we find
$$
\sqrt{\sigma q} \,\geqslant\ \sqrt{\sigma L} \,\geqslant\,
\tfrac{4e+h}{2\sqrt{a}} \1 \bigl( y(a)-\tfrac{1}{q} \bigr) + (1-\tfrac h2) \,=\,
\tfrac{4e+h}{2\sqrt{a}} \1 \delta + (1-\tfrac h2)
\,>\,
\delta \1 v_M \1 q.
$$
Thus $\sqrt q < \frac{\sqrt{\sigma}}{\delta \1 v_M}$, and so 
$$
4e+h \,\leqslant\, 
\tfrac{2 \sqrt{a}}{\delta} \bigl( \sqrt{\sigma q} - (1-\tfrac h2) \bigr)
\,<\,
\tfrac{2 \sqrt{a}}{\delta} \bigl( \tfrac{\sigma}{\delta \1 v_M} -(1-\tfrac h2) \bigr) .
$$
If $v_M <1$, the same arguments go through when replacing~$\sigma$ by~$\sigma'$.
\proofend

The following lemma is proven as in Lemma~2.1.7 in~\cite{McSch12}.

\begin{lemma}  \label{le:atmost1}
Assume that
$(d,e;\mm)$ is an exceptional class such that $\mu_2 (d,e;\mm)(a) > \frac{\sqrt a}{2}$
for some $a \in [8,9)$.  
Then 
\begin{itemize}
\item[(i)] 
The vector $(m_1, \dots , m_8)$ is of the form
\begin{eqnarray*}
(m, \dots, m)       \quad \text{or} \quad
(m, \dots, m, m-1)  \quad \text{or} \quad
(m+1,m, \dots, m) .
\end{eqnarray*}

\item[(ii)]
If $m_1 \neq m_8$, then $\sum_{i=1}^8 \eps_i^2 \geqslant \frac 78$.
\end{itemize}
\end{lemma}

\begin{lemma} \label{le:bound.dle13} 
There is no exceptional class~$(d,e;\mm)$ such that $\mu_2(d,e;\mm)(a) > \frac{\sqrt{a}}{2}$ for some $a \in (8,9)$ with $\ell (a) = \ell(\mm)$. 
\end{lemma}

\proof
Assume that~$(d,e;\mm)$ is an exceptional class such that 
$\mu_2(d,e;\mm)(a) > \frac{\sqrt{a}}{2}$ for some $a \in (8,9)$ with $\ell (a) = \ell(\mm)$. 

We first show that $m_1 = \ldots = m_8$.
Assume the contrary. By Lemma~\ref{le:atmost1},
$\langle \varepsilon, \varepsilon \rangle \geqslant \frac 78$ and
$\sigma \leqslant \frac 18$.
The inequality $\langle \varepsilon, \varepsilon \rangle \geqslant \frac 78$ 
and Proposition~\ref{prop:obs}~(iii) show that~$h=0$. 
Since $M > 8$ and $\sigma \leqslant \frac 18$,
we find $v_M \geqslant 1-\frac{1}{\sqrt 8} > \frac 12$.
Further, since $a \geqslant 8\frac 1q$, 
$$
\delta \,=\, y(a) -\tfrac 1q \,\geqslant\, y(8 \tfrac 1q) -\tfrac 1q \,=\, 
9-3\sqrt{8 \tfrac 1q} \,\geqslant\, 9-3\sqrt{8 \tfrac 12} \,\geqslant\, \tfrac 14.
$$
Altogether, $\frac{\sigma}{\delta\1 v_M} < 1$,
in contradiction with Lemma~\ref{le:bornd}~(iv).

\s
We are now going to show that $e$ must be small.
For this we first notice that by Lemma~\ref{le:bornd}~(iii),
\[
\begin{array}{ll}
\textrm{if }v_M \in\left[\frac{1}{3},\frac{1}{2}\right], & \textrm{then }\frac{\sigma'}{v_M}
\leqslant\frac{1/2}{1/3}=\frac{3}{2},\vphantom{{\displaystyle \frac{3}{2}}}\\
\textrm{if }v_M \in \left[\frac{1}{2},\frac{2}{3}\right], & \textrm{then }\frac{\sigma'}{v_M}\leqslant\frac{7/9}{1/2}=\frac{14}{9},\vphantom{{\displaystyle \frac{3}{2}}}\\
\textrm{if }v_M \geqslant\frac{2}{3}, & \textrm{then }\frac{\sigma}{v_M} \leqslant\frac{3}{2}.
\vphantom{{\displaystyle \frac{3}{2}}}
\end{array}
\]
For fixed $q$ and $h$, the functions
\begin{eqnarray*}
F(a,q,h) &:=& \tfrac{2 \sqrt a}{\delta} \left( \sqrt q - (1-\tfrac h2) \right), \\
G(a,q,h) &:=& \tfrac{2 \sqrt a}{\delta} \left( \tfrac{14}{9} \tfrac 1 \delta - (1-\tfrac h2) \right)
\end{eqnarray*}
are strictly decreasing for~$a \in (8,9)$.
Since $a \geqslant 8 \frac 1q$, we see from Lemma~\ref{le:bornd}~(iv) that
$$
4e+h \,\leqslant\, f(q,h),\, g(q,h) ,
$$
where $f(q,h) := F(8 \tfrac 1q, q, h)$ and $g(q,h) := G(8 \tfrac 1q, q, h)$.
Explicitly, 
\begin{eqnarray*}
f(q,h) &:=& \frac{2 \sqrt{8\tfrac 1q}}{\delta(q)} \left( \sqrt q - (1-\tfrac h2) \right), \\
g(q,h) &:=& \frac{2 \sqrt{8 \tfrac 1q}}{\delta(q)} \left( \frac{14}{9} \frac{1}{\delta (q)} - (1-\tfrac h2) \right),
\end{eqnarray*}
where $\delta (q) := y(8 \frac 1q)-\frac 1q = 9 -3 \sqrt{8\tfrac 1q}$.
We have $\frac{\partial f}{\partial q}(q,h) >0$ for $q \geqslant 3$ and 
$\frac{\partial g}{\partial q}(q,h) <0$ for all $q \geqslant 2$,
and $f(q,h) < g(q,h)$ for $q \in \{2,3\}$.
In fact, $f(q,h) = g(q,h)$ if and only if $\sqrt{q} = \frac{14}{9} \tfrac{1}{\delta(q)}$,
which happens at $q \approx 11.1$.
One readily checks that
$$
f(11,-1),\, g(12,-1) < 23, \quad 
f(11,0),\, g(12,0) < 29, \quad
f(11,1),\, g(12,1) < 35 .
$$
It follows that
$$
4e+h \leqslant 22, 28, 34 \,\mbox{ for } h = -1,0,1, \mbox{respectively},
$$
and so
\begin{equation} \label{e:e578}
e \leqslant 5 \,\mbox{ if }\, h=-1, \quad
e \leqslant 7 \,\mbox{ if }\, h=0, \quad
e \leqslant 8 \,\mbox{ if }\, h=1 .
\end{equation}

However, one readily checks that there are no solutions $(2e+h,e;\mm)$ of~\eqref{e:dio'} 
satisfying~\eqref{e:e578} and $m_1 = \ldots = m_8$.
To illustrate the computation, we take $e=8$ and $h=1$.
The Diophantine system then becomes 
$$
\sum_{i \geqslant 1} m_i = 49, \quad \sum_{i \geqslant 1} m_i^2 = 273.
$$
Since $m := m_1 = \ldots = m_8$, we must have $m \leqslant 5$.
For $m=5$ we get 
$$
\sum_{i \geqslant 9} m_i = 9, \quad \sum_{i \geqslant 9} m_i^2 = 73,
$$
which has no solution for $m_i \leqslant 5$.
Similarly there are no solutions for $m \in \{1,2,3,4\}$.
\proofend

\begin{lemma} \label{le:onlyF2}
The only exceptional class $(d,e;\mm)$ with $\mu_2(d,e;\mm) (8) > \frac{\sqrt{8}}{2}$
is $F_2 = (6,3;3,2^{\times 7})$.
\end{lemma}

\proof
Consider an exceptional class $(d,e;\mm)$ with $\mu_2(d,e;\mm) (8) > \frac{\sqrt{a}}{2}$.
By Lemma~\ref{le:I} below, $\ell (\mm) \leqslant 8$.
If $\ell (\mm) \leqslant 7$, Lemma~\ref{le:atmost1}~(i) shows that $\mm = (1^{\times 7})$;
but the only solution of~\eqref{e:dio'} with this $\mm$ is $(3,1;1^{\times 7})$,
and $\mu_2(3,1;1^{\times 7}) (8) = \frac 75 < \frac{\sqrt 8}2$.
We can thus assume that $\ell(\mm) =8$.
By Lemma~\ref{le:atmost1}, the vector~$\mm$ has the form
$$
\mm = \bigl( m^{\times 8} \bigr) \quad \text{ or } \quad 
\mm = \bigl( m^{\times 7}, m-1 \bigr) \quad \text{ or } \quad 
\mm = \bigl( m+1, m^{\times 7} \bigr)
$$
for some $m \in \NN$.

\s 
If $\mm = \bigl( m^{\times 8} \bigr)$, then the linear of the Diophantine equations yields
$8 m = 2(d+e)-1$,
which is impossible since $8m$ is even and $2(d+e)-1$ is odd.

\s \ni
In the two other cases, Proposition~\ref{prop:obs}~(iii) and Lemma~\ref{le:atmost1}~(ii) 
show that~$d=2e$.

If $\mm = \bigl( m^{\times 7}, m-1 \bigr)$, the Diophantine system becomes
$$
8m = 6e, \quad 8m^2-2m = 4 e^2.
$$  
Inserting $e = \frac 43 m$ into the second equation leads to $4m^2=9m$, which has no
solution in~$\NN$.

If $\mm = \bigl( m+1, m^{\times 7} \bigr)$, the Diophantine system becomes
$$
8m+2 = 6e, \quad 8m^2+2m = 4 e^2.
$$  
Inserting $e = \frac 13(4m+1)$ into the second equation leads to $4m^2-7m-2=0$, 
whose only integral solution is $m=2$. Hence $(d,e;\mm) = (6,3;3,2^{\times 7}) = F_2$.
\proofend

The following lemma is a version of Lemma~2.1.3 in~\cite{McSch12}. 

\begin{lemma} \label{le:I} 
Let $(d,e;\mm)$ be an exceptional class, 
and suppose that $I$ is a maximal nonempty open interval such that 
$\frac{\sqrt a}{2} < \mu_2(d,e;\mm)(a)$ for all $a \in I$.
Then there is a unique 
$a_0 \in I$ such that $\ell(a_0) = \ell(\mm)$.
Moreover $\ell(a) \geqslant \ell(\mm)$ for all $a \in I$. 
\end{lemma}

Here, the last assertion is proven as follows:
If $\ell (a) < \ell (\mm)$, then $\sum_{i \leqslant \ell(a)} m_i^2 < 2de+1$, so that
$\langle \ww(a), \mm \rangle \leqslant \| \ww(a) \| \sqrt{2de} = \sqrt{a} \sqrt{2de}$.
Hence 
$$
\mu_2(d,e;\mm)(a) \,\le\, \frac{\sqrt{2de}\sqrt a}{d+2e} \,\le\,
\frac{\sqrt a}{2} .
$$

\ni
{\it End of the proof of Proposition~\ref{prop:c2}:}
Suppose to the contrary that 
$\mu_2(d,e;\mm)(a) > \frac{\sqrt a}{2}$ for some $a \in [8\frac 1{36},9)$.
By Lemma~\ref{le:I} we may choose $a_0$ with $\ell(a_0) =\ell(\mm)$ in the interval~$I$ 
containing~$a$ on which this inequality holds.    
 
Assume that $a_0 \leqslant 8$. Then $a_0 \leqslant 8 < a$, and so $8 \in I$.
Then Lemma~\ref{le:onlyF2} shows that $(d,e;\mm) = F_2$.
But $F_2$ is not obstructive for $a \geqslant 8 \frac{1}{36}$.

Hence $a_0 >8$. We already know from Proposition~\ref{p:alarge} that $c_2(a) = \frac{\sqrt a}{2}$ 
for $a \geqslant 9$. Hence $a_0 \in (8,9)$. Hence Lemma~\ref{le:bound.dle13} applies,  
and yields the desired contradiction.
\proofend


\section{First applications of the reduction method} 

In this section we first use the reduction method to prove the 
equivalence~\ref{e:EP}.
We then use this method to prove that the obstructions given by the exceptional
classes~$E_n$ and~$F_n$ are sharp at their edges,
and then to compute $c_b(a)$ at end points of the first linear step.

As in \S~\ref{ss:classes} we
define the {\it defect}\/ of a vector $(\mu; \aa) = (\mu; a_1, \dots, a_k)$ by $\delta := \mu-a_1-a_2-a_2$.
Then the Cremona transform~\eqref{e:Cremmua} can be written as
$$
\Cr (\mu;\aa) \,=\, 
\left(
\mu + \delta;\, a_1+\delta,\, a_2+\delta,\,  a_3+\delta,\, a_4,\, \dots,\, a_k    
\right) .
$$

\subsection{Proof of the equivalence~\ref{e:EP}} 
\label{ss:proof.eq}
By continuity we can assume that $a$ is rational.
Recall that $E(1,a) \se P(\la, \la b)$ if and only if there exists an 
embedding~\eqref{e:p1}.
By the Nonsqueezing Theorem we must have $\la \geqslant 1$.
Hence Method~2 formulated in~\S~\ref{ss:trans}
shows that an embedding~\eqref{e:p1} exists
if and only if $\la \geqslant \sqrt{\frac{a}{2b}}$ and if the first reduced vector in the orbit of
\begin{equation} \label{class:eq}
\bigl( \la (b+1); \la b, \la, \ww (a) \bigr) 
\end{equation}
under standard Cremona moves has no negative entries.

The weight decomposition of the ellipsoid $E ((2b-1)\la, 2b\la))$ is $\bigl( (2b-1)\la, \la^{\times (2b-1)} \bigr)$.
The main result of~\cite{McD09} thus shows that $E(1,a) \se E(\la, 2b \la)$ if and only if 
$$
B(\ww (a)) \coprod B\bigl( (2b-1)\la \bigr) \coprod_{2b-1} B(\la) \,\se\, B(2b \la) .
$$ 
Method~2 shows that such an embedding exists
if and only if $\la \geqslant \sqrt{\frac{a}{2b}}$ and if the first reduced vector in the orbit of
\begin{equation} \label{e:secondvector}
\bigl( 2b \la; (2b-1) \la , \la^{\times (2b-1)}, \ww (a) \bigr) 
\end{equation}
under standard Cremona moves has no negative entries.
Applying $b-1$ standard Cremona moves with defect $\delta = - \lambda$ 
to the vector~\eqref{e:secondvector} we reach the vector~\eqref{class:eq}.
\proofend

In the rest of this paper we will show that besides for the volume constraint~$\sqrt{\frac{a}{2b}}$
there are no other obstructions to the embedding problem $E(1,a) \se P(\la, \la b)$ than those
given by the exceptional classes~$E_n$ and~$F_n$.
For this it suffices to show that if we take for~$\la$ 
the value claimed for~$c_b(a)$ in Theorem~\ref{t:main}, 
then there exists an embedding $E(1,a) \se P(\la, \la b)$.
This problem, in turn, we solve by the recipe formulated in Proposition~\ref{p:Kochrezept}.

\subsection{The value of $c_b(a)$ at $a = 2b+2k+1$, and at $a = 2b$ and $a= 2b+2+\frac{1}{2b}$}
\label{ss:2b2}

\begin{lemma} \label{le:41}
$c_b(2b+2k+1) \leqslant \frac{2b+2k+1}{2b+k}$ for $k \in \bigl\{ 0,1,2, \dots,  \lfloor \sqrt{2b} \rfloor  \bigr\}$.
\end{lemma}

\proof
Set $\la = \frac{2b+2k+1}{2b+k} = 1 + \frac{k+1}{2b+k} \in (1,2)$.
Then one standard Cremona move with $\delta =-1$ takes the 
vector~$\bigl( \la (b+1); \la b, \la, 1^{\times (2b+2k+1)} \bigr) $ to
$$
\bigl( \la (b+1)-1; \la b-1, 1^{\times (2b+2k)}, \la-1 \bigr) .
$$
Since $\la b - 1 + (b+k) (\la-2) =0$, applying 
$b+k$ Cremona moves with $\delta = \la -2$ to this vector yields the vector~$(\la; (\la-1)^{\times (2b+2k+1)})$,
which is reduced, since $\delta = 3-2 \la = \frac{2b-k-2}{2b+k} \geqslant 0$ for $k \leqslant \sqrt{2b}$ and $b \geqslant 2$.
\proofend

\begin{lemma} \label{le:42}
$c_b(2b) =1$ and $c_b(2b+2+\frac{1}{2b}) = \frac{2b+1}{2b}$.
\end{lemma}

\proof
In view of the volume constraint $c_b(a) \geqslant \sqrt{\frac{a}{2b}}$, it suffices to show 
the inequalities
$c_b(2b) \leqslant 1$ and $c_b(2b+2+\frac{1}{2b}) \leqslant \frac{2b+1}{2b}$.

Set $\la = 1$.
Then $b$ Cremona moves with $\delta =-1$ take the 
vector~$\bigl( b+1; b, 1^{\times (2b+1)} \bigr) $ to $(1;1)$, which is reduced.

Set $\la = \frac{2b+1}{2b} = 1 + \frac{1}{2b}$.
Then one standard Cremona move with $\delta =-1$ takes the 
vector~$\bigl( \la (b+1); \la b, \la, 1^{\times (2b+2)}, \left(\frac{1}{2b}\right)^{\times 2b} \bigr)$ to
$$
\bigl( \la (b+1)-1; \la b-1, 1^{\times (2b+1)}, \left(\tfrac{1}{2b}\right)^{\times (2b+1)} \bigr) .
$$
Since $\la b - 1 + b (\la-2) =0$, applying 
$b$~Cremona moves with $\delta = \la -2$ yields the vector~$\bigl( \la; 1, \left(\tfrac{1}{2b}\right)^{\times (4b+1)} \bigr)$.
Applying $2b$ Cremona moves with $\delta = \frac{1}{2b}$ yields the vector $\bigl( \frac{1}{2b}; \frac{1}{2b} \bigr)$,
which is reduced.
\proofend

\begin{corollary} \label{cor:2b3}
Theorem~\ref{t:main} holds for $a \in [1,2b+3]$.
\end{corollary}

\proof
By Gromov's Nonsqueezing Theorem, $E(1,1) \se P(\la, \la b)$ implies $\la \geqslant 1$.
(In our language this reads $\mu_b (E_0) (1) = 1$ for $E_0 := (1,0;1)$.)
Since the function~$c_b$ is monotone increasing, 
this and $c_b(2b) =1$ show that $c_b(a) = 1$ for $a \in [1,2b]$.

The functions $c_b$ have the scaling property
$$
\frac{c_b(\la a)}{\la a} \,\leqslant\, \frac{c_b(a)}{a} 
\quad \mbox{ for all } \,\la\, \geqslant 1,
$$
see \cite[Lemma 1.1.1]{McSch12} for the easy proof.
Therefore,

\begin{lemma} \label{le:center}
If for two values $a_0 < a_1$ the points $(a_0,c_b(a_0))$ and $(a_1,c_b(a_1))$
lie on a line through the origin, then
the whole segment between these two points belongs to the graph of~$c_b$,
that is, $c_b$ is linear on $[a_0,a_1]$.
\end{lemma}

\ni
Lemmata~\ref{l:edges}~(i), \ref{le:41} and~\ref{le:42} thus show that the graph of~$c_b$
on~$[1,2b+3]$ is as in Figure~\ref{fig:left}.
\proofend

\subsection{Organization of the proof of Theorem~\ref{t:main}}
 
We order the rest of the proof by increasing difficulty.

For $b \in \NN_{\geqslant 5}$ and
$k = 2, \dots, \lfloor \sqrt{2b} \rfloor -1$, the intervals $I_b(k)$ and $I_b(k+1)$ enclose the interval $[v_b(k), u_b(k+1)]$, that contains the point $2b+2k+2$.
We first show that $c_b(a) = \sqrt{\frac{a}{2b}}$ on this interval.
More precisely, we subdivide this interval into its left and right part
$$
L_b(k) := [v_b(k), 2b+2k+2], \qquad
R_b(k) := [2b+2k+2, u_b(k+1)]
$$
and show in Section~\ref{s:leftpart} and~Section~\ref{s:rightpart} that 
$c_b(a) = \sqrt{\frac{a}{2b}}$ on $L_b(k)$ and $R_b(k)$, respectively.
Theorem~\ref{t:main} then follows for all $a \geqslant 2b+5$.
Indeed, together with Lemmata~\ref{l:edges}~(i) and~\ref{le:41}, we now know
that for $k \geqslant 2$ the edge point and the two end points of the linear steps
lie on the graph of~$c_b(a)$, and hence by Lemma~\ref{le:center} these linear steps
belong to~$c_b(a)$ entirely.
Further, by Proposition~\ref{p:alarge}~(i), Theorem~\ref{t:main} holds for 
$a \geqslant v_b(\lfloor \sqrt{2b} \rfloor)$.

\begin{figure}[ht]
 \begin{center}
 \includegraphics[width=10cm]{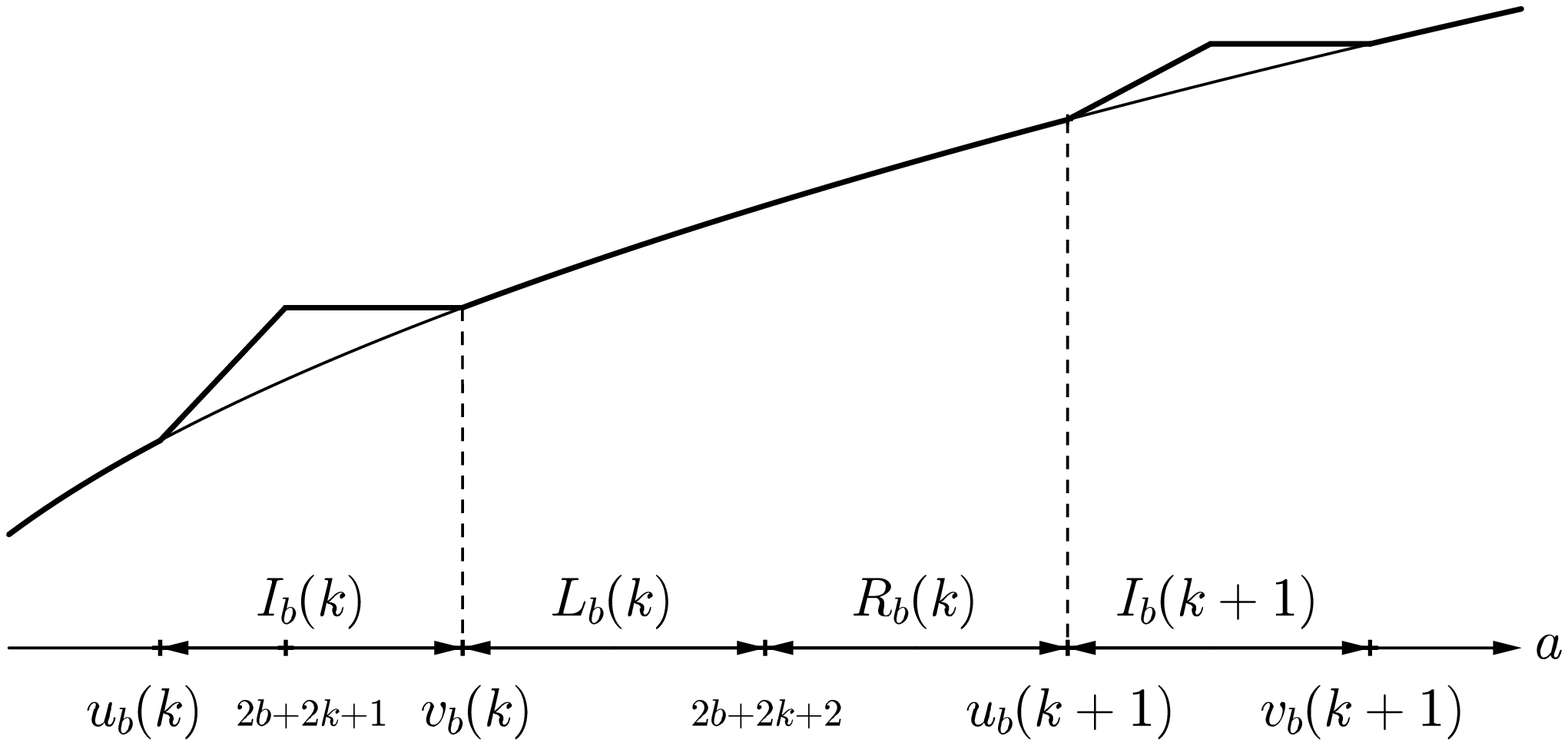}
 \end{center}
 \label{fig:LR}
\end{figure}
%
%

We already know from Corollary~\ref{cor:2b3} that Theorem~\ref{t:main} holds for $a \leqslant 2b+3$.
We are thus left with the interval $[2b+3,2b+5]$.
It suffices to treat the subinterval $[v_b(1), u_b(2)]$.
Indeed, we then know that $c_b(2b+3) = c_b(v_b(1))$, 
whence the second linear step is established, 
and we already know that the third linear step, that begins at~$u_b(2)$, belongs to~$c_b(a)$.
(Note that for $b=2$ there is no third linear step, but then $u_b(2) = 2b+5 = 9$.)
Recall that
$$
v_b(1) \,<\, \alpha_b \,<\, 2b+4 \,<\, \beta_b \,<\, u_b(2)  .   
$$
We shall treat the interval $[v_b(1), 2b+4]$ in~Section~\ref{s:vb1.2b4}.
The case $b=2$ is then complete, since $c_2(8) = \frac{17}{12} = c_2 (8 \frac{1}{36})$ 
and in view of Proposition~\ref{prop:c2}.
The interval~$[2b+4,u_b(2)]$ for $b  \geqslant 3$ is treated in Section~\ref{s:betagamma}.
Showing $c_b(a) = \sqrt{\frac{a}{2b}}$ on the intervals 
$[v_b(1),\alpha_b]$ and $[\beta_b,u_b(2)]$ is the hardest part of the paper, 
since on these intervals the reduction algorithm is rather intricate.
On the other hand, establishing the affine segment over $[\alpha_b,2b+4]$ will be easier,
and it turns out that the reduction method establishes the affine steps of $c_B(a)$ and~~$c_C(a)$
much faster than the positivity of intersection argument used in~\cite{McSch12} and~\cite{FrMu12}.

\m
Since the embedding functions $c_b(a)$ are continuous, it suffices to compute them 
on a dense set.
In the rest of the paper we shall assume that $a \geqslant 1$ is rational. 
Hence $a$ has a finite weight expansion 
$\ww (a) = \left( 1^{\times \lfloor a \rfloor},\, w_1^{\times \ell_1},\, w_2^{\times \ell_2},\, \dots \right)$.
Sometimes it will be convenient to assume also that 
$\ell_1 \geqslant 1$ or $\ell_2 \geqslant 1$ or $\ell_3 \geqslant 2$,
which holds for a dense set of rational~$a$.


\section{The intervals $L_b(k) = [v_b(k), 2b+2k+2]$}
\label{s:leftpart}

Recall that
$$
v_b(k) \,=\, 2b \left( \frac{2b+2k+1}{2b+k} \right)^2 .
$$

\begin{theorem} \label{t:Lbk}
Assume that $b \in \NN_{\geqslant 5}$ and that 
$k \in \bigl\{ 2, \dots, \lfloor \sqrt{2b} \rfloor -1 \bigr\}$.
Then
$c_b(a) = \sqrt{\frac{a}{2b}}$ \,for $a \in L_b(k)$.
\end{theorem}

\proof
The weight expansion at $a \in L_b(k)$ is 
$$
\ww(a) \,=\, \left( 1^{\times 2(b+k)+1},\, w_1^{\times \ell_1},\, 
                          w_2^{\times \ell_2},\, \dots \right) .
$$
Define the numbers $\lambda$ and $z_1$, $z_2$ by
\begin{eqnarray*}
\lambda &=& \sqrt{\frac{a}{2b}} \,=\, \sqrt{\frac{2(b+k)+1+w_1}{2b}} \,=:\, 1+z_1 ,  \\ 
z_2 &:=& (2b+k) \la - (2b+2k+1) \\
&=& \sqrt{\frac{2(b+k)+1+w_1}{2b}} \,(2b+k) - (2b+2k+1) .
\end{eqnarray*}

\begin{lemma} \label{le:ele}
{\rm (i)} $2z_1 \leqslant 1+z_2$. 

\s
{\rm (ii)}
$z_2 \geqslant 0$ and $z_2 \leqslant w_1$.

\s
{\rm (iii)}
For $k \geqslant 3$ and $\ell_1=1$ we have $w_2+z_2-z_1 \geqslant 0$.
\end{lemma}

\proof
(i)
We wish to show that
$$
2b+2k \leqslant 2 + (2b+k-2) \la .
$$
We show that this inequality even holds if $w_1 \leqslant 0$ in $\lambda$ is set zero, 
i.e., that
$$
2b+2k \leqslant2+ \sqrt{\frac{2(b+k)+1}{2b}} \,(2b+k-2) .
$$
After solving for the root, squaring and multiplying with $2b(2b+k-2)^2$
we find that this inequality is equivalent to
$$
4b^2+(2k+1)(k-2)^2 + 2b (k^2 -2k-4) \geqslant 0
$$
which holds true since $k \geqslant 2$ and $b \geqslant 2$.

\m
(ii)
Note that $z_2=0$ at the left boundary $v_b(k)$ of~$L_b(k)$.
Since $z_2$ is increasing on~$L_b(k)$, we see that $z_2 \geqslant 0$.

At $v_b(k)$ we have $w_1 \geqslant 0 = z_2$.
In order to show that $z_2 \leqslant w_1$ on~$L_b(k)$, it thus suffices to check that
the derivative of the function $f_{b,k}(w_1) = w_1-z_2(b,k,w_1)$ is non-negative, 
i.e.,
$$
f'_{b,k}(w_1) \,=\, 1-\frac 12 \sqrt{\frac{2b}{2(b+k)+1+w_1}} \, \frac{2b+k}{2b} \,\geqslant\, 0.
$$
This holds if it holds for $w_1=0$, i.e., if
$$
\frac{4b}{2b+k} \,\geqslant\, \sqrt{\frac{2b}{2b+2k+1}} .
$$
This is equivalent to 
$$ 
8b (2b+2k+1) \,\geqslant\, 4b^2 +4bk+k^2
$$
which hold true since $k^2 \leqslant2b$.

\m
(iii)
Fix $k \geqslant 3$ and $b \geqslant 2$.
Define the function $f_{b,k}$ on $\left[ v_b(k) - \lfloor v_b(k) \rfloor, 1 \right]$ by
\begin{equation} \label{e:fbk}
f_{b,k}(w_1) \,:=\, w_2+z_2-z_1 \,=\, -w_1 +(2b+k-1) \lambda -(2b+2k)+1 .
\end{equation}
Then $f'_{b,k} \leqslant 0$. Indeed, this is equivalent to
$$
2b+k-1 \,\leqslant\, 2 \sqrt{2b (2b+2k+1+w_1)}
$$
which follows from
$$
2b+k \,\leqslant\, 2 \sqrt{2b (2b+2k+1)}.
$$
It therefore suffices to show that $f_{b,k}(1) \geqslant 0$, i.e., 
$$
\sqrt{\frac{b+k+1}{b}} \,\geqslant\, \frac{2b+2k}{2b+k-1} .
$$
Squaring and multiplying by $b(2b+k-1)^2$ this becomes
$$
(1+k) \bigl( (k-3)b + (k-1)^2 \bigr) \,\geqslant\, 0
$$
which holds true since $k \geqslant 3$.
\proofend


\medskip \noindent
In view of Proposition~\ref{p:Kochrezept} we wish to transform the vector
$$
\bigl( (b+1) \lambda, \, b \lambda, \, \lambda, \, \ww (a) \bigr) 
$$
to a reduced vector by a finite sequence of Cremona moves.
One Cremona move yields
$$
\bigl( (b+1) \lambda -1; \, b \lambda -1, \, 1^{\times 2 (b+k)}  \parallel z_1, \,
w_1^{\times \ell_1} ,\, w_2^{\times \ell_2}, \, \dots \bigr) ,
$$
Here and in the sequel we use the notation explained in~Notation~\ref{notation}.
Next, $b+k$ Cremona moves with $\delta = \lambda -2 = z_1-1$ yield
\begin{equation} \label{v:la}
\bigl( \lambda +z_2;\, z_2,\, w_1^{\times \ell_1},\,  z_1^{\times 2(b+k)+1},\, 
                                w_2^{\times \ell_2},\, \dots \bigr)
\end{equation}
Assume that $z_1 \geqslant w_1$. Since $z_2 \leqslant w_1$, the vector~\eqref{v:la} reorders to
\begin{equation} \label{v:la'}
\bigl( \lambda +z_2;\, z_1^{\times 2(b+k)+1}, \, w_1^{\times \ell_1} \parallel  z_2,\, 
                               w_2^{\times \ell_2},\, \dots \bigr) .
\end{equation}
Then $\delta = \lambda +z_2 -3z_1 = 1+z_2-2z_1 \geqslant 0$ by Lemma~\ref{le:ele}~(i).
Since all entries of~\eqref{v:la'} are non-negative, this vector is reduced.

From now on we thus assume that $w_1 \geqslant z_1$. 
Then the vector~\eqref{v:la} becomes
\begin{equation} \label{v:la''}
\bigl( \lambda +z_2;\, w_1^{\times \ell_1} \parallel  z_1^{\times 2(b+k)+1},\, z_2,\, 
                          w_2^{\times \ell_2},\, \dots \bigr) .
\end{equation}
If $\ell_1 \geqslant 3$, then $\delta = 1+z_1+z_2-3w_1 \geqslant z_1+z_2 \geqslant 0$.
If $\ell_1 = 2$, then 
$$
\delta \,=\, 1+z_1+z_2-2w_1 - \left( z_1 \mbox{ or } z_2 \mbox{ or } w_2 \right)
\,\geqslant\, 1 - (2w_1+w_2) \,\geqslant\, 0 .
$$

So assume that $\ell_1=1$, that is, the vector~\eqref{v:la''} is
$$
\bigl( \lambda +z_2;\, w_1 \parallel  z_1^{\times 2(b+k)+1},\, z_2,\, w_2^{\times \ell_2},\, \dots \bigr)
$$

\m
\ni
{\bf Case 1. $z_1 \geqslant z_2, w_2$.}
Then the vector at hand is
$$ 
\bigl( \lambda +z_2;\, w_1,\, z_1^{\times 2(b+k)+1} \parallel z_2,\, w_2^{\times \ell_2},\, \dots \bigr)
$$
Hence $\delta = 1+z_1+z_2-w_1-2z_1 = w_2+z_2-z_1$. 
For $k\geqslant 3$ this number is non-negative by Lemma~\ref{le:ele}~(iii).
Assume now that $k=2$ and that $\delta = w_2+z_2-z_1 <0$.
We reduce the above vector $b+2$ times by~$\delta$ and get
$$
\bigl( w_2+z_1+z_2 +\ast;\, \ast = w_1+(b+2)(w_2+z_2-z_1),\, z_1,\, (w_2+z_2)^{\times (2b+4)} 
\parallel z_2,\, w_2^{\times \ell_2},\, \dots \bigr) .
$$ 
The order is right since by assumption $w_2+z_2 \leqslant z_1$ and by the following lemma.
For this vector, $\delta =0$.

\begin{lemma}
$w_1+(b+2)(w_2+z_2-z_1) \geqslant z_1$
\end{lemma}

\proof
Define the function $f_b$ on $\left[ v_b(2) - \lfloor v_b(2) \rfloor, 1 \right]$ by
\begin{equation} \label{e:fb2}
f_b(w_1) \,:=\, w_1 + (b+2) (w_2+z_2-z_1)-z_1 .
\end{equation}
We compute
$$
f_b(w_1) \,=\, -(b+1) w_1 + (2b^2+5b+1) \lambda - (2b^2+7b+5)
$$
where $\lambda = \sqrt{\frac{2b+5+w_1}{2b}}$.
We wish to show that $f_b(w_1) \geqslant 0$.
We estimate 
$$
f_b'(w_1) \,=\,
-(b+1) + \frac{2b^2+5b+1}{2 \sqrt{2b (2b+5+w_1)}} \,\leqslant\,
 -(b+1) + \frac{2b^2+5b}{4b} \,\leqslant\, 0.
$$
Hence $f_b(w_1) \geqslant f_b(1) = -(b+1) + (2b^2 +5b +1) \sqrt{\frac{b+3}{b}} -(2b^2+7b+5)$.
The right hand side is $\geqslant 0$ if and only if
$$
\sqrt{\frac{b+3}{b}} \,\geqslant\, \frac{2(b^2 +4b +3)}{2b^2 +5b +1} .
$$
Squaring and multiplying by $b (2b^2 +5b +1)^2$ we find that this is equivalent to
the inequality $(b+3)(b-1)^2 \geqslant 0$, which holds true.
\proofend

\m
\ni
{\bf Case 2. $z_2 \geqslant z_1, w_2$.}
Then $\delta = 1+z_1-w_1- \left( z_1 \mbox{ or } w_2 \right) \geqslant 0$.

\m
\ni
{\bf Case 3. $w_2 \geqslant z_1, z_2$.}
The vector at hand is
$$
\bigl( \lambda +z_2; \, w_1,\, w_2^{\times \ell_2} \parallel  z_1^{\times 2(b+k)+1},\, z_2,\, 
              w_3^{\times \ell_3},\, w_4^{\times \ell_4},\, \dots \bigr)
$$

\s \noindent
{\it Subcase 3a:} $\ell_2 \geqslant 2$.
Then $\delta = z_1+z_2-w_2$.
Assume that $\delta <0$, i.e., $w_2 >z_1+z_2$.

If $\ell_2 = 2m_2 \geqslant 2$ is even, we reduce $m_2$ times by~$\delta$ and get
\begin{eqnarray*}
\bigl( z_1 +z_2 + w_2 + \ast; \, \ast = w_1 + m_2(z_1+z_2-w_2) \parallel \\ 
\quad
(z_1+z_2)^{\times m_2},\, z_1^{\times 2(b+k)+1},\, z_2,\, 
w_3^{\times \ell_3},\, w_4^{\times \ell_4},\, \dots \bigr) .
\end{eqnarray*}
Here, $\ast \geqslant z_1+z_2$ and $\ast \geqslant w_3 = w_1-\ell_2w_2$ 
because $m_2 w_2 \leqslant \ell_2 w_2 \leqslant w_1$.

\ni
If $z_1+z_2 \geqslant w_3$, then 
$$
\delta \,=\, w_2- \left( z_1+z_2 \mbox{ or } z_1 \mbox{ or }  z_2 \mbox{ or } w_3 \right)
\,\geqslant\, w_2 - (z_1+z_2) \,>\, 0 .
$$
If $w_3 \geqslant z_1+z_2$, then
$$
\delta \,=\, z_1+z_2 +w_2-w_3 - \left( w_3 \mbox{ if } \ell_3 \geqslant 2, \, 
z_1+z_2 \mbox{ or } w_4 \mbox{ if } \ell_3 =1 \right).
$$
In the first case, $\delta \geqslant 0$ since $w_2=\ell_3 w_3 +w_4 \geqslant 2w_3$, 
and in the second case, 
$\delta = w_2-w_3 \geqslant 0$ or $\delta = z_1+z_2 \geqslant 0$.

If $\ell_2 = 2m_2+1 \geqslant 3$ is odd, we again reduce $m_2$ times by~$\delta$ and get
$$
\bigl( z_1 +z_2 + w_2 + \ast;\, \ast = w_1 + m_2(z_1+z_2-w_2),\, w_2 
\parallel 
(z_1+z_2)^{\times m_2},\, z_1^{\times 2(b+k)+1},\, z_2,\, 
w_3^{\times \ell_3},\, w_4^{\times \ell_4},\, \dots \bigr) .
$$
If $z_1+z_2 \geqslant  w_3$, then $\delta=0$.
If $w_3 > z_1+z_2$, then $\delta = z_1+z_2-w_3 <0$.
The vector at hand is
$$
\bigl( z_1 +z_2 + w_2 + \ast;\, \ast = w_1 + m_2(z_1+z_2-w_2),\, w_2,\, w_3^{\times \ell_3}
\parallel  
(z_1+z_2)^{\times m_2},\, w_4^{\times \ell_4},\, \dots \bigr) ,
$$
and applying one more Cremona transform yields the vector
$$
\bigl( z_1 +z_2 + w_2 + \ast;\, \ast,\, w_2 + z_1+z_2-w_3,\, w_3^{\times \ell_3-1}
\parallel  
(z_1+z_2)^{\times m_2+1},\, w_4^{\times \ell_4},\, \dots \bigr) 
$$
where now $\ast = w_1 + m_2(z_1+z_2-w_2) + (z_1+z_2-w_3)$.
The ordering holds since if $\ell_3 \geqslant 2$ then $w_2 + z_1+z_2-w_3 \geqslant w_2-w_3 \geqslant w_3$, 
and if $\ell_3=1$ then $w_2  + z_1+z_2-w_3 = z_1+z_2+w_4$.
Now $\delta = w_3- \left( w_3 \mbox{ or } z_1+z_2 \mbox{ or } w_4 \right) \geqslant 0$.

\m \noindent
{\it Subcase 3b:} $\ell_2 = 1$. 
Then $\delta = 1+z_1+z_2-w_1-w_2 -x = z_1+z_2-x$ with $x \in \{z_1,z_2,w_3\}$.
If $x \in \{z_1,z_2\}$ then $\delta \in \{z_2,z_1\} \geqslant 0$.
If $x = w_3$,  
then the vector at hand is
$$
\bigl( \lambda +z_2;\, w_1,\, w_2,\, w_3^{\times \ell_3} \parallel  
z_1^{\times 2(b+k)+1},\, z_2,\, w_4^{\times \ell_4},\,\dots \bigr)
$$
Notice that $w_2 = 1-w_1$ and $w_3 = w_1-w_2$.
We have $\delta = z_1+z_2-w_3$.
If $w_3 > z_1+z_2$, we apply one more Cremona transform and obtain
$$
\bigl( z_1+z_2 +w_2 + \ast;\, \ast = z_1+z_2 + w_2,\, z_1+z_2 +w_2 -w_3,\, w_3^{\times (\ell_3-1)} 
\parallel \\
z_1+z_2,\, z_1^{\times 2(b+k)+1},\, z_2,\, w_4^{\times \ell_4},\, \dots \bigr)
$$
The ordering is right because if $\ell_3 \geqslant 2$ then $w_2 \geqslant 2 w_3$,
and if $\ell_3=1$ then $w_2-w_3=w_4$.

If $\ell_3 \geqslant 2$ then $\delta =0$.

If $\ell_3=1$ then $\delta = w_3-(z_1+z_2)>0$ or $\delta =w_3-w_4 \geqslant 0$.

\s \ni
The proof of Theorem~\ref{t:Lbk} is complete.


\section{The intervals $R_b(k) = [2b+2k+2, u_b(k+1)]$} 
\label{s:rightpart}

\begin{theorem} \label{t:dan}
Assume that $b \in \NN_{\geqslant 5}$ and that 
$k \in \bigl\{ 2, \dots, \lfloor \sqrt{2b} \rfloor -1 \bigr\}$.
Then
$c_b(a) = \sqrt{\frac{a}{2b}}$ \,for $a \in R_b(k)$.
\end{theorem}
 
\proof
For notational convenience we shift the index~$k$ by one,
and prove that $c_b(a) = \sqrt{\frac{a}{2b}}$ for $a \in R_b(k-1)$
and $k \in \bigl\{ 3, \dots, \lfloor \sqrt{2b} \rfloor \bigr\}$.

\m
We start with three inequalities that will be useful later on.

\begin{lemma} \label{le:claims}
\begin{itemize}
\item[{\rm (i)}]
For $k \geqslant 4$ we have
$\frac{2b+2k}{2b} \,\geqslant\, \left( \frac{2b+2k-2}{2b+k-2} \right)^2$.  

\s
\item[{\rm (ii)}]
$\sqrt{\frac{2b+2k}{2b}} \geqslant \frac{2b+2k}{2b+k}$.

\m
\item[{\rm (iii)}]
If $k^2 \leqslant 2b$, then 
$\frac{2b+k}{2b} \leqslant \frac{2b+2k}{2b+k-1}$.
\end{itemize}
\end{lemma}

\begin{proof} 
(i) is equivalent to
$$ 
\frac{(2b+2k)(2b+k-2)^2-2b(2b+2k-2)^2}{2b(2b+k-2)^2} \,\geqslant\, 0
$$
which holds true for $k \geqslant 4$ because the nominator of the left hand side
can be written as $2k \bigl( b(k-4)+(k-2)^2 \bigr)$.

\smallskip
(ii) follows from $(2b+k)^2-2b(2b+2k)=k^2$.

\smallskip
(iii) follows from 
$$
(2b+2k)(2b)-(2b+k-1)(2b+k) \,=\, 2b+k - k^2 
$$
since $2b+k-k^2 \geqslant k > 0$ by assumption.
\end{proof}


Except possibly for the right end point, that we can neglect, 
the weight expansion at $a \in R_b(k-1) = [2b+2k, 2b +2k+\frac{k^2}{2b}]$ is 
$$
\ww(a) \,=\, \left( 1^{\times 2b+2k},\, w_1^{\times \ell_1},\, 
                          w_2^{\times \ell_2},\, \dots \right) .
$$
Set $\lambda = \sqrt{\frac{a}{2b}}$.
We wish to transform the vector 
\begin{equation} \label{e:start6}
\bigl( (b+1) \la,\, b \la,\, \lambda,\, \ww (a) \bigr) 
\end{equation}
to a reduced vector by a sequence of Cremona moves.
Define the numbers
\begin{eqnarray*}
z_1 &:=& \la -1 ,  \\ 
y_1 &:=& (2b+k)\la-(2b+2k-1) , \\
z_2 &:=& y_1 - \la .
\end{eqnarray*}
Then $z_1, y_1 \geqslant 0$, and $z_2 \in [0,1]$.
Indeed, as we have seen in Lemma~\ref{le:ele}~(ii), $z_2=0$ at the left end point of $L_b(k-1)$,
and $z_2 \leqslant 1$
since $\lambda \leqslant \frac{2b+k}{2b}$ and by Lemma~\ref{le:claims}~(iii).

Applying one Cremona move to~\eqref{e:start6} we obtain
\begin{equation*} 
\bigl( (b+1)\la-1;\, b \la - 1,\, 1^{\times 2b+2k-1} \parallel 
z_1, \, w_1^{\times \ell_1},\, \ldots \bigr) .
\end{equation*}
Applying $b+k-1$ Cremona transforms with $\delta = \lambda-2$ and reordering we obtain
\begin{equation} \label{e:vector}
\bigl( y_1;\, 1 \parallel z_1^{\times 2b+2k-1},\, z_2,\, w_1^{\times \ell_1} ,\, \ldots \bigr) .
\end{equation}

\subsection{The case $z_1 \geqslant w_1$}
Assume that $z_1 \geqslant w_1$.

Assume first that $k \geqslant 4$, or that $k=3$ and $z_2 \geqslant z_1$.
If $z_2 \geqslant z_1$, then the vector~\eqref{e:vector} reorders to
\begin{equation*} 
\bigl( y_1;\, 1,\, z_2,\, z_1^{\times 2b+2k-1},\, w_1^{\times \ell_1} ,\, \ldots \bigr) .
\end{equation*}
This vector has defect $\delta = y_1-1-z_2-z_1 =0$ and hence is reduced.  
If $z_1 \geqslant z_2$, then the vector~\eqref{e:vector} reorders to the vector
\begin{equation} \label{e:vectorz2z1}
\bigl( y_1;\, 1,\, z_1^{\times 2b+2k-1} \parallel z_2,\, w_1^{\times \ell_1} ,\, \ldots \bigr) 
\end{equation}
which for $k \geqslant 4$ is reduced, since then $\delta = y_1-1-2z_1 = y_1-(2\lambda-1) \geqslant 0$ 
by Lemma~\ref{le:claims}~(i) and the fact that $\lambda \geqslant \sqrt{\frac{2b+2k}{2b}}$.  

Assume now that $k=3$ and $z_1 \geqslant z_2$.
If $\hat \delta := y_1-1-2 z_1 \geqslant 0$, the vector~\eqref{e:vectorz2z1} is reduced. 
Otherwise, we apply $b+2$ Cremona moves to obtain
\begin{equation} \label{e:vector3k}
\bigl( y_1+ (b+2) \hat \delta;\, 1+(b+2) \hat \delta, \,z_1,\, z_2^{\times 2b+5} \parallel
w_1^{\times \ell_1} ,\, \ldots \bigr) .
\end{equation}
The ordering  is right by the following claim, and the defect is $y_1 -1 -z_1-z_2 =0$,
whence this vector is reduced.

\begin{claim*} 
Assume that $k=3$. Then

\s
\begin{itemize}
\item[{\rm (i)}] 
$1+(b+2) \hat \delta \geqslant z_1$,   

\s
\item[{\rm (ii)}] 
If $z_1 \geqslant w_1$, then $z_2 \geqslant w_1$.
\end{itemize}
\end{claim*} 

\begin{proof}
Inequality (i) is equivalent to  
$$ 
(2b^2 +5b +1) \la \,\geqslant\, (2b^2 +8b+6) .
$$
It suffices to check this inequality for $\la = \sqrt{\frac{2b+6}{2b}}$,
where it is equivalent to $3b^2+10b+3 \geqslant 0$,
which holds true for all $b \geqslant 1$.

\medskip
For (ii), we know that $\lambda -1 \geqslant w_1$, i.e.,   
\begin{equation} \label{e:agel}
a \,\le\, \lambda + (2b+5) .
\end{equation}
Since $a = 2b\lambda^2$, this is equivalent to
\begin{equation} \label{e:lage}
\lambda \,\geqslant\, \frac{1+ \sqrt{1+8b (2b+5)}}{4b}.
\end{equation}
We wish to show that $z_2 -2 \geqslant w_1$, i.e.,
$a \leqslant 1+ (2b+2) \lambda$.
In view of~\eqref{e:agel}, this will hold if $\lambda + (2b+5) \leqslant1+ (2b+2)\lambda$,
i.e., 
\begin{equation} \label{e:2b4}
\frac{2b+4}{2b+1} \,\le\, \lambda.
\end{equation}
By~\eqref{e:lage}, this would follow from 
$$
\frac{2b+4}{2b+1} \,\le\, \frac{1+\sqrt{1+8b (2b+5)}}{4b} .
$$
Isolating the root and squaring, this becomes the true inequality
$\frac{72 \1 b}{(2b+1)^2} \geqslant 0$.
\end{proof}

\subsection{The case $w_1 \geqslant z_1$}
Assume now that 
\begin{equation} \label{eqn:assumption1}
w_1 \,\geqslant\, z_1 .
\end{equation}
The vector~\eqref{e:vector} in question is
\begin{equation} \label{e:vector'}
\bigl( y_1;\, 1 \parallel w_1^{\times \ell_1} ,\, z_1^{\times 2b+2k-1},\, 
                                                z_2,\,  \ldots \bigr) .
\end{equation}
Define
$$
z_3 \,:=\, y_1 - 1 - w_1.
$$
Note that $z_3 \geqslant 0$ on our interval, and $z_3 =0$ at the right end point  
$a = \frac{(2b+k)^2}{2b}$. 
The significance of $z_3$ and of the following lemma will become clear later. 

\begin{lemma} \label{le:assumption4}
If $z_3 \geqslant w_1$, then the vector~\eqref{e:vector'} is reduced.
\end{lemma}

\proof
For $\delta := y_1 - 1 - z_2 - w_1$ we have $z_2+\delta=z_3$ and $w_1 + \delta = z_1$.
Applying one Cremona move to
$$
\bigl( y_1;\, 1 \parallel z_2,\, w_1^{\times \ell_1},\, z_1^{\times 2b+2k-1},\, \ldots \bigr) 
$$
we thus obtain
\begin{equation} \label{e:vector.z3.6}
\bigl( y_1+ \delta;\, 1+ \delta,\, z_3 \parallel 
w_1^{\times \ell_1-1},\, z_1^{\times 2b+2k},\, \ldots \bigr) .
\end{equation}
The ordering is right because $z_3 \geqslant w_1 \geqslant z_1$
by assumption and by~\eqref{eqn:assumption1}.  
The defect of~\eqref{e:vector.z3.6} is thus
$y_1-1-z_3-(w_1 \mbox{ or } z_1 \mbox{ or } w_2) \geqslant y_1-1-z_3-w_1 =0$. 
\proofend

From now on we thus assume that
\begin{equation} \label{eqn:assumption4}
w_1 \,\geqslant\, z_3 .
\end{equation}

\begin{lemma} \label{lem:boundzi}
$z_2 \geqslant z_1, z_3$.
\end{lemma}

\begin{proof}
The inequality $z_2 \geqslant z_1$ translates to
$$
(2b+k-1) \la - (2b+2k-1) \,\geqslant\, \lambda - 1 ,
$$
or, equivalently,  
\begin{equation} \label{e:l2b} 
\lambda \,\geqslant\, \frac{2b+2k-2}{2b+k-2}.
\end{equation}
But we know that $\lambda \geqslant \sqrt{\frac{2b+2k}{2b}}$, whence in the case $k \geqslant 4$
the inequality~\eqref{e:l2b} follows from Lemma~\ref{le:claims}~(i).  
In the case $k = 3$, \eqref{e:l2b} is~\eqref{e:2b4}.   

The inequality $z_2 \geqslant z_3$ is $-\lambda \geqslant -1 - w_1$.  
This is equivalent to $\lambda -1 \leqslant w_1$, which follows from~\eqref{eqn:assumption1}. 
\end{proof}  

The rest of the proof of Theorem~\ref{t:dan} is divided into the
cases $\ell_1 = 2m$ even and $\ell_1 = 2m+1$ odd. 

\m \ni
{\bf Case I: $\ell_1 = 2m$ even.}  
We can assume by continuity that $\ell_1 >0$, so that $m \geqslant 1$.
By applying $m$ Cremona transforms to the vector~\eqref{e:vector'} 
with $\delta = y_1-1-2w_1$ we obtain 
\begin{equation} \label{eqn:theclass}
\bigl( y_2+y_1 - 1;\, y_2 \parallel
z_1^{\times 2b+2k-1},\ z_2,\, z_3^{\times \ell_1},\, w_2^{\times \ell_2},\, \ldots \bigr) 
\end{equation}
where $y_2 := 1 + m (y_1-1-2w_1)$.   
The ordering is right by the previous and the next lemma.


\begin{lemma} \label{lem:boundgamma3}
$y_2 \geqslant z_2, w_2$.
\end{lemma}

\begin{proof}
The inequality $y_2 \geqslant z_2$ is equivalent to
$$
1 + m (y_1 -1-2w_1) \,\geqslant\, y_1-\lambda .
$$
Since $\ell_1w_1 \leqslant1$ and $\lambda \geqslant 1$, it suffices to show that 
$(m-1)(y_1-1)\geqslant 0$.  
This follows since $y_1 \geqslant 1$, by Lemma~\ref{le:claims}~(ii).      

\s
The inequality $y_2 \geqslant w_2$ is equivalent to 
$$
1 + m(y_1-1-2w_1) \,\geqslant\, w_2 .
$$  
Since $\ell_1 w_1 \leqslant1$, it suffices to show that $m (y_1-1) \geqslant w_2$.  
For this, it suffices to show that $y_1-1 \geqslant w_1$, i.e.,
$a \leqslant\lambda (2b+k)$. 
This follows from the fact that $a \leqslant \frac{(2b+k)^2}{2b}$.       
\end{proof}

\begin{lemma} \label{le:w2z3}
If $z_3 \geqslant w_2$, then the vector~\eqref{eqn:theclass} is reduced.
\end{lemma}

\proof
Assume that $z_3 \geqslant w_2$.  
If $z_1 \geqslant z_3$, then~\eqref{eqn:theclass} is 
$$
\bigl( y_2 + y_1-1; \, y_2,\, z_2,\, z_1^{\times 2b+2k-1},\, z_3^{\times \ell_1},\, w_2^{\times \ell_2} \parallel \dots \bigr) ,
$$
which is reduced.
Hence we can assume that $z_3 \geqslant z_1$.  In this case, we apply one Cremona transform to
$$ 
(y_2+y_1 - 1;\, y_2,\, z_2,\, z_3^{\times \ell_1} \parallel z_1^{\times 2b+2k-1},\, w_2^{\times \ell_2},\, \ldots ) 
$$
with $\delta = z_1-z_3$ and obtain
$$ 
\bigl( y_2+y_1 -1 + \delta ;\, y_2+ \delta,\, z_2 + \delta,\, z_3^{\times \ell_1-1} \parallel z_3 + \delta,\, z_1^{\times 2b+2k-1},\, w_2^{\times \ell_2},\, \dots \bigr) 
$$
since $\ell_1 \geqslant 2$. First note that $z_3 + \delta = z_1 \geqslant 0$.  
To see that the ordering is right, we need to check that $z_2 + \delta \geqslant z_3$. 
This is equivalent to $y_1 - 1 \geqslant 2 z_3$, which is equivalent to $y_1-1-2w_1 \leqslant0$, 
which holds by~\eqref{eqn:assumption4}. 
Since the defect vanishes, this vector is reduced.
\proofend

From now on we thus assume that
\begin{equation} \label{eqn:w2z3}
w_2 \,\geqslant\, z_3.
\end{equation}  

\begin{lemma} \label{le:w2z1}
If $z_1 \geqslant w_2$, then the vector~\eqref{eqn:theclass} is reduced.
\end{lemma}

\proof
Assume that $z_1 \geqslant w_2$. Then the vector~\eqref{eqn:theclass} is
$$
\bigl( y_2+y_1-1 ;\, y_2,\, z_2,\, z_1^{\times 2b+2k-1},\, w_2^{\times \ell_2},\, z_3^{\times \ell_1} 
\parallel \dots \bigr) 
$$  
with defect $y_1 -1 -z_2 - z_1 =  0$.
\proofend

From now on we thus assume that 
\begin{equation} \label{eqn:w2z1}
w_2 \,\geqslant\, z_1.
\end{equation}  

By now, our vector is
\begin{eqnarray}
\mbox{If } w_2 \geqslant z_2: \quad  
\bigl( y_2+y_1-1 ;\, y_2,\, w_2^{\times\ell_2} \parallel  z_2 ,\,
                                                  z_1^{\times 2b+2k-1},\, z_3^{\times \ell_1}, \, w_3^{\times \ell_3},\, \dots \bigr)  \label{e:w2z2} \\
\mbox{If } z_2 \geqslant w_2: \quad 
\bigl( y_2+y_1-1 ;\, y_2,\, z_2,\, w_2^{\times\ell_2} \parallel 
                                                  z_1^{\times 2b+2k-1},\, z_3^{\times \ell_1}, \, w_3^{\times \ell_3},\, \dots \bigr)  \label{e:z2w2} 
\end{eqnarray}
   
\noindent
{\it Subcase $\ell_2 \geqslant 2$:}  
  
In case~\eqref{e:w2z2}
we have $\delta \geqslant y_1 - 1- w_1$, since $2w_2 \leqslant w_1$.  
Since $y_1 - 1- w_1= z_3 \geqslant 0,$ the vector is reduced.

In case~\eqref{e:z2w2} we have $\delta = z_1-w_2 < 0$. Applying one Cremona transform yields
\begin{equation} \label{eqn:theclass2} 
(y_2+y_1-1 + \delta;\, y_2 + \delta,\, z_2+z_1-w_2,\, w_2^{\times \ell_2-1} \parallel  
z_1^{\times 2b+2k},\, z_3^{\times \ell_1},\, w_3^{\times \ell_3},\, \ldots ) .
\end{equation}  
The ordering is right since $z_2+z_1 \geqslant 2w_2$. Indeed,  this is equivalent to $y-1 \geqslant 2 w_2$.  
Since $w_1 \geqslant 2w_2$, this follows from $y_1 - 1 \geqslant w_1$,
which holds because $y_1-1-w_1 = z_3 \geqslant 0$.   
The defect of~\eqref{eqn:theclass2} vanishes.  
 
\medskip \noindent
{\it Subcase $\ell_2 = 1$:}  
We distinguish again two cases.

Assume first that $w_3 \geqslant z_2$.
We are then in case~\eqref{e:w2z2},
and since $z_2 \geqslant z_1$ and $z_2 \geqslant z_3$, the vector at hand is
$$
\bigl( y_2+y_1-1 ;\, y_2,\, w_2,\, 
w_3^{\times \ell_3},\, z_2 \parallel \dots \bigr) .
$$
This vector is reduced, since $w_1 = w_2+w_3$ and hence 
$\delta = y_1-1-w_2-w_3 = z_3$.

\smallskip
Assume now that either $w_2 \geqslant z_2 \geqslant w_3$ or $z_2 \geqslant w_2$. 
Since also $z_2 \geqslant z_1$ and $z_2 \geqslant z_3$,
in both~\eqref{e:w2z2} and~\eqref{e:z2w2} we have $\delta =z_1-w_2$.
Further, $w_2 = w_1-w_3$ since $\ell_2=1$, and so $z_2 + \delta = z_2+z_1-w_2= w_3+z_3$.
Hence both vectors transform to
\begin{equation*} 
( y_2+y_1-1 + \delta;\, y_2 + \delta \parallel  w_3+z_3,\,  
z_1^{\times 2b+2k},\, z_3^{\times \ell_1},\, w_3^{\times \ell_3},\, \ldots ) .
\end{equation*}  
This vector is reduced after reordering:
If $w_3+z_3 \geqslant z_1$, then 
$$
\delta = z_1+z_2-w_3-z_3 - (z_1 \mbox{ or } z_3 \mbox{ or } w_3) = w_2 - (z_1 \mbox{ or } z_3 \mbox{ or } w_3) \geqslant 0
$$ 
by \eqref{eqn:w2z3} and \eqref{eqn:w2z1},
and if $z_1 \geqslant w_3+z_3$, then $\delta = z_1+z_2-2z_1 = z_2-z_1 \geqslant 0$.

\b \ni
{\bf Case II: $\ell_1 = 2m+1$ odd.}  
We start from the vector~\eqref{e:vector'}.
By applying $m \geqslant 0$ Cremona transforms with $\delta = y_1-1-2w_1$ we obtain  
$$ 
\bigl( \hat{y}_2+y_1-1;\, \hat{y}_2,\, z_3^{\times (\ell_1-1)},\, w_1,\, z_1^{\times 2b+2k-1},
                                                  \,z_2,\, w_2^{\times \ell_2},\, \ldots \bigr)
$$
where $\hat{y}_2 := 1+m(y_1-1-2w_1)$.  

Now apply another Cremona transform to the partially reordered vector 
$$ 
\bigl( \hat{y}_2+y_1-1;\, \hat{y}_2,\, w_1,\, z_2,\, z_1^{\times 2b+2k-1},\, 
                                              z_3^{\times (\ell_1-1)},\, w_2^{\times \ell_2},\, \ldots \bigr) .
$$
With $\delta = y_1-1-w_1-z_2 = z_1-w_1$ we obtain 
\begin{equation} \label{eqn:theclassodd}
\bigl( \hat{y}_2+y_1-1 + \delta;\, \hat{y}_2 + \delta \parallel z_1^{\times 2b+2k},\, 
                                              z_3^{\times \ell_1},\, w_2^{\times \ell_2},\, \ldots \bigr) 
\end{equation}
since $w_1 + \delta = z_1$ and $z_2 + \delta = z_3$.  
We are again assuming, by continuity, that $\ell_2 \geqslant 1$.  
The ordering is right in view of the following lemma.

\begin{lemma} \label{lem:hatgamma2large} 
\begin{itemize}
\item[{\rm (i)}]   $\hat{y}_2 + \delta \geqslant z_1$,
 
\s
\item[{\rm (ii)}] $\hat{y}_2 + \delta \geqslant z_3$,

\s 
\item[{\rm (iii)}]  $\hat{y}_2 + \delta \geqslant w_2$.
\end{itemize}
\end{lemma}

\begin{proof}  
Using $1 = \ell_1 w_1 + w_2$ and $y_1 - 1 = z_1 + z_2$ we compute
$$
\hat y_2 + \delta \,=\, (m+1)(z_1+z_2) -z_2 + w_2 .
$$
Assertions (i) and~(iii) follow at once.
Assertion (ii) follows at once for $m \geqslant 1$, and for $m=0$ also holds
since then $w_1+w_2 = 1 \geqslant z_2$.
\end{proof} 

We now show that the vector~\eqref{eqn:theclassodd} is reduced, 
or can be transformed in one step to a reduced vector. 
(We will only need to transform the vector in one case).  
In view of Lemma~\ref{lem:hatgamma2large}, we just have to consider the various possibilities 
for the orderings of $z_1, z_3, w_2$.  
Denote by~$\delta_*$ the defect of the reordering of~\eqref{eqn:theclassodd}.

\m
\ni
{\bf Case 1. $z_1 \geqslant z_3, w_2$.}
Then $\delta_* = y_1-1-2z_1=z_2-z_1 \geqslant 0$ by Lemma~\ref{lem:boundzi}. 

\m
\ni
{\bf Case 2. $z_3 \geqslant z_1, w_2$.}
Then $\delta_* \geqslant y_1-1-2z_3 = w_1-z_3 \geqslant 0$ by \eqref{eqn:assumption4}.

\m
\ni
{\bf Case 3. $w_2 \geqslant z_1, z_3$.}
Then the vector~\eqref{eqn:theclassodd} is
\begin{equation} \label{e:vector.last}
\bigl( \hat{y}_2+y_1 -1 + \delta;\, \hat{y}_2 + \delta,\, w_2^{\times \ell_2} \parallel 
z_1^{\times 2b+2k},\, z_3^{\times \ell_1},\, w_3^{\times \ell_3},\, \ldots \bigr) .
\end{equation}

\s
\ni
{\it Subcase $\ell_2 \geqslant 2$:}  
Then~\eqref{e:vector.last} is reduced if $y_1 -1 \geqslant 2w_2$.  
We know that $2w_2 \leqslant w_1$. Hence it suffices to show that $y_1 - 1 \geqslant w_1$, 
which follows from the fact that $z_3 \geqslant 0$.  

\m
\ni
{\it Subcase $\ell_2 =1$:}  
We distinguish three cases.

Assume first that $w_3 \geqslant z_1, z_3$. 
Then~\eqref{e:vector.last} is reduced, since 
$$
\delta_* \,=\, y_1-1-(w_2+w_3) \,=\, y_1-1-w_1 \,=\, z_3 .
$$

Assume next that $z_3 \geqslant z_1,w_3$.
Then~\eqref{e:vector.last} is reduced, since
$$
\delta_* \,=\, y_1-1-w_2-z_3 \,=\, w_1-w_2 .
$$

Assume finally that $z_1 \geqslant z_3, w_3$. Then the vector in question is
$$
\bigl( \hat{y}_2+y_1 -1 + \delta;\, \hat{y}_2 + \delta,\, w_2,\, z_1^{\times 2b+2k}
\parallel z_3^{\times \ell_1},\, w_3^{\times \ell_3},\, \ldots \bigr) .
$$
If $\hat \delta := y_1-1-w_2-z_1 = z_2-w_2 \geqslant 0$, this vector is reduced. 
Otherwise, we apply one Cremona transform and obtain
\begin{equation} \label{e:verylast}
\bigl( \hat{y}_2+y_1 -1 + \delta + \hat \delta;\, \hat{y}_2 + \delta + \hat \delta,\, w_2+\hat \delta,\, z_1+\hat \delta,\,
z_1^{\times 2b+2k-1} ,\, \ldots \bigr) .
\end{equation}
Note that $z_1+\hat \delta = y_1-1-w_2 \geqslant y_1-1-w_1 = z_3 \geqslant 0$ and that
$w_2+\hat \delta = z_2 \geqslant z_1$ by Lemma~\ref{lem:boundzi}.
Hence~\eqref{e:verylast} reorders to the vector 
$$
\bigl( \hat{y}_2+y_1-1 + \delta + \hat \delta;\, \hat{y}_2 + \delta + \hat \delta,\, z_2,\,
z_1^{\times 2b+2k-1} ,\, \ldots \bigr) 
$$
which is reduced, since its defect is $y_1-1-z_2-z_1 = 0$.
\proofend

The proof of Theorem~\ref{t:dan} is finally complete. 


\section{The interval $[v_b(1), 2b+4]$} \label{s:vb1.2b4}

Recall that for $b \in \NN_{\geqslant 2}$ we defined 
$v_b(1):=2b \left( \frac{2b+3}{2b+1} \right)^{2}$ and
$$
\alpha_{b} \,:=\, \frac{1}{b} \left( b^{2}+2b+\sqrt{\left( b^{2}+2b\right)^{2}-1} \right)
\in \left] v_b(1),2b+4 \right[.
$$

\begin{thm} \label{thm:[2b+3,2b+4]}
For every $b \in \NN_{\geqslant 2}$ we have
$$
c_{b}(a) \,=\, 
\left\{\begin{array} {cl}        
\sqrt{\frac{a}{2b}}    &    \mbox{if }\;  a \in \left[v_b(1),\alpha_{b}\right],\\ [0.2em]
\frac{ba+1}{2b(b+1)} &      \mbox{if }\;  a \in \left[\alpha_{b},2b+4\right].
\end{array}\right.
$$
In particular, $c_{b} \left( \alpha_{b}\right) = \sqrt{\frac{\alpha_b}{2b}}$
and $c_{b}(2b+4)=1+\frac{2b+1}{2b(b+1)}$.
\end{thm}

\proof
Let $a \in \left[v_b(1),2b+4\right]$ be a rational number. 
For $w_1(b) = v_b(1)-(2b+3)$ we compute $w_1' (b) = \frac{16}{(2b+1)^3}$.
Hence $w_1(b) \geqslant w_1(2) = \frac{21}{25} > \frac 56$ for $b \geqslant 2$, and so
$\ell_1 = 1$ and $\ell_2 \geqslant 5$.
The weight expansion of~$a$ thus has the form
\[
\ww (a) \,=\, \bigl( 1^{\times(2b+3)},w_{1},w_{2}^{\times\ell_{2}},\ldots,w_{N}^{\times \ell_N} \bigr).
\]
We wish to show that for $\lambda=c_b(a)$ as in the theorem, 
the vector $\left( (b+1)\lambda;\; b\lambda,\:\lambda,\: \ww(a) \right)$
can be reduced to a reduced vector. 

\subsection{The interval $[v_b(1), \alpha_b]$}  \label{ss:ine}

Assume that $a \in [v_b(1), \alpha_b]$. Then $\lambda = \sqrt{\frac{a}{2b}}$.
Define the numbers
\begin{eqnarray*}
z_{1} &:= & \lambda-1,\\
z_{2} &:= & (2b+1)\lambda-(2b+3),\\
z_{3} &:= & (2b+1)\lambda-(a-1),\\
z_{4} &:= & b\left(z_{3}-z_{1}\right)+w_{1},\\
z_{5} &:= & 2b(b+1)\lambda-(ba+1),\\
z_{6} &:= & b\left(2z_{5}+z_{1}-z_{4}-2z_{3}\right)+z_{4}.
\end{eqnarray*}
In the following, the symbol $\stackrel{\e}{=}$ means that an identity
is readily checked by expanding the relevant $z_i$
as polynomials of degree two in~$\lambda$ with coefficients polynomials in~$b$.
For instance, 
\begin{eqnarray}
z_{3} &=& 1+z_{2}-w_{1} \stackrel{\e}{=} z_{1}+z_{5}-z_{4}, \label{e:z3}\\
z_{6} &\stackrel{\e}{=}& b\left(2b(b+1)-1\right)\lambda-\left(b^{2}a-w_{1}\right) . \label{e:z6}
\end{eqnarray}

\noindent 
In this section, all newly created numbers will be one of $z_1, \dots, z_6$ or~$0$,
and we shall write down each $z_i$ of every vector. 
In other words, the dots~$\dots$ in any vector are either~$w_j$ or~$0$.

\subsubsection{Inequalities}

\begin{lemma} \label{le:prep7}
On the interval $[v_b(1),\alpha_b]$ the following inequalities hold true.

\s
\begin{itemize}
\item[\rm (i)] 
$b\lambda-1 \geqslant 1$ and $w_1 \geqslant z_1 \geqslant w_2$.

\s
\item[\rm (ii)] 
$w_1 \geqslant 1-z_1+z_2 \geqslant z_1 \geqslant z_2$.

\s
\item[\rm (iii)] 
$z_1 \geqslant z_3 \geqslant z_2, \1 w_2$.

\s
\item[\rm (iv)] 
$z_1 \geqslant z_5$. 
Moreover, $z_5 \geqslant z_3$ is equivalent to $z_4 \geqslant z_1$.

\s
\item[\rm (v)] 
$z_{4} \geqslant z_{3}$.

\s
\item[\rm (vi)] 
$z_6 \geqslant z_{2}, z_{5}, w_{2}$.

\s
\item[\rm (vii)]
If $b \geqslant 3$, then $z_1-z_4+2z_5-2w_2 \geqslant 0$.

\s
\item[\rm (viii)]
$z_i \geqslant 0$ for all $i \in \{1, \dots, 6\}$.
\end{itemize}
\end{lemma}

\proof
(i)
We have $b\lambda-1\geqslant b-1\geqslant1$. 
In order to prove $w_{1} \geqslant z_{1}$, we show that
the function $f_{b}(a):=w_{1}-z_{1}=a-(2b+2)-\sqrt{\frac{a}{2b}}$
is non-negative. Since $f_{b}'(a)=1-\frac{1}{4b}\sqrt{\frac{2b}{a}}>0$,
it suffices to see that $f_{b}\left(v_b(1)\right)=\frac{4b^{2}-5}{(2b+1)^{2}}\geqslant0$,
which holds true for $b \geqslant 2$.

To prove $z_{1}\geqslant w_{2}$, define the function $f_{b}(a):=z_{1}-w_{2}=\sqrt{\frac{a}{2b}}+a-(2b+5)$.
Since $f_{b}'(a)=\frac{1}{4b}\sqrt{\frac{2b}{a}}+1>0$, it suffices
to see that $f_{b}\left(v_b(1)\right)=\frac{4b-2}{(2b+1)^{2}}\geqslant 0$,
which holds true for $b \geqslant 2$.

\m
(ii)
We compute 
$$
1-z_1+z_2 \,=\, 2b(\lambda-1)-1 \,\geqslant\, \lambda-1 \,=\, z_1.
$$ 
This proves the second inequality,
and that the first inequality $w_1 \geqslant 1-z_1+z_2$ is equivalent to
$2b \la^2 - 2b \la -2 \geqslant 0$.
Since the left hand side is increasing for $\la \geqslant 1$, it suffices to check
this inequality at $\la (v_b(1)) = \frac{2b+3}{2b+1}$, where it becomes 
$\frac{4b-2}{(2b+1)^2} \geqslant 0$.

The third inequality $z_1 \geqslant z_2$ is equivalent to $\sqrt{2ab} \leqslant2b+2$. 
Squaring this leads to $a\leqslant2b+4+\frac{2}{b}$, which is verified for $a \leqslant \alpha_b < 2b+4$.

\m
(iii)
The inequality $z_1 \geqslant z_3$ is equivalent to $w_{1} \geqslant 1-z_{1}+z_{2}$, hence true.
The other two inequalities follow from $z_{3}=z_{2}+w_{2}$.

\s
(iv)
The inequality $z_1 \geqslant z_5$ is equivalent to $a\geqslant\frac{\left(2b^{2}+2b-1\right)^{2}}{2b^{3}}$.
This inequality is satisfied since $\frac{\left(2b^{2}+2b-1\right)^{2}}{2b^{3}}\leqslant v_b(1)$
is equivalent to $8b^{3}+12b^{2}-1\geqslant0$ which is true for $b\geqslant2$.

The inequality $z_5 \geqslant z_3$ is equivalent to $z_4 \geqslant z_1$
since $z_3 =z_1 +z_5-z_4$.

\s
(v)
Define the function $f_b(\la) := z_{4}-z_{3} \stackrel{\e}{=} 
\la \left(2b^{2}-2b-1\right)-(b-2)2b \la^2 -4$.
For $b=2$ we compute $f_2(\la) = 3 \la - 4 \geqslant f_2 \left( \la(v_2(1)) \right) = \frac{1}{5}>0$.
For $b\geqslant3$ we have
\[
f_b'(\la) \,=\, 2b^2 -2b-1-4b(b-2)\la \,\leqslant\,
-2b^2+6b-1 \,\leqslant\, -1 
\]
since $\la \geqslant 1$.
It thus suffices to show that $f_b(\la) >0$ at $\la = \sqrt{\frac{2b+4}{2b}}$, that is,
$$
\sqrt{\tfrac{2b+4}{2b}}\left(2b^{2}-2b-1\right) \,\geqslant\, 2b^2-4 .
$$
Squaring both sides leads to $4b^{2}-7b+2 \geqslant 0$ which is verified
for $b\geqslant3$.

\s 
(vi)
The first inequality means that the function
$$
f_b(a) \,=\, z_6-z_2 \,\stackrel{\e}{=}\, \bigl( 2b^3+2b^2-3b-1 \bigr) \1 \lambda + (1-b^2)a
$$
is non-negative for $a \in [v_b(1), \alpha_b]$.
Equivalently, 
$$
\tfrac{1}{\sqrt{2b}}  \bigl( 2b^3+2b^2-3b-1 \bigr) \,\geqslant\, \sqrt{a} \, (b^2-1) .
$$
It suffices to show this inequality for $a = 2b+4$, i.e., 
$$
\tfrac{1}{2b} \bigl( 2b^3+2b^2-3b-1 \bigr)^2 \,\geqslant\, (2b+4) (b^2-1)^2 .
$$
This is equivalent to $(b-1)^2 \geqslant 0$, which holds true.

\smallskip
We next show that the function
$$
f_b(a) \,=\, z_6-z_5 \,\stackrel{\e}{=}\, -2-2b + (2b^3-3b) \lambda + (1+b-b^2)a
$$
is non-negative for $a \in [v_b(1), \alpha_b]$.

If $b=2$, then $f_b(a) = -a+5\sqrt{a} -6 >0$ on $[2b+3,2b+4] =[7,8]$.

For $b \geqslant 3$ we compute that
$$
f_b'(a) \,=\, (2b^3-3b) \, \lambda'_b(a) + (1+b-b^2)
$$
is negative on $[ v_b(1), \alpha_b ]$, since $\lambda_b'(a) = \frac{1}{2 \sqrt{2ab}}$ is decreasing 
and $f_b'(2b) = \frac 14 +b-\frac{b^2}{2} < 0$ for $b \geqslant 3$.
It thus suffices to show that 
$$
f_b(2b+4) \,=\, 2 (1 + 2b - b^2 - b^3)  + (2b^3-3b) \sqrt{\tfrac{b+2}{b}} 
$$
is positive.
This is equivalent to $b^2+2b-4 \geqslant 0$, which holds true.

\smallskip
We finally show that the function
$$
f_b(a) \,=\, z_6-w_2 \,\stackrel{\e}{=}\, -7-4b  +  b (2b^2+2b-1) \, \lambda + (2-b^2)a
$$
is non-negative for $a \in [v_b(1), \alpha_b]$.

If $b=2$, then $f_b(a) = -2a+11\sqrt{a} -15 >0$ on $[2b+3,2b+4] =[7,8]$.

For $b \geqslant 3$ we compute that
$$
f_b'(a) \,=\,  b (2b^2+2b-1) \, \lambda'_b(a) + 2-b^2
$$
is negative on $[ v_b(1), \alpha_b ]$, since 
$f_b'(2b) = \frac 14 (2b^2+2b-1) + 2 - b^2 < 0$ for $b \geqslant 3$.
It thus suffices to show that 
$$
f_b(2b+4) \,=\, 1 - 2  b^2 (b+2)  + b (2b^2+2b-1) \sqrt{\tfrac{b+2}{b}}  
$$
is positive.
This is equivalent to $b^2+2b-1 \geqslant 0$, which holds true.

\s 
(vii)
We compute  
$$
\delta_b(a) := z_1-z_4+2z_5-2w_2 \,\stackrel{\e}{=}\, 
-8-4b + (1+4b+2b^2) \, \lambda + (1-b) \,a 
$$
and
$$
\delta_b'(a) \,=\, 1-b + \frac{2b^2+4b+1}{2 \sqrt{2}\sqrt{ab}} .
$$

Assume first that $b=3$. Then $\delta_3 (a) = -20+\frac{31}{\sqrt{6}} \sqrt{a} - 2a$.
Since $\delta_3'(a) = -2+\frac{31}{2 \sqrt{6}\sqrt{a}}$ is positive for $a \in [2b+3,2b+4] = [9,10]$,
and since $\delta_3( v_3(1) ) = \frac{1}{49} >0$, the function $\delta_3(a)$ is positive on~$[v_3(1), \alpha_3]$.

Assume now that $b=4$.
Then $\delta_4(a) = -24 + \frac{49 \sqrt{a}}{2 \sqrt{2}} -3a$.
Hence $\delta_4(2b) = \delta_4(8)=1$ and $\delta_4(2b+4) = \delta_4(12) = -60+49 \sqrt{\frac 32} >0$, 
and so $\delta_4(a)>0$ for all $a \in [2b,2b+4]$.

Assume finally that $b \geqslant 5$.
Then $\delta_b'(a) < 0$ for $a \in [2b,2b+4]$.
Indeed, $\delta_b'(a)$ is decreasing and $\delta_b'(2b) = 1-b + \frac{2b^2+4b+1}{4b} <0$. 
We are left with showing that 
$$
\delta_b(2b+4) \,=\, - (4+6b+2b^2) + (1 + 4b + 2b^2) \, 
\sqrt{\tfrac{b+2}{b}} 
$$
is positive, which is true since equivalent to $\frac{b+2}{b} >0$.

\s
(viii)
We show that $z_2, z_5 \geqslant 0$.
The other inequalities then follow from the previous items.
The inequality $z_2 \geqslant 0$ is equivalent to $\lambda\geqslant\frac{2b+3}{2b+1}$,
which holds true.
Moreover, $z_{5} \geqslant 0$ is equivalent to 
\begin{equation} \label{e:laz5}
\lambda \, \geqslant \, \frac{ba+1}{2b(b+1)} ,
\end{equation}
which means that the line $a \mapsto \frac{ba+1}{2b(b+1)}$ of the affine step 
is below the volume constraint $\sqrt{\frac{a}{2b}}$.
This holds true on $[2b, \alpha_b]$, since $\sqrt{\frac{a}{2b}}$ is convex 
and since~\eqref{e:laz5} is an equality at~$\alpha_b$ and a strict inequality
at~$2b$.
\proofend


\subsubsection{Reductions}

Reducing the vector
$\bigl( (b+1)\lambda;\; b\lambda,\:\lambda,\:1^{\times(2b+3)},\: w_{1},\: w_{2}^{\times\ell_{2}} 
                 \parallel \ldots \bigr)$
with $\delta=-1$ yields
\begin{equation*} 
\bigl( (b+1)\lambda-1;\, b\lambda-1,\:\underset{=\,z_{1}}{\underbrace{\lambda-1}},\:0,\:1^{\times(2b+2)},\: w_{1},\: w_{2}^{\times\ell_{2}},\ldots \bigr).
\end{equation*}
By Lemma~\ref{le:prep7}~(i) this vector reorders to
\[
\bigl( (b+1)\lambda-1;\; b\lambda-1,\:1^{\times(2b+2)},\: w_{1},\: z_{1},\: w_{2}^{\times\ell_{2}}\parallel\ldots,\:0
\bigr).
\]
Applying $b$ Cremona transforms with $\delta = \lambda-2$ and regrouping the produced $z_{1}$'s, we get
\begin{equation*} 
\Bigl( \underset{=\,z_{2}+2}{\underbrace{(2b+1)\lambda-(2b+1)}};\:\underset{=\,1-z_{1}+z_{2}}{\underbrace{2b\lambda-(2b+1)}},\:1^{\times2},\: w_{1},\: z_{1}^{\times (2b+1)},\: w_{2}^{\times \ell_{2}},\ldots
\Bigr).
\end{equation*}
By Lemma~\ref{le:prep7}~(ii), this vector reorders to
\[
\bigl( z_{2}+2;\:1^{\times2},\: w_{1},\:1-z_{1}+z_{2},\: z_{1}^{\times(2b+1)},\: w_{2}^{\times\ell_{2}}\parallel\ldots
\bigr).
\]
Applying one Cremona transform with $\delta =z_{2}-w_{1}$ yields the vector
\begin{equation*} 
\bigl( 2z_{2}+2-w_{1};\;
\underset{=\,z_{3} \mbox{ by } \eqref{e:z3}}{\underbrace{\left(1+z_{2}-w_{1}\right)}}^{\times2},\: 
z_{2},\:1-z_{1}+z_{2},\: z_{1}^{\times(2b+1)},\: w_{2}^{\times\ell_{2}},\ldots
\bigr) ,
\end{equation*}
which by Lemma~\ref{le:prep7}~(iii) reorders to
\[
\bigl( 2z_{2}+2-w_{1};\;1-z_{1}+z_{2},\: z_{1}^{\times(2b+1)},\: z_{3}^{\times2}\parallel\: z_{2},\: w_{2}^{\times\ell_{2}},\ldots \bigr).
\]
Applying $b-1$ Cremona transforms with $\delta = z_3-z_1$
and regrouping the produced $z_{3}$'s, we get
\begin{equation} \label{eq:vector step 4}
\Bigl(
\underbrace{(b-1)\left(z_{3}-z_{1}\right)+2z_{2}+2-w_{1}}_{\stackrel{\e}{=}\,2z_{1}+z_{5}};\:\underbrace{b\left(z_{3}-z_{1}\right)+w_{1}}_{=\,z_{4}},\: z_{1}^{\times3},\: z_{3}^{\times2b},\: z_{2},\: w_{2}^{\times\ell_{2}},\ldots
\Bigr).
\end{equation}

\noindent
We now distinguish the cases $z_{4}\geqslant z_{1}$ and $z_{1}\geqslant z_{4}$.

\medskip
\noindent \textbf{Case 1:} $z_{4}\geqslant z_{1}$. The ordered vector is then
\[
\bigl( 2z_{1}+z_{5};\: z_{4},\: z_{1}^{\times3},\: z_{3}^{\times2b}\parallel\: z_{2},\: w_{2}^{\times\ell_{2}},\ldots
\bigr).
\]
One more Cremona transform with $\delta = z_5-z_4$ yields
\begin{equation*} 
\bigl(
2\left(z_{1}+z_{5}\right)-z_{4};\: z_{5},\: 
\bigl( \underbrace{z_{1}+z_{5}-z_{4}}_{=\,z_{3} \mbox{ by } \eqref{e:z3}} \bigr)^{\times2},\: 
z_{1},\: z_{3}^{\times2b},\: z_{2},\: w_{2}^{\times\ell_{2}},\ldots
\bigr) ,
\end{equation*}
which by Lemma~\ref{le:prep7}~(iv) reorders to
\begin{equation*} 
\bigl( 
2 \left( z_1+z_5 \right)-z_4;\: z_1,\: z_5,\: z_3^{\times(2b+2)}\parallel\: z_{2},\: w_{2}^{\times\ell_{2}},\ldots
\bigr).
\end{equation*}
We already know that all entries of this vector are non-negative,
and its defect is $\delta = z_1 + z_5 - z_4 - z_3 = 0$.
Hence this vector is reduced.

\b
\noindent 
\textbf{Case 2}: $z_1 \geqslant z_4$.
Reorder the vector \eqref{eq:vector step 4} as
\[
\bigl( 2z_{1}+z_{5};\: z_{1}^{\times3},\: z_{4},\: z_{3}^{\times2b}\parallel\: z_{2},\: w_{2}^{\times\ell_{2}},\ldots
\bigr).
\]
Recall from Lemma~\ref{le:prep7}~(iv) that $z_1 \geqslant z_5$.
Apply one Cremona transform with $\delta = z_5-z_1$ to
obtain
\begin{equation*} 
\bigl(
2z_{5}+z_{1};\: z_{5}^{\times3},\: z_{4},\: z_{3}^{\times2b},\: z_{2},\: w_{2}^{\times\ell_{2}},\ldots
\bigr).
\end{equation*}
Since $z_{3}\geqslant z_{5}$ by Lemma~\ref{le:prep7}~(iv), this vector reorders to
\[
\bigl( 2z_{5}+z_{1};\: z_{4},\: z_{3}^{\times2b}\parallel\: z_{2},\: z_{5}^{\times3},\: w_{2}^{\times\ell_{2}},\ldots
\bigr).
\]
Applying $b$ Cremona transforms with $\delta = 2z_5 + z_1 - z_4 - 2z_3$
and regrouping the produced $z_5$'s, we obtain the vector
\[
\Bigl( \underbrace{(b+1)\left(2z_{5}+z_{1}\right)-b\left(z_{4}+2z_{3}\right)}_{=:\,\mu};\:\underset{=\,z_{6}}{\underbrace{b\left(2z_{5}+z_{1}-z_{4}-2z_{3}\right)+z_{4}}},\: z_{2},\: z_{5}^{\times(2b+3)},\: w_{2}^{\times\ell_{2}},\ldots
\Bigr) ,
\]
which by Lemma~\ref{le:prep7}~(vi) reorders to
\begin{equation} \label{eq:vector step 6}
\bigl( \mu;\, z_{6} \parallel z_{2}, z_{5}^{\times(2b+3)},w_{2}^{\times\ell_{2}},\ldots
\bigr).
\end{equation}
Notice that this vector does not contain $z_1,z_3,z_4$.

\begin{proposition} \label{p:b23}
Assume that $a \leqslant \alpha_{b}$ and $z_{1} \geqslant z_{4}$.
If $b=2$ also assume that $w_2 \leqslant\max \{ z_2,z_5\}$.
Then the vector~\eqref{eq:vector step 6} is reduced.
\end{proposition}

\begin{proof}
We already know that all entries of~\eqref{eq:vector step 6} are non-negative.
%
Using~\eqref{e:z3} we compute 
\begin{equation} \label{e:d2eq}
\mu-z_6 \,=\, z_{1} - z_{4} + 2z_{5} \,=\, z_3+z_5 .
\end{equation}

\smallskip \noindent
{\it Subcase 1:} $z_5 \geqslant w_2$.
Then $\delta = \mu-z_6-z_5- (z_2 \mbox{ or } z_5) \,=\, z_3 -  (z_2 \mbox{ or } z_5) \,\geqslant\, 0$
where in the last step we have used (iii) and~(iv) of Lemma~\ref{le:prep7}.

\medskip \noindent
{\it Subcase 2:} $z_2 \geqslant w_2 \geqslant z_5$.
Then 
$$
\delta \,=\, \mu-z_6- (z_2+w_2) \,=\, z_3 + z_5 -z_3 \,=\, z_5 \,\geqslant\, 0.
$$

\noindent
{\it Subcase 3:} $w_2 \geqslant z_2,z_5$.
This is the case where we assume that $b \geqslant 3$.
Recall that $\ell_2 \geqslant 2$. Hence
$$
\delta_b(a) \,=\, \mu -z_6 -2w_2 \,=\, z_1-z_4+2z_5-2w_2  
$$
is non-negative by Lemma~\ref{le:prep7}~(vii).
\end{proof}

In view of Proposition~\ref{p:b23} we can assume that $b=2$ and that
$w_2 \geqslant \max \{ z_2,z_5\}$. The vector at hand then is
\begin{equation} \label{eq:vector:a0b=2}
\bigl( \mu;\, z_{6}, w_{2}^{\times\ell_{2}}  \parallel z_{2}, z_{5}^{\times(2b+3)},\ldots
\bigr).
\end{equation}
We set $z_7 := z_2+z_5$ and compute
\begin{equation*}
\delta \,=\, \mu-z_6 -2w_2 \stackrel{\eqref{e:d2eq}}{=} z_3+z_5-2w_2
\stackrel{\eqref{e:z3}}{=} 1+z_2-w_1+z_5-2w_2
\,=\, z_7-w_2 .
\end{equation*}
If $\delta \geqslant 0$ we are done. So assume that $\delta = z_7 -w_2 < 0$,
and set $m := \left \lfloor \frac{\ell_{2}}{2}\right\rfloor$ and $\hat \mu := \mu+m \delta$, $\hat z_6 := z_6+m \delta$. 
Applying $m$ Cremona transforms and swapping the position of $w_2$ and $z_7^{\times 2m}$
in case that $\ell_2$ is odd, we obtain
\begin{eqnarray}
\bigl(
\hat \mu; \, \hat z_6,\: z_{7}^{\times2m},\: z_{2},\: z_{5}^{\times(2b+3)},\: 
w_3^{\times \ell_3}, \ldots
\bigr)  && \mbox{ if $\ell_{2}=2m$},  \label{e:z7even}
\\
\bigl(
\hat \mu; \, \hat z_6,\: w_{2},\: z_{7}^{\times2 m},\: z_{2},\: z_{5}^{\times(2b+3)},\:
w_3^{\times \ell_3}, \ldots
\bigr) && \mbox{ if $\ell_{2}=2m+1$} . \label{e:z7odd}
\end{eqnarray}

\begin{proposition}
After reordering, the vector~\eqref{e:z7even} is reduced.
After reordering, the vector~\eqref{e:z7odd} is reduced if $z_7 \geqslant w_3$,
and transforms to a reduced vector by one Cremona move if $w_3 > z_7$.
\end{proposition}  

\proof
We first show the inequalities
\begin{equation} \label{e:ineq}
\hat z_6 \,\geqslant\, w_2 \,\geqslant\, z_7 \,\geqslant\, z_2, z_5 .
\end{equation}
Then also $\hat z_6, z_7 \geqslant 0$.
We have $w_2 - z_7 = -\delta >0$ and $z_7 = z_2+z_5 \geqslant z_2,z_5$.
We are thus left with proving $\hat z_6 \geqslant w_2$.
For $m \in \NN$ we compute
\begin{eqnarray*}
f_m(a) &:=& \hat z_6 -w_2 \,=\, z_6+m z_7 -(m+1)w_2  \\
&=& 
-(m+2)a + \left( \tfrac{17}{2}m +11 \right) \sqrt{a} - (16m+15) .
\end{eqnarray*}
Then $f_m'(a) = -(m+2) + \frac{\tfrac{17}{2}m +11}{2\sqrt a} >0$ for all $m \in \NN$ and $a \in [2b+3,2b+4] = [7,8]$, 
since this holds true for $a=8$.
Recall that $\ell_2 \geqslant 5$. Since $\ell_2 = \lfloor \frac{w_1}{w_2} \rfloor = \lfloor \frac{-7+a}{8-a} \rfloor$
and $\ell_2 (\alpha_2) = 30$, we can assume that $2 \leqslant m \leqslant 15$.
If the multiplicity of~$w_{2}$ is $\ell_{2}$, then $w_{1} \in \left[ \frac{\ell_{2}}{\ell_{2}+1};\frac{\ell_{2}+1}{\ell_{2}+2} \right[$.
Thus $\hat z_6 -w_2$ is given by~$f_m$ for $a \in \left[7+ \frac{2m}{2m+1}, 7+ \frac{2m+2}{2m+3} \right[ \cap [v_2(1),\alpha_2]$.
Since each $f_m$ is increasing on $[7,8]$, it now suffices to check that 
$f_2( v_2(1) ) = f_2 ((\frac{14}{5})^2) = \frac{1}{25} >0$ and that 
$f_m(7+\frac{2m}{2m+1}) \geqslant 0$ for $m \in \{3, \dots, 15\}$, which is readily checked
(for instance by noticing that $m \mapsto f_m(7+\frac{2m}{2m+1})$ is increasing).

\bigskip
\noindent
{\it Case 1:} $z_7 \geqslant w_3$.
The part $(\mu;\,a_1,a_2,a_3)$ of the ordered vectors is then as in~\eqref{e:z7even} 
and~\eqref{e:z7odd}. Therefore,
$\hat \delta = \mu-z_6 - 2z_7 = \delta -2(z_7-w_2) = - \delta > 0$ if $\ell_{2}$ is even, and 
$\hat \delta = \mu-z_6 - w_2 - z_7 = \delta -(z_7-w_2) = 0$ if $\ell_2$ is odd.
Hence the vectors \eqref{e:z7even} and \eqref{e:z7odd} are reduced.

\m
\noindent
{\it Case 2:} $w_3 > z_7$.
In this case, the vectors at hand are
\begin{eqnarray}
\bigl(
\hat \mu; \, \hat z_6,\: w_3^{\times \ell_3} \parallel z_{7}^{\times2m},\: 
                  w_4^{\times \ell_4},\: z_{2},\: z_{5}^{\times(2b+3)}, 
\ldots
\bigr)  && \mbox{ if $\ell_{2}=2m$},  \label{e:z7evenhat}
\\
\bigl(
\hat \mu; \, \hat z_6,\: w_{2},\: w_3^{\times \ell_3} \parallel z_{7}^{\times2m},\:  
                  w_4^{\times \ell_4},\: z_{2},\: z_{5}^{\times(2b+3)}, 
\ldots
\bigr) && \mbox{ if $\ell_{2}=2m+1$} . \label{e:z7oddhat}
\end{eqnarray}

Assume first that $\ell_2$ is even.
If $\ell_3 = 1$, then~\eqref{e:ineq} shows that
$$
\hat \delta = \mu-z_6 - w_3 - (z_7 \mbox{ or } w_4) = 
w_2 + z_7 -w_3 - (z_7 \mbox{ or } w_4) = (w_2-w_3 \mbox{ or } z_7) \geqslant 0 .
$$
If $\ell_3 \geqslant 2$, then 
$\hat \delta = \mu-z_6 - 2w_3 = w_2 + z_7 -2w_3 \geqslant z_7 \geqslant 0$.

Assume now that $\ell_2$ is odd.
Then $\hat \delta = \mu-z_6-w_2-w_3 = z_7-w_3 < 0$.
Applying one more Cremona move to the vector~\eqref{e:z7oddhat} yields
$$
\bigl(
\hat \mu + \hat \delta; \, \hat z_6+\hat \delta,\: w_{2}+\hat \delta,\: w_3^{\times \ell_3-1} \parallel z_7^{\times 2m+1},\:  
                  w_4^{\times \ell_4},\: z_{2},\: z_{5}^{\times(2b+3)} ,
\ldots
\bigr)
$$
The ordering is right because if $\ell_3=1$, then $w_2+\hat \delta = w_2+z_7-w_3 = z_7+w_4$,
and if $\ell_3 \geqslant 2$, then $w_2+\hat \delta = w_2+z_7-w_3 \geqslant w_3$.

If $\ell_3 = 1$, then the defect is now
$\tilde \delta = \mu-z_6 -w_2-\hat \delta - (z_7 \mbox{ or } w_4) = w_3 - (z_7 \mbox{ or } w_4) >0$,
and if $\ell_3 \geqslant 2$, then $\tilde \delta = w_3-w_3 =0$.

\s
This completes the proof of Theorem~\ref{thm:[2b+3,2b+4]} for $a \leqslant \alpha_b$.

\subsection{The interval $[\alpha_b, 2b+4]$} \label{ss:ab4}
It turns out that the reduction process for $a \in [\alpha_b, 2b+4]$
is the same as for $a \in [v_b(1),\alpha_b]$ in Case~2. 
Set $\lambda=\frac{ba+1}{2b(b+1)}$ and define $z_1, \dots, z_6$ as in~\S~\ref{ss:ine}.
Applying the same Cremona moves (i.e., the same sequence of Cremona transforms 
and reorderings) as in Case~2, we obtain the vector~\eqref{eq:vector step 6}, 
namely 
\begin{equation} \label{eq:vectoraffine} 
\bigl( \mu;\, z_6 \parallel z_{2},z_{5}^{\times(2b+3)},w_{2}^{\times\ell_{2}},\ldots
\bigr) .
\end{equation}
It suffices to prove the following statement.
\begin{proposition}
If $a\geqslant\alpha_{b}$, then the vector~\eqref{eq:vectoraffine}
is reduced.
\end{proposition}

\begin{proof}
The identity $\lambda = \frac{ba+1}{2b(b+1)}$ is equivalent to $z_{5}=0$.
We now show that $z_{6},z_{2} \geqslant w_2$, implying $z_{6},z_{2} \geqslant 0$.
Using~\eqref{e:z6} we find that the inequality $z_{6} \geqslant w_{2}$ is equivalent to the inequality
\[
w_1 \,\geqslant\, \frac{3b+3}{3b+4}
\]
which is satisfied since $\frac{3b+3}{3b+4}\leqslant\alpha_{b}-(2b+3)$
for all $b \geqslant \frac{2}{3}\left(-1+\sqrt{7}\right)$. 
The inequality $z_{2}\geqslant w_{2}$ is equivalent to the inequality
\[
w_{1} \geqslant\frac{4b^{2}+3b-1}{4b^{2}+3b}
\]
which is satisfied since $\frac{4b^{2}+3b-1}{4b^{2}+3b} \leqslant \alpha_{b}-(2b+3)$
for all $b\geqslant\frac{5}{4}$. 

The ordered vector is thus 
\[
\bigl( \mu;z_{6},z_{2},w_{2}^{\times\ell_{2}},\ldots,0^{\times(2b+3)} \bigr).
\]
(The inequality $z_{6}\geqslant z_{2}$ holds true, but there is no need to prove it).
Using again $\mu-z_6 = z_1-z_4+2z_5$ and $z_1+z_5-z_4= 1+z_2-w_1$ from~\eqref{e:z3}
we find, since $z_5=0$,
$$
\delta \,=\, (\mu-z_{6}) - (z_{2}+w_{2}) \,=\, (z_1-z_4) - (z_2+1-w_1) \,=\, 0 .
$$ 
Hence the vector~\eqref{eq:vectoraffine} is reduced.
\end{proof}


\section{The interval $[2b+4, u_b(2)]$ for $b \geqslant 3$}  \label{s:betagamma}

Recall that $\gamma_b := u_b(2) = \frac{(2b+2)^2}{2b} = 2b+4+\frac 2b$ and that 
$$
\beta_{b} \,:=\, 
\frac{\left(2b^{2}+4b+1\right)^{2}}{2b(b+1)^{2}} \,=\, 2b+4 + \frac{1}{2b(b+1)^2}
\;\in\; \left] 2b+4, \gamma_b \right[ .
$$
Throughout this section we assume that $b \geqslant 3$.

\begin{theorem}
For $b \geqslant 3$ we have 
$$
c_{b}(a) \,=\,
\left\{\begin{array} {cl} 
1+\frac{2b+1}{2b(b+1)} & \textrm{if }\; a\in\left[2b+4,\beta_{b}\right], \\ [0.2em]
\sqrt{\frac{a}{2b}}    & \textrm{if }\; a\in\left[\beta_{b},\gamma_b\right].
\end{array} \right.
$$
\end{theorem}
    
\proof
In view of Theorem~\ref{thm:[2b+3,2b+4]}
it suffices to prove that $c_{b}(a)=\sqrt{\frac{a}{2b}}$ on $\left[\beta_{b},\gamma_b\right]$.
Let $a \in \left[ \beta_{b},\gamma_b \right]$ be a rational number with weight expansion
\[
\ww (a) \,=\, 
\bigl(
1^{\times(2b+4)},w_{1}^{\times\ell_{1}},w_{2}^{\times\ell_{2}},\ldots,w_{n}^{\times\ell_{n}} 
\bigr).
\]

\subsection{Inequalities}

Set $\lambda=\sqrt{\frac{a}{2b}}$.
We wish to show that the vector $\bigl( (b+1)\lambda;\; b\lambda,\:\lambda,\: \ww (a) \bigr)$
can be reduced to a reduced vector.
Notice that
$$
\lambda (\beta_b) = 1+\frac{2b+1}{2b (b+1)}, \qquad
\lambda (\gamma_b) = 1+\frac 1b .
$$
Define the numbers
\begin{eqnarray*}
z_1 &:= & \lambda-1,\\
z_2 &:= & (2b+1)\lambda-(2b+3),\\
z_3 &:= & 1+b(z_2-z_1),\\
z_4 &:= & 1+(b+1)(z_2-z_1),\\
z_5 &:= & z_{1}+z_{2}-w_{1}
\end{eqnarray*}
and $m=\left\lfloor \frac{\ell_1}{2}\right\rfloor$ 
where $\ell_1 = \left\lfloor \frac{1}{w_1}\right\rfloor$. 
 
\begin{lemma} \label{le:inequalities}
On the interval $[\beta_b,\gamma_b]$ the following inequalities hold true.

\begin{itemize}
\item[\rm (i)] 
$1-z_1+z_2 \geqslant z_1 \geqslant z_2 \geqslant 0$,

\s
\item[\rm (ii)]  
$1-z_1+z_2 \geqslant w_1$, 

\s
\item[\rm (iii)]  
$z_3 \geqslant z_2,z_4,w_1$ and $z_4 \geqslant 0$,

\s
\item[\rm (iv)]
$z_3 + b (z_4-z_2) \geqslant z_2$,

\s
\item[\rm (v)]
$z_2 + z_4- w_1 \geqslant w_1$,

\s
\item[\rm (vi)]
$2 z_2 \geqslant w_1$ and $z_2 \geqslant w_3$.

\s
\item[\rm (vii)]
$1-z_1+z_2+m (z_1+z_2-2w_1) \geqslant w_1$.
\end{itemize}

\m \ni
In particular, $z_i \geqslant 0$ for all~$i$.
\end{lemma}

\proof
(i) The inequality $z_1 \geqslant z_2$ was already shown in the proof of Lemma~\ref{le:prep7}~(ii).

The inequality $z_2 \geqslant 0$ is equivalent to $(2b+1) \lambda \geqslant 2b+3$.
Since $\lambda$ is increasing, it suffices to verify this in $a = \beta_b$,
that is, that
$$
(2b+1) \left( 1+ \tfrac{2b+1}{2b (b+1)} \right) \,\geqslant\, 2b+3 ,
$$
or, equivalently, $(2b+1)^2 \geqslant 4b(b+1)$, which holds true.

The inequality $1-z_1+z_2 \geqslant z_1$ is equivalent to $(2b-1) \lambda \geqslant 2b$.
It suffices to verify this in $a = \beta_b$, that is, that
$$
(2b-1) \left( 1+ \tfrac{2b+1}{2b (b+1)} \right) \,\geqslant\, 2b ,
$$
or, equivalently, $2b^2 \geqslant 2b+1$, which holds true.

\s
(ii) is equivalent to $a-3 \leqslant 2b \lambda$.
Since the slope of $2b \lambda = \sqrt{2ba}$ is $\sqrt{\frac{b}{2a}} < 1$,
it suffices to check this inequality at $a=\gamma_b$, i.e., 
that $2b+2 \geqslant a-3$, which holds true.

\s
(iii)
$z_3 \geqslant z_2$ is equivalent to $(2b^2-2b-1) \lambda \geqslant 2b^2-4$.
It suffices to verify this in $a = \beta_b$, that is, that
$$
(2b^2-2b-1) \left( 1+ \tfrac{2b+1}{2b (b+1)} \right) \,\geqslant\, 2b^2-4 ,
$$
or, equivalently, $2b \geqslant 1$, which holds true.

$z_3 \geqslant z_4$ follows from $z_1 \geqslant z_2$.

$z_3 \geqslant w_1$ is equivalent to $2b^2 \lambda \geqslant 2b^2 +a -5$ or, using $a = 2b \lambda^2$,
to 
$$
f_b(\la) \,:=\, -2b \lambda^2 +2b^2 \la -2b^2 +5 \,\geqslant\, 0 .
$$
Since
$b \geqslant 3$, the derivative $f_b'(\la) = 2b(b-2\la)$ is positive, 
and $f_b(\la (\beta_b)) = \frac{2b^2+2b-1}{2b(b+1)^2} >0$.

$z_4 \geqslant 0$ is equivalent to $2b(b+1)\lambda \geqslant 2b^2+4b+1$,
which holds true, since this is an equality at $a=\beta_b$.

\s 
(iv) is equivalent to 
$(2b^2+4b+1) \lambda \geqslant 2 (b^2+3b+2)$.
At $a = \beta_b$, this inequality is equivalent to
$$
(2b^2+4b+1) (2b+1) \,\geqslant\, 2b (b+1) (2b+3)
$$
which in turn simplifies to $1 \geqslant 0$.

\s
(v) is equivalent to 
$(2b^2+4b+1)\lambda \geqslant 2 (a +b^2+b-2)$, or, using $a = 2b\lambda^2$, to
\begin{equation} \label{e:fbl}
f_b(\lambda) \,:=\, 4b \lambda^2 - (2b^2+4b+1) \lambda + 2(b^2+b-2) \,\le\, 0
\end{equation}
on $[\beta_b,\gamma_b]$. Its derivative is $f_b'(\lambda) = 8b \lambda -(2b^2+4b+1)$.

Assume first that $b=3$. Then $f_b'(\lambda) = 24 \lambda -31 \geqslant 0$ since
this holds true in $\lambda (\beta_b) = \frac{31}{24}$.
Hence~\eqref{e:fbl} follows from $f_b ( \lambda(\gamma_b)) = f_3 (\frac 43) =0$.

Assume now that $b \geqslant 4$. Then $f_b'(\lambda) \leqslant 0$ since $f_b'(\la(\gamma_b)) = 8(b+1)-(2b^2 +4b+1) \leqslant 0$.
Hence~\eqref{e:fbl} follows from $f_b (\lambda (\beta_b)) = - \frac{b-1}{2b (b+1)^2} \leqslant 0$.

\s
(vi) is equivalent to 
\begin{equation} \label{e:fbl2}
f_b(\lambda) \,:=\, b \lambda^2 - (2b+1) \lambda + (b+1) \,\leqslant\, 0
\end{equation}
on $[\beta_b,\gamma_b]$.
Since $f_b'(\lambda) = 2b \lambda - (2b+1) \geqslant 2b \lambda (\beta_b) -(2b+1) = \frac{b}{b+1} \geqslant 0$ 
on $[\beta_b,\gamma_b]$, inequality~\eqref{e:fbl2} follows from 
$f_b(\lambda(\gamma_b)) =0$.

Further, $z_2 \geqslant w_1/2 \geqslant w_3$ since $w_1 = \ell_2 w_2 + w_3 \geqslant w_2+w_3 \geqslant 2w_3$. 

\s
(vii) 
Recall that $1 = \ell_1 w_1+w_2$.
If $\ell_1 = 2m+1$, then (vii) becomes 
$$
w_2-z_1+z_2+m(z_1+z_2) \,\geqslant\, 0,
$$
which holds true.
If $\ell_1 = 2m$, then (vii) becomes $w_2-z_1+z_2+m(z_1+z_2) \geqslant w_1$.
This holds true since it holds true for $m=1$ by assertion~(vi).
\proofend

The following lemma will be very useful.

\begin{lemma} \label{le:wzl}
If $w_2 \geqslant z_2$, then $\ell_2 =1$.
\end{lemma}

\proof
Recall that we can assume $\ell_3 \geqslant 1$, that is, $w_3 >0$.
If $\ell_2 \geqslant 2$, then $w_1 = \ell_2 w_2+w_3 > 2w_2 \geqslant 2z_2 \geqslant w_1$, 
by Lemma~\ref{le:inequalities}~(vi).
\proofend

\subsection{Reductions}

Applying one Cremona transform to
\[
\bigl( 
(b+1)\lambda;\; b\lambda,\:\lambda,\:1^{\times(2b+4)},\: w_{1}^{\times\ell_{1}}, \ldots
\bigr)
\]
with $\delta = -1$ yields
\[
\bigl( (b+1)\lambda-1;\; b\lambda-1,\:\underset{=\,z_{1}}{\underbrace{\lambda-1}},\:0,\:1^{\times(2b+3)},\: w_{1}^{\times\ell_{1}},\ldots
\bigr)
\]
which we reorder to
\[
\bigl( 
(b+1)\lambda-1;\; b\lambda-1,\:1^{\times(2b+3)} \parallel z_{1},\: w_{1}^{\times\ell_{1}},\ldots,\:0
\bigr).
\]
Applying $b$ Cremona transforms with $\delta = \lambda-2$ we obtain
\[
\Bigl(
\underset{=\,z_{2}+2}{\underbrace{(2b+1)\lambda-(2b+1)}};\:\underset{=\,1-z_{1}+z_{2}}{\underbrace{2b\lambda-(2b+1)}},\:1^{\times3},\: z_{1}^{\times(2b+1)},\: w_{1}^{\times\ell_{1}},\ldots,\:0
\Bigr)
\]
which by Lemma~\ref{le:inequalities} reorders to
\[
\Bigl( 
z_{2}+2;\:1^{\times3},\:1-z_{1}+z_{2} \parallel z_{1}^{\times(2b+1)},\: w_{1}^{\times\ell_{1}},\ldots,\:0
\Bigr).
\]
Applying one Cremona transform with $\delta = z_2 -1$ yields
\[
\bigl(
2z_{2}+1;\: z_{2}^{\times3},\:1-z_{1}+z_{2},\: z_{1}^{\times(2b+1)},\: w_{1}^{\times\ell_{1}},\ldots,\:0
\bigr)
\]
which we reorder to
\begin{equation} \label{eq:vector step 3}
\bigl(
2z_{2}+1;\;1-z_{1}+z_{2} \parallel z_{1}^{\times(2b+1)},\: z_{2}^{\times3},\: w_{1}^{\times\ell_{1}},\ldots,\:0
\bigr).
\end{equation}

We now distinguish several cases, according to the order of $z_1 \geqslant z_2$ and $w_1$.

\m \ni 
\textbf{Case 1.} $z_{1} \geqslant z_{2},w_{1}$.
Applying $b-1$ Cremona move to the vector~\eqref{eq:vector step 3} with $\delta = z_2-z_1$
we get the vector
\begin{equation} \label{eq:vector.case1}
\bigl( 
z_{1}+z_{2}+z_3;\: z_3,\: z_{1}^{\times3},\: z_{2}^{\times(2b+1)},\: w_{1}^{\times\ell_{1}}, \ldots
\bigr).
\end{equation}

\m \ni
\textbf{Case 1.a.} $z_{1}\geqslant z_{2}\geqslant w_{1}$.
%
%
If $z_3 \geqslant z_1$, we apply one more Cremona move with $\delta = z_2-z_1$ and obtain
$$
\bigl( 
2z_{2}+z_3;\: \underbrace{z_3+z_2-z_1}_{=\,z_4},\: z_{1},\: z_{2}^{\times(2b+3)},\: w_{1}^{\times\ell_{1}}, \ldots
\bigr).
$$
The assumption $z_3 \geqslant z_1$ is equivalent to $z_4 \geqslant z_2$.
Hence this vector is ordered up to possibly swapping $z_4$ and~$z_1$, 
and in either case $\delta =0$, whence this vector is reduced. 
We can thus assume for the rest of Case~1.a that 
\begin{equation} \label{e:z2z5}
z_1 \,\geqslant\, z_3 \quad \mbox{ and } \quad z_2 \,\geqslant\, z_4.
\end{equation} 

By Lemma~\ref{le:inequalities}~(iii) the
vector~\eqref{eq:vector.case1} reorders to
\begin{equation} \label{eq:vector.case1a}
\bigl( 
z_{1}+z_{2}+z_3;\: z_{1}^{\times3},\: z_3 \parallel 
z_{2}^{\times(2b+1)},\: w_{1}^{\times\ell_{1}}, \ldots
\bigr).
\end{equation}
One Cremona transform with $\delta = z_4-z_1$ yields the vector
$$
\bigl( 
2z_4+z_{1};\: z_4^{\times3},\: z_3,\: z_{2}^{\times(2b+1)},\: 
w_{1}^{\times\ell_{1}}, \ldots
\bigr)
$$
which by~\eqref{e:z2z5} reorders to
$$
\bigl( 
2z_4+z_{1};\: z_3,\: z_{2}^{\times(2b+1)} \parallel z_4^{\times3},\:
w_{1}^{\times\ell_{1}}, \ldots
\bigr) .
$$
Under $b$ Cremona transforms with $\delta = z_4-z_2$ this vector becomes
$$
\bigl( 
2z_4+z_{1} +b (z_4-z_2);\: z_3+b (z_4-z_2),\: z_{2} \parallel 
z_4^{\times (2b+3)},\: w_{1}^{\times\ell_{1}}, \ldots
\bigr) 
$$
where the ordering follows from Lemma~\ref{le:inequalities}~(iv).
Then $\delta = z_4 - (z_4 \mbox{ or } w_1)$.
If $z_4 \geqslant w_1$ we are done.
If $w_1 \geqslant z_4$, one more Cremona transform with $\delta = z_4 -w_1$
yields the vector
$$
\bigl( 
2z_4+z_{1} +b (z_4-z_2) + \delta;\: z_3+b (z_4-z_2) + \delta,\: z_2+z_4-w_1,\: 
w_1^{\times (\ell_{1}-1)} \parallel z_4^{\times (2b+4)},\: \ldots
\bigr) 
$$
which is ordered by Lemma~\ref{le:inequalities}~(v) and has defect~$0$.

\b \ni
\textbf{Case 1.b.} $z_{1} \geqslant w_{1} \geqslant z_{2}$. 
Assume first that $z_1 \geqslant z_3$.
The vector~\eqref{eq:vector.case1} then reorders to
\begin{equation} \label{}
\bigl( 
z_{1}+z_{2}+z_3;\: z_{1}^{\times3},\:z_3, \: w_{1}^{\times\ell_{1}} \parallel
z_{2}^{\times(2b+1)},\: w_2^{\times \ell_2}, \ldots
\bigr) .
\end{equation}
Since $z_4 \leqslant z_3 \leqslant z_1$, we also have $z_4 \leqslant z_1$, and so $\delta = z_4-z_1 \leqslant 0$.
One Cremona transform yields
\begin{equation*} 
\bigl( 
2z_4+z_1;\: z_4^{\times3},\:z_3, \: w_{1}^{\times \ell_{1}},\:
z_{2}^{\times(2b+1)},\: w_2^{\times \ell_2}, \ldots
\bigr) .
\end{equation*}
Since $z_1+z_4=z_2+z_3$ and $z_1 \geqslant z_3$, we have $z_4 \leqslant z_2$, whence this vector reorders to
\begin{equation*} 
\bigl( 
2z_4+z_1;\: z_3,\: w_{1}^{\times\ell_{1}} \parallel
z_{2}^{\times(2b+1)},\: z_4^{\times3},\: w_2^{\times \ell_2}, \ldots
\bigr) .
\end{equation*}
By Lemma~\ref{le:inequalities}~(v) we can estimate
$$
\delta \,=\, (z_4+z_2-w_1) - (w_1 \mbox{ or } z_2 \mbox{ or } z_4 \mbox{ or } w_2) 
\,\geqslant\, w_1-w_1 \,=\, 0 .
$$

For the rest of Case 1.b we can thus assume that
$$
z_3 \geqslant z_1 \quad \,\mbox{ and } \quad z_4 \geqslant z_2 .
$$
The vector~\eqref{eq:vector.case1} then reorders to
\begin{equation*} 
\bigl( 
z_{1}+z_{2}+z_3;\: z_3,\: z_{1}^{\times3},\: w_{1}^{\times\ell_{1}} \parallel
z_{2}^{\times(2b+1)},\: w_2^{\times \ell_2}, \ldots
\bigr) .
\end{equation*}
Applying one Cremona transform with $\delta =-z_1 + z_2$ yields 
$$
\bigl(
2z_{2}+z_3;\: z_4 \leftrightarrow z_1,\: w_{1}^{\times\ell_{1}} \parallel
z_{2}^{\times(2b+3)},\: w_2^{\times \ell_2},\ldots
\bigr) .
$$
The ordering is right up to possible swapping $z_4 \leftrightarrow z_1$ since $z_4 \geqslant w_1$
by Lemma~\ref{le:inequalities}~(v).
Abbreviate
$$
\ast := z_2+z_4-w_1 \quad \mbox{ and } \quad z_5 := z_1+z_2-w_1 .
$$
Then $z_5 \geqslant z_2$.
Applying one Cremona transform with $\delta = z_2-w_1$ we obtain
\begin{equation} \label{e:tief}
\bigl(
\ast + z_1+z_2;\: \ast,\: z_5,\: w_{1}^{\times (\ell_{1}-1)},\: 
z_{2}^{\times(2b+4)},\: w_2^{\times \ell_2}, 
\ldots \bigr)
\end{equation}
By Lemma~\ref{le:inequalities}~(v) we have $\ast \geqslant w_1$.
If also $z_5 \geqslant w_1$, then $\delta = w_1 - (w_1 \mbox{ or } z_2  \mbox{ or } w_2) \geqslant 0$.
So assume that $z_5 \leqslant w_1$.
Then the vector~\eqref{e:tief} reorders to
\begin{equation} \label{e:tiefer}
\bigl(
z_1+z_2+\ast;\: \ast,\:  w_{1}^{\times (\ell_{1}-1)} \parallel
z_5,\: z_{2}^{\times(2b+4)},\: w_2^{\times \ell_2}, 
\ldots \bigr).
\end{equation}

\m \ni
{\it Subcase 1:} $\ell_1 = 2m+1$ with $m \geqslant 0$.
Applying $m$ Cremona transforms with $\delta_* := z_5-w_1$ we get
\begin{equation} \label{e:last1b}
\bigl(
z_1+z_2+\ast+m\delta_*;\: \ast+m\delta_*,\:  
z_5^{\times \ell_1},\: z_{2}^{\times(2b+4)} \leftrightarrow 
w_2^{\times \ell_2}, 
\ldots \bigr).
\end{equation}
We claim that this vector is reduced after reordering.

Assume that $z_5 \geqslant w_2$.
Then the ordering in~\eqref{e:last1b} is right 
by Lemma~\ref{le:lastordering}~(i) below, and $\delta = w_1 - (z_5 \mbox{ or } z_2 \mbox{ or } w_2) \geqslant 0$.

Assume that $w_2 \geqslant z_5$.
Recall that $z_5 = z_1+z_2-w_1 \geqslant z_2 \geqslant w_3$. 
By Lemma~\ref{le:wzl} we have $\ell_2 =1$, 
and so by Lemma~\ref{le:lastordering}~(i) the vector~\eqref{e:last1b} 
reorders to
\begin{equation*}
\bigl(
z_1+z_2+\ast+m\delta_*;\: \ast+m\delta_* \leftrightarrow w_2,\:  
z_5^{\times \ell_1},\: z_{2}^{\times(2b+4)} \parallel 
\ldots \bigr).
\end{equation*}
Now 
$\delta = z_1+z_2 - w_2 - z_5 = w_1-w_2 \geqslant 0$.

\m \ni
{\it Subcase 2:} $\ell_1 = 2m$ with $m \geqslant 1$.
Applying $m-1$ Cremona transforms to~\eqref{e:tiefer} with $\delta_* = z_5-w_1$ we get
\begin{equation} \label{eq:subcase2}
\bigl(
z_1+z_2+\ast+(m-1)\delta_*;\: \ast+(m-1)\delta_*,\: w_1,\: 
z_5^{\times (\ell_1-1)},\: z_{2}^{\times(2b+4)},\: 
w_2^{\times \ell_2}, 
\ldots \bigr).
\end{equation}

Assume that $z_5 \geqslant w_2$. 
Then Lemma~\ref{le:lastordering}~(ii) shows that~\eqref{eq:subcase2} reorders to
\begin{equation*} 
\bigl(
z_1+z_2+\ast+(m-1)\delta_*;\: \ast+(m-1)\delta_* \leftrightarrow w_1,\: 
z_5^{\times (\ell_1-1)} \parallel 
z_{2}^{\times(2b+4)},\: w_2^{\times \ell_2}, 
\ldots \bigr) ,
\end{equation*}
and $\delta = 0$.

Assume that $w_2 \geqslant z_5$.
Then $\ell_2=1$ by Lemma~\ref{le:wzl}, and we reorder~\eqref{eq:subcase2} to
\begin{equation*}
\bigl(
z_1+z_2+\ast+(m-1)\delta_*;\: \ast+(m-1)\delta_*,\: w_1,\: w_2,\: 
z_5^{\times (\ell_1-1)},\: z_{2}^{\times(2b+4)}, 
\ldots \bigr).
\end{equation*}
One Cremona transform with $\hat \delta = z_5-w_2$ yields the vector
\begin{equation*}
\bigl(
z_1+z_2+\ast+(m-1)\delta_*+\hat \delta;\: 
\ast+(m-1)\delta_* +\hat \delta ,\: z_1+z_2-w_2,\: z_5^{\times \ell_1},\:
z_{2}^{\times(2b+4)}, 
\ldots \bigr).
\end{equation*}
Recall that $z_1+z_2 - w_2 \geqslant z_5 \geqslant z_2 \geqslant w_3$ (by Lemma~\ref{le:inequalities}~(vi))
and note that
$$
\ast + (m-1) \delta_* +\hat \delta \,\geqslant\, z_5+\hat \delta \,=\, 
2z_1+2z_2-2w_1-w_2 \,\geqslant\, 0
$$
by Lemma~\ref{le:lastordering}~(ii), by the assumption $z_1 \geqslant w_1$ and by Lemma~\ref{le:inequalities}~(vi).

If $\ast + (m-1) \delta_* +\hat \delta \geqslant z_5$, then $\delta = w_2-z_5 \geqslant 0$.

If $\ast + (m-1) \delta_* +\hat \delta \leqslant z_5$, then $\delta = \ast + (m-1) \delta_* -z_5 \geqslant 0$.

\begin{lemma} \label{le:lastordering}
Assume that $z_1 \geqslant w_1 \geqslant z_5$. 

\s
{\rm (i)} If $\ell_1 = 2m+1$, then $\ast +m\delta_* \geqslant z_5$.

\s
{\rm (ii)} If $\ell_1 = 2m$, then $\ast +(m-1)\delta_* \geqslant z_5$.
\end{lemma}

The proof is given in Section~\ref{ss:lemmata}.

\m
\ni
\textbf{Case 2.} $w_1 \geqslant z_{1} \geqslant z_{2}$.
Then $z_1 \geqslant z_2 \geqslant z_5$.
Recall from Lemma~\ref{le:inequalities}~(vi) that $z_2 \geqslant w_3$.
We shall therefore not display $w_3^{\times \ell_3}$ in the vectors below.
The vector~\eqref{eq:vector step 3} reorders to
\begin{equation} \label{eq:vector.case2start}
\bigl(
2z_{2}+1;\; 1-z_{1}+z_{2},\: w_1^{\times \ell_1} \parallel 
z_{1}^{\times(2b+1)},\: z_{2}^{\times3},\: w_2^{\times \ell_2},
\ldots 
\bigr).
\end{equation}

\ni 
\textbf{Case 2.a.} $\ell_{1}=2m+1$ is odd.
Applying $m$ Cremona transforms with $\delta_*= z_5-w_1 \leqslant 0$ 
we obtain the vector
\begin{equation*}
\bigl(
2z_{2}+1+m\delta_*;\; 1-z_{1}+z_{2}+m\delta_*,\: w_1,\: z_5^{\times (\ell_1-1)},\:  
z_{1}^{\times(2b+1)},\: z_{2}^{\times3},\: w_2^{\times \ell_2},
\ldots 
\bigr).
\end{equation*}
By assumption, $z_1 \geqslant z_2 \geqslant z_5$.
By Lemma~\ref{le:inequalities}~(vii) this vector reorders to
\begin{equation} \label{eq:vector.case2i.2}
\bigl(
2z_{2}+1+m\delta_*;\; 1-z_{1}+z_{2}+m\delta_*,\: w_1 \parallel 
z_{1}^{\times(2b+1)},\: z_{2}^{\times3},\: 
z_5^{\times (\ell_1-1)},\: w_2^{\times \ell_2},
\ldots 
\bigr).
\end{equation}

\s
\ni
{\it Subcase 1:} $z_1 \geqslant w_2$.
Applying one Cremona move with $\delta = z_2-w_1$ we obtain
\begin{equation*} 
\bigl(
3z_2+1-w_1+m\delta_*;\; 1-z_{1}+2z_2-w_1+m\delta_*,\:  
z_1^{\times 2b},\: z_{2}^{\times 4},\: 
z_5^{\times \ell_1},\: w_2^{\times \ell_2},
\ldots 
\bigr).
\end{equation*}
Applying $b$ Cremona transforms with $\delta = z_2-z_1$ 
and setting 
$$
\ast_1 \,:=\, 1 + m\delta_* + (b+1)(z_2-z_1) + z_2-w_1
$$
we obtain
\begin{equation} \label{e:ast3last} 
\bigl(
\ast_1 + z_1+z_2;\; \ast_1,\:
z_{2}^{\times (2b+4)},\: z_5^{\times \ell_1},\: w_2^{\times \ell_2},
\ldots 
\bigr).
\end{equation}
We claim that this vector is reduced after reordering.
To see this, assume first that $z_2 \geqslant w_2$.
If $\ast_1 \geqslant z_2$, then $\delta = z_1-z_2 \geqslant 0$, and
if $z_2 \geqslant \ast_1$, then $\delta = \ast_1+z_1-2z_2 \geqslant 0$ by Lemma~\ref{le:hard3}.
Assume now that $w_2 \geqslant z_2$. Then $\ell_2=1$ by Lemma~\ref{le:wzl}. 
If $\ast_1 \geqslant z_2$, then $\delta = z_1-w_2 \geqslant 0$, 
and if $z_2 \geqslant \ast_1$, then $\delta = \ast_1 +z_1-z_2-w_2 \geqslant 0$ by Lemma~\ref{le:hard3}.

\m
\ni 
{\it Subcase 2:} $w_2 \geqslant z_1$. 
Then $\ell_2=1$ by Lemma~\ref{le:wzl}, and
\begin{equation} \label{e:wwzz}
w_1 \,\geqslant\, w_2 \,\geqslant\, z_1 \,\geqslant\, z_2  \,\geqslant\, z_1+z_2-w_2 \,\geqslant\, z_1+z_2-w_1 = z_5 .
\end{equation}
The vector~\eqref{eq:vector.case2i.2} becomes
\begin{equation*} 
\bigl(
2z_{2}+1+m\delta_*;\; 1-z_{1}+z_{2}+m\delta_*,\: w_1,\: w_2,\: 
z_{1}^{\times(2b+1)},\: z_{2}^{\times3},\:
z_5^{\times (\ell_1-1)},
\ldots 
\bigr).
\end{equation*}
Applying one Cremona move with $\delta = z_1+z_2-w_1-w_2$ we obtain
\begin{equation*} 
\bigl(
\ast + z_1+z_2;\; \ast,\:  
z_1^{\times(2b+1)},\: z_2^{\times3},\:
z_1+z_2-w_2,\: z_5^{\times \ell_1}, 
\ldots 
\bigr),
\end{equation*} 
where 
$\ast := 1+2z_2+m\delta_*-w_1-w_2$.
Applying $b$~Cremona transforms with $\delta = z_2-z_1$ we obtain the vector
\begin{equation*} 
\bigl(
\ast_2 + z_1+z_2;\; \ast_2,\:
z_1,\: z_2^{\times (2b+3)},\:
z_1+z_2-w_2,\: z_5^{\times \ell_1},
\ldots 
\bigr),
\end{equation*}
where 
$$
\ast_2 \,:=\, 1 + m\delta_* + b (z_2-z_1) + 2z_2 - w_1 - w_2  \,=\, \ast_1+z_1-w_2.
$$
This vector is reduced after reordering. 
Indeed, if $\ast_2 \geqslant z_2$ then $\delta =0$,
and if $z_2 \geqslant \ast_2$ then $\delta = \ast_2 -z_2 = \ast_1 +z_1-z_2-w_2 \geqslant 0$ by 
Lemma~\ref{le:hard3}.

\begin{lemma} \label{le:hard3}
Assume that $w_1 \geqslant z_1 \geqslant z_2 \geqslant z_5$ and that $\ell_1 = 2m+1$.
Then 
$$
\ast_1 \,\geqslant\, 2z_2-z_1,\1 w_2+z_2-z_1 .  
$$
\end{lemma}   

The proof is given in Section~\ref{ss:lemmata}.

\b
\ni 
\textbf{Case 2.b.} $\ell_{1}=2m$ is even.
Applying to the vector~\eqref{eq:vector.case2start} $m$ Cremona transforms with $\delta_*= z_5-w_1 \leqslant 0$ 
we obtain the vector
\begin{equation*}
\bigl(
2z_{2}+1+m\delta_*;\; 1-z_{1}+z_{2}+m\delta_*,\: z_5^{\times \ell_1},\:  
z_{1}^{\times(2b+1)},\: z_{2}^{\times3},\: w_2^{\times \ell_2},
\ldots 
\bigr).
\end{equation*}
By Lemma~\ref{le:inequalities}~(vii) this vector reorders to
\begin{equation} \label{eq:vector.case2i}
\bigl(
2z_{2}+1+m\delta_*;\; 1-z_{1}+z_{2}+m\delta_* \parallel
z_{1}^{\times(2b+1)},\: z_{2}^{\times3},\: 
z_5^{\times \ell_1},\: w_2^{\times \ell_2},
\ldots 
\bigr).
\end{equation}

\s
\ni
{\it Subcase 1:} $z_1 \geqslant w_2$.
Applying $b$ Cremona transforms with $\delta = z_2-z_1$ 
and setting 
$$
\ast_3 \,:=\, 1 + m\delta_* +(b+1)(z_2-z_1)
$$ 
we obtain 
\begin{equation} \label{eq:vector.z1w2}
\bigl(
\ast_3 + z_1+z_2;\; \ast_3,\: z_1,\: z_{2}^{\times (2b+3)},\: 
z_5^{\times \ell_1},\: w_2^{\times \ell_2},
\ldots 
\bigr).
\end{equation}
If $z_2 \geqslant w_2$, then Lemma~\ref{le:hard2} shows that the ordering is
\begin{equation*} 
\bigl(
\ast_3 + z_1+z_2;\; \ast_3 \leftrightarrow z_1,\: z_{2}^{\times (2b+3)} \parallel 
z_5^{\times \ell_1},\: w_2^{\times \ell_2},
\ldots 
\bigr) ,
\end{equation*}
and this vector is reduced since $\delta =0$.
So assume that $z_1 \geqslant w_2 \geqslant z_2$.
Then $\ell_2=1$ by Lemma~\ref{le:wzl},
and we reorder the vector~\eqref{eq:vector.z1w2} to
\begin{equation*} 
\bigl(
\ast_3 + z_1+z_2;\; \ast_3,\: z_1,\: w_2,\: z_{2}^{\times (2b+3)},\: 
z_5^{\times \ell_1},
\ldots 
\bigr) .
\end{equation*}
Applying one Cremona transform with $\delta = z_2-w_2$ we obtain
\begin{equation*} 
\bigl(
\ast_3+z_2-w_2+z_1+z_2;\; \ast_3 +z_2-w_2 \leftrightarrow  
z_1 +z_2-w_2,\: z_2^{\times (2b+4)},\: 
z_5^{\times \ell_1},
\ldots \bigr).
\end{equation*}
Note that $z_1 + z_2 - w_2 \geqslant z_2$ by assumption.
If the ordering is right, then $\delta = w_2 - z_2 \geqslant 0$.
Otherwise, $z_2 > \ast_3 +z_2-w_2$, and then $\delta = \ast_3-z_2 \geqslant 0$ by Lemma~\ref{le:hard2}.

\m
\ni
{\it Subcase 2:} $w_2 \geqslant z_1$. 
By Lemma~\ref{le:wzl} we have $\ell_2=1$, and 
the vector~\eqref{eq:vector.case2i} becomes
\begin{equation*} 
\bigl(
2z_{2}+1+m\delta_*;\; 1-z_{1}+z_{2}+m\delta_*,\: w_2,\: 
z_{1}^{\times(2b+1)},\: z_{2}^{\times3} \parallel 
z_5^{\times \ell_1},
\ldots 
\bigr).
\end{equation*}
Applying one more Cremona move with $\delta = z_2-w_2$ we obtain
\begin{equation*} 
\bigl(
\ast_4 +z_1+z_2;\; \ast_4,\: z_1^{\times 2b},\: z_2^{\times 4},\:
z_1+z_2-w_2,\; z_5^{\times \ell_1},
\ldots \bigr)
\end{equation*}
where $\ast_4 := 1+m\delta_* - z_1 + 2z_2 - w_2$. 
Applying $b$ Cremona transforms with $\delta = z_2-z_1$ we obtain the vector
\begin{equation*} 
\bigl(
\ast_4+b(z_2-z_1) +z_1+z_2;\; \ast_4 +b(z_2-z_1),\: 
z_{2}^{\times (2b+4)},\:
z_1+z_2-w_2,\: z_5^{\times \ell_1},
\ldots 
\bigr) .
\end{equation*}
We claim that this vector is reduced after reordering.
Indeed, if the ordering is right, then $\delta = z_1-z_2 \geqslant 0$.
Otherwise, $z_2 > \ast_4 + b(z_2-z_1)$, and then 
$$
\delta \,=\,  \ast_4 + b(z_2-z_1) +z_1-2z_2 \,=\, \ast_3 +z_1-z_2-w_2 \,\geqslant\, 0
$$
in view of Lemma~\ref{le:hard2}.

\begin{lemma} \label{le:hard2}
Assume that $w_1 \geqslant z_1 \geqslant z_2 \geqslant z_5$ and that $\ell_1 = 2m$.
Then 
$$
\ast_3 \,\geqslant\, z_2, w_2 + z_2-z_1 .
$$
\end{lemma}

\subsection{Proof of Lemmata \ref{le:lastordering}, \ref{le:hard3} and~\ref{le:hard2}}
\label{ss:lemmata}

In this section we prove Lemmata~\ref{le:lastordering}, \ref{le:hard3} and~\ref{le:hard2},
that we restate for the readers convenience.
Recall that $\delta_* = z_1+z_2-2w_1$ and $\ast = z_2+z_4-w_1 = 1+(b+1)(z_2-z_1) +z_2-w_1$. 
Hence
\begin{eqnarray*}
\ast + m\delta_* \;=\; \ast_1 &=& 1 + m\delta_* + (b+1)(z_2-z_1) + z_2-w_1 , \\
              \ast_3 &=& 1 + m\delta_* +(b+1)(z_2-z_1) .
\end{eqnarray*}

\begin{lemma} \label{le:lastordering'}
Assume that $z_1 \geqslant w_1 \geqslant z_5$. 

\s
\begin{itemize}
\item[(i)]
If $\ell_1 = 2m+1$, then $\ast +m\delta_* \geqslant z_5$.

\s
\item[(ii)]
If $\ell_1 = 2m$, then $\ast +(m-1)\delta_* \geqslant z_5$.
\end{itemize}
\end{lemma}

\begin{lemma} \label{le:hard2'}
Assume that $w_1 \geqslant z_1 \geqslant z_2 \geqslant z_5$.

\s
\begin{itemize}
\item[(i)]
If $\ell_1 = 2m+1$, then $\ast_1 \geqslant 2z_2-z_1, w_2 + z_2-z_1$.

\s
\item[(ii)]
If $\ell_1 = 2m$, then $\ast_3 \geqslant z_2, w_2+z_2-z_1$.
\end{itemize}
\end{lemma}

Note that $\delta_* \leqslant 0$ in both lemmata.
The proofs are along the following lines. 
All inequalities are, roughly, of the form 
\begin{equation} \label{e:rough1}
1+m\delta_* + b(z_2-z_1) \geqslant 0
\end{equation}
or, using $1 = (2m (+1) ) \,w_1+w_2$,
\begin{equation} \label{e:rough2} 
m (z_1+z_2) + b(z_2-z_1) \geqslant 0 .
\end{equation}

\begin{figure}[ht]
 \begin{center}
  \psfrag{m}{$m$}
  \psfrag{b2}{$\frac b2$}
  \psfrag{b3}{$\frac b3$}
  \psfrag{L1}{$\mbox{Lemma~\ref{le:lastordering'}}$}
  \psfrag{L2}{$\mbox{Lemma~\ref{le:hard2'}}$}
 \leavevmode\epsfbox{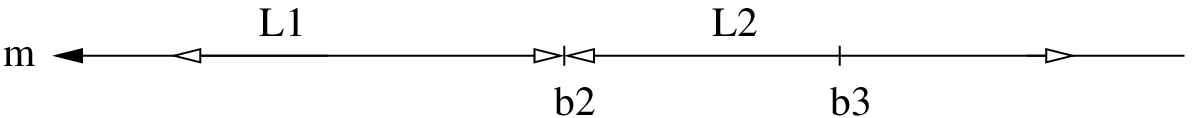}
 \end{center}
 \caption{}
 \label{fig.m}
\end{figure}
%
%

In Lemma~\ref{le:lastordering'}, the assumption $z_1 \geqslant w_1$ translates, roughly, to $m \succcurlyeq \frac b2$.
Further, $w_1 \geqslant z_5$ translates to $3z_2 \geqslant z_1$, which together with~\eqref{e:rough2} implies 
Lemma~\ref{le:lastordering'} for $m \succcurlyeq \frac b2 +1$.
For the remaining one or two $m \approx \frac{b+1}2$ we prove the lemma using~\eqref{e:rough1} and $\delta_* \leqslant 0$.

Lemma~\ref{le:hard2'} is proven similarly: 
The case $m \succcurlyeq \frac b3$ is settled using $2z_2 \geqslant z_1$ and~\eqref{e:rough2},
and the case $m \preccurlyeq \frac b3-1$ is settled using~\eqref{e:rough1} and $\delta_* \leqslant 0$.

\m
\ni
{\it Proof of Lemma~\ref{le:lastordering'}:}
The inequality $z_1 \geqslant w_1$ implies that
\begin{equation} \label{e:ell1b}
\ell_1 \,\geqslant\, b.
\end{equation}
Indeed, $z_1 \geqslant w_1$ is equivalent to
$\sqrt{\frac{a}{2b}} \geqslant a -(2b+3)$ or,
$$
a \,\leqslant\, 2b+3 + \frac{1+\sqrt{16b^2+24b+1}}{4b} ,
$$
which in turn translates to 
$$
\frac{1}{w_1} \,\geqslant\, \frac{4b}{1+\sqrt{16b^2+24b+1}-4b}.
$$
Since the right hand side is larger than~$b$, inequality~\eqref{e:ell1b} follows.

We next observe that $w_1 \geqslant z_5$ implies that
\begin{equation} \label{e:3z2z1}
3 z_2 \,\geqslant\, z_1 .
\end{equation}
Indeed, 
$(3 z_2-z_1) - (w_1-z_5) = 2 (2z_2-w_1) \geqslant 0$ by Lemma~\ref{le:inequalities}~(vi).
This is the main ingredient for proving

\m
\ni
{\bf Claim 1.}
{\it (i) holds for $m \geqslant \frac b2 +1$.

(ii) holds for $m \geqslant \frac b2 + \frac 32$.
}

\proof
(i) follows from $\ast + m \delta_*  \geqslant z_1$, 
and since $1= (2m+1) w_1+w_2$, this inequality follows from
$$
(b+2) (z_2-z_1) +m(z_1+z_2) \,\geqslant\, 0.
$$
Using \eqref{e:3z2z1} we estimate
\begin{eqnarray*}
(b+2) (z_2-z_1) +m(z_1+z_2) &=& (-b+m-2)z_1+(b+m+2)z_2 \\
&\geqslant&
(-b+2m-2) \tfrac 23 z_1 
\end{eqnarray*}
which is non-negative if $m \geqslant \frac b2 +1$.

\s
(ii) follows from $\ast + (m-1) \delta_*  \geqslant w_1$, 
and since $1= 2m \1 w_1+w_2$, this inequality follows from
$$
(b+1) (z_2-z_1) +(m-1)(z_1+z_2) +z_2 \,\geqslant\, 0.
$$
Using \eqref{e:3z2z1} we estimate
\begin{eqnarray*}
(b+1) (z_2-z_1) +(m-1)(z_1+z_2) +z_2 &=& (-b+m-2)z_1+(b+m+1)z_2 \\
&\geqslant&
(-2b+4m-5) \tfrac 13 z_1 
\end{eqnarray*}
which is non-negative if $m \geqslant \frac b2 + \frac 54$.
\proofend

\ni
{\it Proof of (i).}
In view of~\eqref{e:ell1b} and~Claim~1 (i) we can assume that
$m \in [\frac{b-1}{2}, \frac{b+1}{2}]$.
We wish to show that for these $m$ (of which are one or two)
we have $\ast + m \delta_* \geqslant z_5$. Since $\delta_* \leqslant 0$, this follows if 
$\ast + \frac{b+1}2 \, \delta_* \geqslant z_5$, that is, 
$$
f_b(\la) \,:=\, -2b (b+1) \la^2 + b(3b+4) \la - (b^2+b-2) \,\geqslant\, 0
$$
for $a \in [2b+4+ \frac{1}{2m+2}, 2b+4+\frac{1}{2m+1}]$ and $m \in [\frac{b-1}{2}, \frac{b+1}{2}]$.
Since $f_b'(\la) \leqslant -b^2 < 0$ and since $m \geqslant \frac{b-1}2$, it suffices to show that
$f_b(\la) \geqslant 0$ at $\la = \sqrt{\frac{2b+4+\frac 1b}{2b}}$, that is, 
$$
1+\frac 2b + \frac{1}{2b^2} \,\geqslant\, 
\left( \frac{3b^2+7b+3+\frac 1b}{3b^2+4b} \right)^2 .
$$
Subtracting $1$ and multiplying by $2b^2 (3b^2+4b)^2$ this becomes
$3b^4-8b^3-30b^2-12b-2 \geqslant 0$, which holds true for $b \geqslant 5$.

\s
To deal with the cases $b \in \{3,4\}$ we return to $\ast + m \delta_* \geqslant z_5$, i.e.,
\begin{equation} \label{e:astz5}
1+(b+1)(z_2-z_1) + m (z_1+z_2-2w_1) - z_1 \,\geqslant\, 0 .
\end{equation}

Assume that $b=4$. Then $m=2$, and \eqref{e:astz5} becomes
$$
7z_2 +1 \,\geqslant\, 4z_1 + 4w_1 \quad \mbox{ on }\; I := [12+\tfrac 16, 12+\tfrac 15],
$$
i.e., $f(a) := -a+\frac{59}{8} \sqrt{\frac a2}-6 \geqslant 0$ on~$I$.
This holds true since $f'(a) < 0$ on~$I$ and $f(12+\frac 15) >0$.
Finally, if $b=3$, then $m \in \{1,2\}$.
For $m=2$, \eqref{e:astz5} becomes $-2a+\frac{13}{2} \sqrt{\frac{3a}{2}} -5 \geqslant 0$ on $[10+\frac 16, 10+\frac 15]$, 
which holds true;
and for $m=1$, \eqref{e:astz5} becomes $-a+\frac{31}{2} \sqrt{\frac a6} -10 \geqslant 0$ on $[10+\frac 14, 10+\frac 13]$, 
which holds true too.

\b
\ni
{\it Proof of (ii).}
In this case, \eqref{e:ell1b} and~Claim~1 (ii) show that we can assume that
$m \in [\frac{b}{2}, \frac{b}{2} +1]$.
We wish to show that for these~$m$ 
we have $\ast + (m-1) \delta_* \geqslant z_5$. Since $\delta_* \leqslant 0$, this follows if 
$\ast + \frac b2 \, \delta_* \geqslant z_5$, that is, 
$$      
f_b(\la) \,:=\, -2 b^2 \la^2 + (3 b^2+3b -1) \la - b(b+2) \,\geqslant\, 0
$$
for $a \in [2b+4+ \frac{1}{2m+1}, 2b+4+\frac{1}{2m}]$ and $m \in [\frac{b}{2}, \frac{b}{2}+1]$.
Since $f_b'(\la) \leqslant -b^2 +3b-1 < 0$ and since $m \geqslant \frac{b}2$, it suffices to show that
$f_b(\la) \geqslant 0$ at $\la = \sqrt{\frac{2b+4+\frac 1b}{2b}}$, that is, 
$$
1+\frac 2b + \frac{1}{2b^2} \,\geqslant\, 
\left( \frac{3b^2+6b+1}{3b^2+3b-1} \right)^2 .
$$
Subtracting $1$ and multiplying by $2b^2 (3b^2+3b-1)^2$ this becomes
$3b^4-6b^3-21b^2-2b+1 \geqslant 0$, which holds true for $b \geqslant 4$.

Assume that $b=3$. Then $m=2$, and $\ast + (m-1) \delta_* \geqslant z_5$ becomes
$-a+\frac{31}{2} \sqrt{\frac a6} -10 \geqslant 0$ on $[10+\frac 15, 10+\frac 14]$, 
which holds true.
\proofend

\ni
{\it Proof of Lemma~\ref{le:hard2'}:}
(i) 
is equivalent to
\begin{equation} \label{e:ihard22}
1+ m \delta_* + b(z_2-z_1) +z_2-w_1 \,\geqslant\, z_2, w_2.
\end{equation}
Since $1=(2m+1)w_1 + w_2$, this is equivalent to $m (z_1+z_2) + b(z_2-z_1) +z_2+w_2 \geqslant z_2,w_2$, 
which follows if
\begin{equation} \label{e:iihard22}
m (z_1+z_2) + b(z_2-z_1)  \,\geqslant\, 0 .
\end{equation}

\s
\ni 
{\bf Claim 1.} {\it \eqref{e:iihard22} holds for $m \geqslant \frac b3$.}

\s
Indeed, since $2z_2 \geqslant w_1 \geqslant z_1$ by Lemma~\ref{le:inequalities} and by assumption,   
$$
m (z_1+z_2) + b(z_2-z_1)  \,=\, (m-b)z_1+(m+b)z_2 
\,\geqslant\, (3m-b) \tfrac{z_1}{2} .
$$

\ni
{\bf Claim 2.} {\it \eqref{e:ihard22} holds for $m \leqslant \frac b3-1$.}

\proof
Since $\delta_* \leqslant 0$ and $w_1 \geqslant z_2,w_2$, it suffices to show that
\begin{equation} \label{e:dazu}
1+(\tfrac b3-1) \delta_* + b(z_2-z_1) +z_2-w_1 \,\geqslant\, w_1 ,
\end{equation}
or, equivalently, that
\begin{equation} \label{e:fbl3}
f_b(\gl) \,:=\, - 4 b^2 \gl^2 + (8 b^2+2b-3) \gl - 2(2b^2+b-3) \,\geqslant\, 0.
\end{equation}
Note that $f_b'(\la) = -8b^2 \la + (8 b^2+2b-3) < 0$ for $\la \geqslant \la(\beta_b)$ since 
$(b+1) f_b'(\la (\beta_b)) = -(6b^2+5b+3) < 0$.
Hence \eqref{e:fbl3} follows from $b\,f_b(\lambda (\gamma_b)) = b-3 \geqslant 0$.
\proofend

\ni
{\bf Claim 3.} {\it \eqref{e:ihard22} holds for $m \leqslant \frac{b-1}3$ if $b \geqslant 7$.}

\proof
It suffices to show that
$$
1 + \tfrac{b-1}3 \,\delta_* + b(z_2-z_1) +z_2-w_1 \,\geqslant\, w_1 ,
$$
or, equivalently, that
\begin{equation} \label{e:fbl4}
g_b(\gl) \,:=\, -(4b^2+8b) \gl^2 + (8 b^2+6b+1) \gl - 4b^2+2b+14 \,\geqslant\, 0.
\end{equation}
Since $g_b'(\gl) < 0$ for $\la \geqslant 1$, \eqref{e:fbl4} follows from $b\,g_b(\la(\gamma_b)) = b-7$.
\proofend

In view of the three claims above we are left with showing~(i) for $b \in \{ 4,5\}$ and $m=1$.

Assume that $b=5$. It suffices to show that $1+ \delta_* + 5 (z_2-z_1) +z_2 \geqslant 2 w_1$ for $a \in [\beta_b,\gamma_b]$, 
that is, 
$$
f(\la) \,:=\, -40 \la^2 + 73 \la -30 \,\geqslant\, 0 \quad \mbox{ for }\, a \in \left[ \beta_b,\gamma_b \right] .
$$
This holds true since $f'(\la) < 0$ for $\la \geqslant 1$ and $f(\la (\gamma_b)) =0$.

\s
Assume that $b=4$. 
Then $\ast_1 = 1+\delta_* + 5 (z_2-z_1) +z_2-w_1$.
The inequality $\ast_1 \geqslant 2z_2-z_1$ becomes $1 + 5z_2 \geqslant 3 w_1+3z_1$, 
or 
$$
f(\la) \,:=\, -8 \la^2 + 14 \la -5 \,\geqslant\, 0, 
$$
which holds true since $f'(\la) < 0$ for $\la \geqslant 1$
and $f (\la (\gamma_b)) = 0$.
The inequality $\ast_1 \geqslant w_2 +z_2-z_1 = 1-3w_1+z_2-z_1$ becomes $6z_2 \geqslant 3z_1$,
which holds true.
 
\b
(ii) of Lemma~\ref{le:hard2'} is equivalent to
\begin{equation} \label{e:ii1}
1+m\delta_* + (b+1) (z_2-z_1) \,\geqslant\, z_2, w_2+z_2-z_1.
\end{equation}
Since $1=2m\1 w_1+w_2$, this is equivalent to $m(z_1+z_2) + b(z_2-z_1) \geqslant z_1-w_2, 0$,
which follows if
\begin{equation} \label{e:ii2}
m(z_1+z_2) + b (z_2-z_1) \,\geqslant\, z_1, 0.
\end{equation}

\s
\ni 
{\bf Claim 1.} {\it \eqref{e:ii2} holds for $m \geqslant \frac{b+2}3, \frac b3$.}

\s
\ni
{\bf Claim 2.} {\it \eqref{e:ii1} holds for $m \leqslant \frac b3, \frac{b-2}{3}$.}

\proof
For $m \leqslant \frac b3$, the inequality $\geqslant z_2$ in~\eqref{e:ii1} follows from 
$1+ \frac b3 \,\delta_* + (b+1)(z_2-z_1) \geqslant z_2$, 
which is equivalent to~\eqref{e:dazu}.
For $m \leqslant \frac{b-2}{3}$, the inequality $\geqslant w_2 +z_2-z_1$ in~\eqref{e:ii1} follows from 
$1+ \frac{b-2}{3} \,\delta_* + b(z_2-z_1) \geqslant w_1$ or, 
\begin{equation} \label{e:iif1}
f_b(\la) \,:=\, (-4b^2+2b)\la^2 + (8b^2-2b-4) \la -4b^2+7 \,\geqslant\, 0.
\end{equation}
Note that $f_b'(\la) <0$ for $\la \geqslant \la (\beta_b)$ since $(b+1)\,f_b'(\la (\beta_b)) = -2(3b^2+b+1)<0$.
Hence~\eqref{e:iif1} follows from $b\, f_b(\la(\gamma_b)) = b-2$.
\proofend

\ni
{\bf Claim 3.} {\it \eqref{e:ii1} holds for $m = \frac{b+1}{3}$ if $b \geqslant 5$, for $m = \frac{b-1}{3}$ if $b \geqslant 4$.}

\proof
The first assertion is that 
$1+ \frac{b+1}{3} \,\delta_* + (b+1) (z_2-z_1) \geqslant z_2$ for $b \geqslant 5$, or, 
\begin{equation} \label{e:iig1}
g_b(\la) \,:=\, (-4b^2-4b)\la^2 + (8b^2+4b-1) \la -4b^2+10 \,\geqslant\, 0.
\end{equation}
Since $g_b'(\la) < 0$ for $\la \geqslant 1$, \eqref{e:iig1} follows from $b\,g_b(\la(\gamma_b)) = b-5$.

The second assertion follows if
$1+ \frac{b-1}{3} \,\delta_* + b (z_2-z_1) \geqslant w_1$ for $b \geqslant 4$, that is,  
\begin{equation} \label{e:iih1}
h_b(\la) \,:=\, (-4b^2-2b)\la^2 + (8b^2-2) \la -4b^2 + 2b+11 \,\geqslant\, 0.
\end{equation}
Since $h_b'(\la) < 0$ for $\la \geqslant 1$, \eqref{e:iih1} follows from $b\,h_b(\la(\gamma_b)) = b-4$.
\proofend

The three claims above imply (ii).


\begin{remark}
{\rm 
One can use the reduction method also for showing
that $c_2(a) = \frac{\sqrt{a}}{2}$ on $[\beta_2,u_2(2)] = [8 \frac{1}{36},9]$, of course.
Contrary to all other assertions in Lemma~\ref{le:inequalities}, 
assertion~(v) does not hold for $b=2$ if $a \geqslant 8.0831$, however.
The reduction scheme for $b=2$ on $[\beta_b,u_b(2)]$ is therefore quite different 
from the one for $b \geqslant 3$, in particular in Case 1.b.
}
\end{remark}


\end{document}